\documentclass[a4paper,fleqn]{cas-sc}
\usepackage{blindtext}
\usepackage[utf8]{inputenc}

\usepackage[numbers]{natbib}

\usepackage{graphicx}
\usepackage{mathtools}
\usepackage{amssymb}
\usepackage{amsfonts}
\usepackage{amsthm}
\usepackage{amsmath}
\usepackage{float}
\usepackage{caption}
\usepackage{subcaption}
\usepackage{comment}
\usepackage{enumitem}
\usepackage{mathrsfs}

\usepackage{tikz-cd}

\usepackage{setspace}

\usepackage{hyperref}
\hypersetup{
	colorlinks=true,
	linkcolor=blue,
	filecolor=magenta,      
	urlcolor=cyan,
	pdfpagemode=FullScreen,
}

\usepackage{cases}
\usepackage{orcidlink}

\def\tsc#1{\csdef{#1}{\textsc{\lowercase{#1}}\xspace}}
\tsc{WGM}
\tsc{QE}

\newcounter{assumption}

\newenvironment{assumption}[1]{%
	\def\theassumption{#1}%
	\refstepcounter{assumption}%
	\begin{quote}%
		\textbf{(\theassumption) }\ignorespaces
	}{%
	\end{quote}%
}

\newcommand{\assref}[1]{\hyperref[#1]{(\ref*{#1})}}

\newenvironment{proofof}[1]
{\par\medskip\noindent\textit{Proof of #1.}\quad\ignorespaces}
{\hfill$\square$\par\medskip}

\DeclareMathOperator{\Sym}{Sym}

\DeclareMathOperator{\Gr}{Gr}

\DeclareMathOperator{\rk}{rk}

\DeclareMathOperator{\Cov}{Cov}

\DeclareMathOperator{\Span}{Span}
\DeclareMathOperator{\id}{id}
\DeclareMathOperator{\Imm}{Im}

\DeclareMathOperator{\hor}{hor}
\DeclareMathOperator{\St}{St}

\DeclareMathOperator{\supp}{supp}
\DeclareMathOperator{\vol}{vol}

\DeclareMathOperator{\Tr}{Tr}

\newcommand{\Chi}{\cl X}

\newcommand{\cl}[1]{\mathcal{#1}}
\newcommand{\bb}[1]{\mathbb{#1}}
\newcommand{\one}{\mathbf{1}}

\newtheorem{proposition}{Proposition}
\newtheorem{theorem}{Theorem}
\newtheorem{corollary}{Corollary}
\newtheorem{lemma}{Lemma}
\newdefinition{definition}{Definition}
\newdefinition{remark}{Remark}

\def\R{\mathbb{R}}

\def\E{\mathbb{E}}
\def\P{\mathbb{P}}

\allowdisplaybreaks

\begin{document}
	\let\WriteBookmarks\relax
	\def\floatpagepagefraction{1}
	\def\textpagefraction{.001}
	
	\shorttitle{Regime-Switching Diffusions on Stratified Spaces}    
	
	\shortauthors{L. Marconi, M. Farnè, A. Aue}  
	
	\title [mode = title]{Geometric Regime--Switching Diffusions on Stratified Riemannian Spaces with an Application to Covariance Matrices}  
	
	
	
	%
	
	\author[1]{Leonardo Marconi}[orcid=0009-0000-1597-7875]
	
	\cormark[1]
	
	
	\ead{leonardo.marconi5@unibo.it}
	
	
	\credit{Conceptualization, Methodology, Formal analysis,
		Investigation, Writing -- original draft, Writing -- review \& editing}
	
	\affiliation[1]{organization={Alma Mater Studiorum University of Bologna},
		addressline={Via delle Belle Arti 41}, 
		city={Bologna},
		postcode={40126}, 
		country={Italy}}
	
	\author[1]{Matteo Farnè}[orcid=0000-0002-2403-6599]
	
	
	\ead{matteo.farne@unibo.it}
	
	
	\credit{Project Administration, Funding Acquisition, Validation, Supervision, Writing -- review \& editing}
	
		
	\author[2]{Alexander Aue}[orcid=0000-0003-1553-0509]
	
	
	\ead{aaue@ucdavis.edu}
	
	
	\credit{Resources, Validation, Supervision, Writing – review \& editing}
	
	\affiliation[2]{organization={University of California Davis},
		addressline={1 Shields Ave}, 
		city={Davis},
		postcode={95616}, 
		state={California},
		country={US}}
	
	\cortext[1]{Corresponding author}
	
	
	
	\begin{abstract}
		We construct geometric regime-switching diffusions, a class of Markov processes on locally compact stratified Riemannian state spaces. In contrast with classical regime-switching and stochastic hybrid diffusions, the regimes are not external labels,
		but intrinsic strata of a singular geometric state space. Changes of regime may therefore change dimension, rank, geometry or combinatorial type while the state space maintains its ambient topology.
		
		On each stratum the motion is a conservative Feller diffusion, while inter-stratum transitions are specified by state-dependent jump rates and landing kernels along a directed graph. We characterize the process through a martingale problem on a natural
		stratified core. Under a uniform bound on the total jump rate, we construct a conservative càdlàg strong Markov process by combining the stratumwise diffusions with a Poisson thinning mechanism. Uniqueness is proved using an auxiliary disjoint-union topology and a
		bounded perturbation argument.
		
		Standard Foster--Lyapunov conditions for the extended generator give positive Harris recurrence, uniqueness of the invariant probability measure and, under aperiodicity, \(V\)-uniform geometric ergodicity. The framework is applied to the cone of positive semidefinite covariance matrices, stratified by rank. The resulting process combines fixed-rank covariance diffusions with stochastic rank changes and is \(V\)-uniformly geometrically ergodic.	
	\end{abstract}

	
	\begin{highlights}
		\item We construct geometric regime-switching diffusions on stratified spaces.
		\item Regimes are intrinsic strata rather than external switching labels.
		\item Inter-stratum jumps may change dimension, rank, geometry or type.
		\item The martingale problem is well posed for bounded jump rates.
		\item A geometrically ergodic rank-switching covariance model is constructed.
	\end{highlights}

	
	\maketitle
	
	
	
	
	

	%
	%
	
	
	
	
	\section{Introduction}

	\subsection{Background, motivation and main contributions}
	
	Diffusion processes on smooth Riemannian manifolds are a classical object of study in probability theory. They arise as solutions of stochastic differential equations driven by Brownian motion, admit descriptions in terms of generators, semigroups and Dirichlet forms, and play a central role in stochastic analysis on manifolds, geometric analysis and geometric statistics. Classical references on manifold-valued diffusions include \cite{brzezniak2000stochastic,hsu2002stochastic,ikeda2014stochastic}.
	
	In many applications, however, the natural state space is not a single smooth manifold. It may instead be a singular space obtained by gluing together smooth pieces of possibly different dimensions. Examples include manifolds with corners, spaces with conical or edge singularities, polyhedral complexes, orbit spaces of group actions, spaces of trees and graphs, the cone of positive semidefinite covariance matrices and moduli spaces arising in geometry and statistics. Such spaces do not generally carry a single global smooth
	structure. They can often be modelled instead by a stratification, namely a
	decomposition into smooth pieces whose intrinsic geometric structure may vary from
	stratum to stratum, for instance in dimension, boundary behavior, rank, or
	combinatorial type. 
	
	The study of analysis and probability on singular spaces has developed along several directions. One influential approach constructs diffusion processes from local regular Dirichlet forms and associated heat kernels on metric spaces \cite{sturm1998diffusion}. A different foundational line starts with Walsh's Brownian motion on a spider, where several one-dimensional branches are glued at a single vertex and the behavior at the singular point is prescribed by a rule for choosing outgoing rays \cite{barlow2006walsh,walsh1978diffusion}. This model became a prototype for diffusions on branched spaces and metric graphs, where the singular set is treated through boundary or interface conditions  \cite{freidlin2000diffusion,HajriRaimond2016}. Related works also study Brownian motions on spaces with varying dimension, where smooth pieces are glued through lower-dimensional interfaces \cite{chen2019brownian}.
	
	In this paper, we adopt a different and complementary point of view. Rather than starting from a single object, such as Dirichlet form, heat kernel or diffusion operator, on the whole singular space, or imposing transition or boundary conditions, we start from Markovian dynamics on the individual smooth strata and specify, as part of the model, how the process is allowed to move between strata. Thus the singular structure is not only a region through which a diffusion may pass, but also a structure along which stochastic changes of dimension, rank, geometry or combinatorial type can occur.

	A second, parallel, line of work concerns stochastic processes whose continuous evolution is combined with discrete switching. A foundational example is Davis' theory of piecewise-deterministic Markov processes, in which trajectories follow deterministic flows between random jump times \cite{davis1984piecewise}. Regime-switching diffusions and switching diffusion processes replace the deterministic inter-jump motion by diffusion dynamics whose coefficients are modulated by a finite or countable external mode process, typically on a common Euclidean phase space \cite{KhasminskiiZhuYin2007,mao2006stochastic,yin2009hybrid}; state-dependent switching and jump-diffusion variants were studied, among others, in \cite{Xi2009,XiZhu2018}. Stochastic hybrid systems further enlarge this picture by allowing mode-dependent continuous domains, stochastic transitions and reset kernels; see \cite{hu2000towards} and references therein. A construction particularly close in spirit to the present paper is the Markov-string viewpoint, going back to Meyer's mixing operation for Markov processes \cite{meyer1975renaissance} and developed in the stochastic hybrid systems literature by \cite{bujorianu2006toward}, where a global process is obtained by concatenating component Markov processes through stopping rules and renewal kernels.  These frameworks provide important predecessors for the pathwise construction used below.
	
	Our framework combines these viewpoints at the level of Markov processes and martingale problems. We consider a locally compact separable metric space \(E\) which is decomposed, measurably, into countably many smooth Riemannian strata,
	\begin{equation}
		\label{eq:state-space}
		E=\bigsqcup_{\alpha\in I}S_\alpha .
	\end{equation}
	On each stratum \(S_\alpha\) the process evolves according to a prescribed conservative Feller diffusion. At random times, determined by state-dependent rates, it jumps to another stratum, and the post-jump location is sampled from a state-dependent probability kernel supported on the target stratum. A directed graph on the set of strata encodes the admissible transitions.

	Classical regime-switching diffusions usually change the dynamics by means of an external mode while keeping a common continuous phase space. General stochastic hybrid systems go further, since different modes may carry different continuous domains and transitions may involve reset kernels. Here, instead, we propose a different framework. The state space is a single locally compact metric space \(E\), decomposed into intrinsic Riemannian strata. The ambient topology and metric are therefore part of the model: they determine the Borel structure, the path space and the notion of convergence between states, and they may encode how one stratum arises as a limit of another. Thus a change of regime is not only a change of domain or coefficients, but a movement inside a stratified geometric object, and can represent a genuine alteration or degeneration of the intrinsic geometry of the state variable.	
	This shift is particularly relevant when there is no canonical smooth structure across strata, and hence no natural global stochastic differential equation on \(E\). 
	The primitive objects are instead intrinsically local and geometrically heterogeneous and the problem is then to turn these data into a single well-posed Markov process on \(E\). This probabilistic formulation makes it possible to establish existence, uniqueness, strong Markovness and ergodic properties without requiring a global differential structure on the ambient singular space. Key probability references include \cite{ethier2009markov,meyn1993stability3,meyn2012markov}.
	
	Our main theoretical contribution is, therefore, to show that conservative Feller diffusions living on different smooth manifolds can be assembled, through state-dependent inter-stratum jumps, into a conservative càdlàg strong Markov process on a single stratified state space, with a well-posed martingale problem and standard ergodic criteria.	
	We call the resulting processes geometric regime-switching diffusions. They retain the hybrid character of classical switching models, but the regimes are now intrinsic geometric pieces of a singular state space rather than external labels attached to a fixed continuous phase space. Accordingly, our objective is not to construct a diffusion generated by a global differential operator on the singular space, nor to develop a general theory of stochastic differential equations across strata. Instead, we deliberately formulate the theory at the level of Markov families and martingale problems, where the local geometric dynamics and the inter-stratum jump mechanism can be combined in a unified probabilistic framework.
	
	A guiding example, developed in Section~\ref{sec:cov-example}, is the cone of positive semidefinite covariance matrices. This cone is naturally stratified by rank. Each fixed-rank component is a smooth manifold, while the full cone is singular because strata of different ranks meet at the boundary. In this example the rank is not an auxiliary discrete variable superimposed on the model: it is a geometric property of the covariance matrix itself. When the rank changes, the process moves between manifolds of different dimensions and different tangent structures; the local coordinates and the corresponding generator change accordingly. The covariance-cone model therefore illustrates a situation in which the ambient topology and metric structure of the whole space are part of the modelling problem, and in which continuous fixed-rank evolution and stochastic rank changes should be treated as two components of the same Markovian dynamics.

	\subsection{Organization of the paper}
	\label{subsec:outline}
	
	The paper is organized as follows. Section~\ref{sec:stratified-jump-diffusions} introduces the abstract framework,
	while Section~\ref{sec:cov-example} is devoted to the rank-switching
	covariance-cone model and verifies the assumptions of the general theory in
	that setting.
	
	Section \ref{subsec:markov-conventions} introduces the Markov process notation and conventions used throughout the paper.	
	Section~\ref{subsec:set up} introduces the abstract stratified setting and the standing assumptions. The state
	space \(E\) is decomposed as in \eqref{eq:state-space}. On every
	stratum we are given a conservative Feller diffusion with a generator and a core
	\(\cl C_\alpha\). Inter-stratum transitions are encoded by a directed graph: each edge has a tail stratum, a target stratum, a state-dependent jump rate and a landing
	kernel supported on the target stratum. Under a uniform bound on the total jump rate, these data
	define a stratified martingale operator
	on a stratified test space, that is the direct sum of the cores \(\cl C_\alpha\).
	
	Section~\ref{subsec:existence} proves existence. For every initial condition, we construct a conservative
	càdlàg process by interlacing the prescribed stratumwise diffusions with jumps generated by a
	Poisson thinning procedure. Between two successive jump times the process evolves inside a single
	stratum according to the corresponding diffusion. At a jump time, an outgoing edge is selected
	according to the structural jump rates, and the post-jump location is sampled from the associated
	landing kernel. The uniform bound on the total jump rate prevents accumulation of jump times and
	gives non-explosion. We carry out the construction through an auxiliary input space, built by
	the Ionescu--Tulcea theorem, whose coordinates contain candidate diffusion paths and independent
	marked Poisson random measures.
	
	Section~\ref{subsec:strong-markov} establishes the strong Markov property. The proof uses the standard restart argument for processes obtained by concatenating Markovian pieces. One first proves the restart property at the canonical jump times,
	using the product structure of the input space. Then, one proves the same property at stopping times
	occurring before the next jump, using the strong Markov property of the stratumwise diffusions and
	the independent increments of the Poisson random measure.
	
	Section \ref{subsec:uniqueness} proves uniqueness. The key idea is to use an auxiliary
	disjoint-union topology on $E$.
	This topology is only a functional-analytic device: the process itself remains an \(E\)-valued
	process while the Borel \(\sigma\)-field is unchanged.
	With this auxiliary topology, the direct sum of the stratumwise Feller semigroups defines a
	Feller semigroup, and the
	stratified test space is a core for its generator. The jump part is then treated
	as a bounded perturbation, with norm controlled by the uniform bound on the
	total jump rate. A resolvent argument yields uniqueness of the martingale problem.

	Section~\ref{subsec:harris} records the ergodic consequences needed later. We do not develop a new Harris
	recurrence theory, but apply standard Foster--Lyapunov criteria for continuous-time Markov
	processes to the class constructed in the previous sections. A Lyapunov drift condition for the
	extended generator implies positive
	Harris recurrence and its standard ergodic consequences, including
	uniqueness of the invariant probability measure and almost-sure
	convergence of time averages. If, in addition, the process is aperiodic, then the same
	framework gives \(V\)-uniform geometric ergodicity.
	
	Section \ref{sec:cov-example} applies the general framework to the cone of positive semidefinite covariance matrices, stratified
	by rank. On each fixed-rank stratum the intra-stratum dynamics is built using the diffeomorphic description
	\[S_r:=\{\Sigma \in \Cov(n)\,|\, \rk(\Sigma)=r\}\cong (\St(n,r)\times\Sym(r))/O(r),\]
	where $\St(n,r)$ denotes the Stiefel manifold, $\Sym(r)$ the Euclidean space of symmetric $r$-dimensional matrices and $O(r)$ the group of \(r\times r\) orthogonal matrices. 
	Intra-stratum diffusions are defined using a horizontal diffusion on $\St(n,r)$ together with an Ornstein--Uhlenbeck dynamics on $\Sym(r)$. Rank changes \(r\mapsto r\pm1\) are then introduced through
	the jump mechanism. We verify
	the assumptions of the general theory, prove a Lyapunov drift estimate, establish petite-set and
	aperiodicity conditions, and conclude that the resulting rank-switching covariance process is a
	conservative strong Markov process with a well-posed martingale problem and \(V\)-uniform geometric
	ergodicity for \(V(\Sigma)=1+\Tr(\Sigma)\).

	Finally, the appendices collect the technical material used in the main text. Appendix~\ref{appsec:interlacing} contains the
	measurability details for the interlacing construction, Appendix~\ref{appsec:strong-markov} gives the proof of the strong
	Markov property, Appendix~\ref{appsec:uniqueness} proves the auxiliary direct-sum semigroup result, and Appendix~\ref{appsec:cov-application}
	contains the covariance-stratum verifications.

	\section{Geometric regime-switching diffusions}
	\label{sec:stratified-jump-diffusions}

	\subsection{Markov process conventions}
	\label{subsec:markov-conventions}
	We collect here the Markov process terminology used throughout the paper. Let
	\(E\) be a locally compact separable metric space, equipped with its Borel
	\(\sigma\)-algebra \(\cl B(E)\). Denote by $B_b(E)$ the space of bounded Borel functions on $E$.

	Let \(X=(X_t)_{t\geq 0}\) be an \(E\)-valued process evolving on a probability space \((\Omega,\bb P_x)\), where $x \in E$ is the starting condition of $X$. 
	We write
	\(
	\cl F_t^0:=\sigma(X_s:0\leq s\leq t)
	\)
	for the raw canonical filtration, and \((\cl F_t^x)_{t\geq 0}\) for the right-continuous completed natural filtration under $\P_x$.
	
	Since \(E\) is locally compact and separable, we may fix an increasing sequence
	\((O_m)_{m\geq 1}\) of open relatively compact subsets of \(E\) such that
	\(
	\overline O_m \subset O_{m+1},
	\text{ and }
	\bigcup_{m\geq 1} O_m = E .
	\)
	Here relatively compact means that \(\overline O_m\) is compact in \(E\). We shall call
	such a sequence a relatively compact exhaustion of \(E\).
	For each \(m\geq 1\), define the exit time from \(O_m\) by
	\(
	T_m := \inf\{t\geq 0:X_t\notin O_m\},
	\)
	with the convention \(\inf\varnothing=\infty\). Since \(O_m\subset O_{m+1}\), the
	sequence \((T_m)_{m\geq 1}\) is increasing. We define the lifetime, or explosion time,
	of \(X\) by
	\(
	\zeta := \lim_{m\to\infty} T_m .
	\)
	The process is called non-explosive, or conservative, if
	\(
	\P_x(\zeta=\infty)=1, \text{ for all } x\in E .
	\)
		
	We denote by
	\(
	D_E:= D([0,\infty);E),
	\)
	the Skorokhod space of all càdlàg paths \(\omega:[0,\infty)\to E\), where càdlàg means
	right-continuous with left limits. A non-explosive process can be described as a family of probability measures $(\P_x)_{x \in E}$ on \(D([0,\infty);E)\) together with the canonical process defined by
	\(
	X_t(\omega):=\omega(t),
	\) for  \(t\geq 0\).
	
	We shall use the following notation for laws. If \(Y\) is a random element with values in a measurable space \((F,\cl F)\), defined on a probability space with probability measure \(\P\), we write
	\(
	\cl L_{\P}(Y):=\P\circ Y^{-1}\in \cl P(F)
	\)
	for the law of \(Y\). When the underlying probability measure is indexed by an initial condition, for instance \(\P_x\), we write
	\(
	\cl L_x(Y):=\cl L_{\P_x}(Y),
	\)
	provided that the underlying probability measure is clear from the context.

	A Markov family on \(E\) is a family of probability measures
	\(
	(\P_x)_{x\in E}
	\)
	on path space such that
	\(
	\P_x(X_0=x)=1,
	\)
	and, for all \(B \in \cl B(E)\),
	\(\P_x(X_t \in B \,|\, \cl F_s)=\P_x(X_t \in B \,|\, X_s).\)
	Its transition function is denoted by
	\(
	P_t(x,B):=\P_x(X_t\in B).
	\)
	We define the operator $P_t$ acting on $f \in B_b(E)$ and $\sigma$-finite measures $\mu$ on $E$ by
	\[
	P_t f(x):=\mathbb E_x[f(X_t)]
	=
	\int_E f(y)\,P_t(x,dy), \quad \mu P_t(B):=\int_E P_t(x,B)\,\mu(dx), \ B \in \cl B(E).
	\]
	The family \((P_t)_{t\geq 0}\) is called the transition semigroup of the process.

	The Markov family is called strong Markov if, for every \(x\in E\), every finite
	\((\cl F_t^x)\)-stopping time \(T\), every \(t\geq 0\), and every \(f\in B_b (E)\),
	\(
	\mathbb E_x[f(X_{T+t})\mid \cl F_T^X]
	=
	P_t f(X_T)
	\text{ } \P_x\text{-a.s.}
	\)
	Equivalently, conditionally on \(\cl F_T\), the post-\(T\) evolution has the same
	law as a fresh copy of the process started from the random state \(X_T\).
	
	The Markov family is called Borel if the map
	\(
	E \to \cl P(D_E), \  x\mapsto \P_x
	\)
	is Borel measurable. This implies the following, weaker and more usual condition: for every $t \ge 0$ and \(B \in \cl B(E)\) the map $x \to P_t(x,B)$ is measurable. We call a Borel right process a
	càdlàg strong Markov family with Borel transition function.

	Finally, let \(A:\cl D\to B_b(E)\) be a linear operator, where
	\(\cl D\subset B_b(E)\). A probability measure \(\P\) on
	\(D([0,\infty);E)\) is said to solve the martingale problem for \((A,\cl D)\)
	with initial condition \(x\in E\) if
	\(
	\P(X_0=x)=1
	\)
	and, for every \(f\in\cl D\), the process
	\[
	M_t^f
	:=
	f(X_t)-f(X_0)-\int_0^t Af(X_s)\,ds,
	\qquad t\geq 0,
	\]
	is an \((\cl F_t^X)\)-martingale under \(\P\). The martingale problem is
	called well-posed if, for every \(x\in E\), there exists a unique solution in law.
	
	\subsection{Set-up and assumptions}
	\label{subsec:set up}

	Let $(E,d)$ be a locally compact separable metric space with Borel
	$\sigma$--algebra $\cl B(E)$.
	
	\begin{definition}
		\label{def:stratified-space}
		Let $I$ be a countable index set and suppose
		\(
		E = \bigsqcup_{\alpha\in I} S_\alpha
		\)
		is a partition into pairwise disjoint subsets such that:
		\begin{enumerate}[label=(S\arabic*),ref=S\arabic*]
			\item for every $\alpha\in I$, $S_\alpha$ is a connected, smooth, finite--dimensional
			manifold without boundary, equipped with a Riemannian metric $g_\alpha$;
			\item\label{S2} the inclusion $\iota_\alpha: S_\alpha \hookrightarrow E$ is a homeomorphism onto its image
			and this image is a Borel subset of \(E\).
		\end{enumerate}
		We call $(E,\{S_\alpha,g_\alpha\}_{\alpha\in I})$ a stratified Riemannian state space in the Borel sense.
	\end{definition}

	We now introduce the continuous dynamics inside each stratum. The formalization is deliberately local: once the current stratum is fixed, the process evolves as a conservative Feller diffusion on that smooth manifold, with a core on which the generator can be tested. No compatibility of the differential operators across different strata, and no global smooth structure on \(E\), is required.
	
	\begin{assumption}{D}
		\label{ass:diffusions}
		For every \(\alpha\in I\), there is a conservative càdlàg Markov family
		\(
		(\bb Q_x^\alpha)_{x\in S_\alpha}
		\)
		on \(D([0,\infty);S_\alpha)\), with canonical process
		\(
		\eta_t(\omega)=\omega(t),
		\)
		such that the following properties hold.
		
		\begin{enumerate}[label=(D\arabic*),ref=(D\arabic*)]
			\item \label{it:D1}
			The map
			\(
			S_\alpha \to \cl P(D([0,\infty);S_\alpha)), \  x\mapsto \bb Q_x^\alpha
			\)
			is Borel measurable, where \(\cl P(D([0,\infty);S_\alpha))\) denotes the space of probability measures on the Skorokhod space.

			\item \label{it:D2}
			For every \(x\in S_\alpha\) the law \(\mathbb Q_x^\alpha\) is concentrated on
			paths starting from \(x\), i.e., 
			\(
			\bb Q_x^\alpha(\eta_0=x)=1,
			\)
			and \(\eta\) is \(S_\alpha\)-valued and defined for all times, that is
			\(
			\bb Q_x^\alpha(\eta_t\in S_\alpha\ \text{for all }t\geq0)=1.
			\)
			The paths are continuous in the topology of \(S_\alpha\), equivalently in the topology
			induced on \(S_\alpha\) by the ambient space \(E\).
			
			\item \label{it:D3}
			Let
			\(
			P_t^\alpha g(x)
			:=
			\mathbb E_{\bb Q_x^\alpha}[g(\eta_t)],
			\,
			g\in B_b(S_\alpha).
			\)
			Then \((P_t^\alpha)_{t\geq0}\) restricts to a strongly continuous contraction
			semigroup on \(C_0(S_\alpha)\). That is, $P_t^\alpha C_0(S_\alpha)\subset C_0(S_\alpha),$ and 
			\(
			\|P_t^\alpha g \|_{\infty} \le \|g\|_{\infty}, \text{ and }
			\lim_{t\to0}
			\|P_t^\alpha g-g\|_\infty=0,
			\) for all $g \in C_0(S_\alpha)$.
			
			\item\label{it:D4}
			Let \(L_\alpha\) denote the generator of the Feller semigroup
			\((P_t^\alpha)_{t\geq0}\) on \(C_0(S_\alpha)\). We assume that there exists a linear subspace $\cl C_\alpha \subset C_0(S_\alpha)$ such that
			\(
			\cl C_\alpha\subset D( L_\alpha),
			\)
			and \(\cl C_\alpha\) is a core for \(L_\alpha\): for every
			\(g\in D(L_\alpha)\), there exist \((g_n)_{n \ge 0}\subset \cl C_\alpha\) such that
			\(
			\|g_n-g\|_\infty
			+
			\|L_\alpha g_n-L_\alpha g\|_\infty
			\to0
			\) as \(n \to +\infty\).
			Moreover, we assume that $\cl C_\alpha$ is measure determining on $S_\alpha$, that is if \(\mu\)
			and \(\nu\) are finite Borel measures on \(S_\alpha\) and
			\[
			\int_{S_\alpha} g\,d\mu
			=
			\int_{S_\alpha} g\,d\nu,
			\qquad g\in\cl C_\alpha,
			\]
			then \(\mu=\nu\).

		\end{enumerate}
		
	\end{assumption}
	\begin{remark}
		Since
		$(\bb Q_x^\alpha)_{x\in S_\alpha}$ is a Markov family with transition semigroup
		$(P_t^\alpha)_{t\geq 0}$, and since \ref{it:D4} gives
		$\cl C_\alpha\subset D(L_\alpha)$,
		Dynkin's formula for strongly continuous semigroups yields that
		\[
		g(\eta_t)-g(\eta_0)-\int_0^t  L_\alpha g(\eta_s)\,ds,
		\qquad t\geq 0,
		\]
		is a $\bb Q_x^\alpha$-martingale for every $x\in S_\alpha$ and every
		$g\in \cl C_\alpha$. 
	\end{remark}

	We next encode how the process is allowed to jump between strata via a directed graph of state--dependent jump kernels. The directed graph specifies which transitions are allowed, the rates determine when such transitions occur, and the landing kernels determine the post-jump location. The uniform bound on the total rate is used twice: it prevents accumulation of jump times in the interlacing construction, and it makes the jump operator a bounded perturbation in the uniqueness argument.
	
	\begin{assumption}{J}
		\label{ass:jumps}
		There is a directed graph $G=(I,\cl E)$ with vertex set $I$ and countable edge set $\cl E$.
		For each edge $e\in \cl E$ we write $\alpha(e)\in I$ for its tail and $\beta(e)\in I$
		for its head, and we are given:
		\begin{itemize}
			\item a Borel measurable probability kernel
			\(
			K_e:S_{\alpha(e)}\times \cl B(E)\to [0,1]
			\)
			such that
			\begin{equation}
				\label{eq:no-jump-in-place}
				K_e(x,S_{\beta(e)})=1,\quad x\in S_{\alpha(e)}, \qquad\text{and}\qquad K_e(x,\{x\})=0;
			\end{equation}
			\item a Borel measurable jump rate
			\(
			\lambda_e:S_{\alpha(e)}\to [0,\infty).
			\)
		\end{itemize}
		For $x\in S_\alpha$ we define the total jump rate out of $x$ and the associated jump kernel respectively by
		\[
		\Lambda(x):=\sum_{e\in\cl E:\,\alpha(e)=\alpha}\lambda_e(x) \in [0,+\infty], \qquad \nu(x,B)
		:=\sum_{e\in\cl E:\,\alpha(e)=\alpha}\lambda_e(x)\,K_e(x,B), \quad
		B\in\cl B(E).
		\]
		We assume that the total rate is uniformly bounded, that is
		\(
			\sup_{x\in E}\Lambda(x)\le \overline\Lambda<\infty.
		\)
	\end{assumption}

	Assumptions~\assref{ass:diffusions} and~\assref{ass:jumps} will be in force throughout
	Section~\ref{sec:stratified-jump-diffusions}. The boundedness of $\Lambda$ ensures that
	the jump part can be realized by thinning from a homogeneous Poisson process of rate
	$\overline\Lambda$ \cite{lewis1979simulation}, and will be crucial for non-explosion and for the bounded
	perturbation argument used in the uniqueness result.

	The following definition packages the local generators and the inter-stratum jump mechanism into a single martingale operator on \(E\). The test space \(\cl D_0\) should be understood as the finite-stratum test space: its elements are global bounded Borel functions obtained by choosing core functions on finitely many strata and setting the remaining stratum components equal to zero.

	\begin{definition}
		\label{def:pre-generator}
		Set
		\(\cl D_0:=\bigoplus_{\alpha\in I} \cl C_\alpha,\)
		i.e.\ \(\cl D_0\) consists of all bounded Borel functions \(f:E\to\mathbb R\) such that \(f|_{S_\alpha}\in\cl C_\alpha\) for every \(\alpha\in I\), and \(f|_{S_\alpha}=0\) for all but finitely many \(\alpha\in I\).
		
		For $f\in\cl D_0$ and $x\in S_\alpha$ we define \(Jf(x):=\int_E(f(y)-f(x))\nu(x,dy)\) and
		\[
			A f(x)
			:=L_\alpha(f|_{S_\alpha})(x)
			+Jf(x).
		\]
		This yields a linear operator
		$A:\cl D_0\to B_b(E)$, where $B_b(E)$ denotes the space of bounded Borel functions on $E$.
	\end{definition}

	We now show that under the standing assumptions the martingale problem
	for $(A,\cl D_0)$ is well-posed.

	\subsection{Interlacing construction}
	\label{subsec:existence}
	The next result shows that the local ingredients specified in \assref{ass:diffusions} and \assref{ass:jumps} can be
	assembled into a genuine global Markov process on the original stratified state space.

	\begin{theorem}
		\label{thm:existence}
		Suppose Assumptions~\assref{ass:diffusions} and~\assref{ass:jumps} hold.
		Then for every starting point $x\in E$ there exists a filtered
		probability space $(\Omega,\cl F_t^x,\P_x)$ and an
		$E$--valued càdlàg process $(X_t)_{t\ge 0}$ such that:
		\begin{enumerate}[label=(\roman*),ref=\roman*]
			\item $X_0=x$ $\P_x$--a.s.;
			\item between its jump times, the process remains in a single stratum:
			if $X_{t-}\in S_\alpha$ and $X_t=X_{t-}$, then $X_s\in S_\alpha$
			for all $s$ in a neighborhood of $t$; on each such interval,
			$X$ evolves as the diffusion with generator $L_\alpha$;
			\item the jumps of $X$ are governed by the structural kernels and rates:
			whenever $X_{t-}=x\in S_\alpha$, jumps of type $e\in\cl E$ with $\alpha(e)=\alpha$
			occur with conditional intensity $\lambda_e(x)$, and conditional on such a jump of type $e$,
			the post-jump location has law $K_e(x,\cdot)$;
			\item for every $f\in\cl D_0$, the process
			\[
			M_t^f=f(X_t)-f(X_0)-\int_0^t A f(X_s)\,ds
			\]
			is an $(\cl F_t^x)$--martingale. In particular, the law of $X$
			on $D([0,\infty);E)$ is a solution of the martingale problem for
			$(A,\cl D_0)$ with initial condition $x$.
		\end{enumerate}
		Moreover, $X$ is non-explosive and the map $ x \mapsto \P_x$ is Borel.
	\end{theorem}

	\begin{proof}
		The proof is organized so as to separate the probabilistic construction from the auxiliary measurability and pathwise details. In the main text we describe the interlacing mechanism, show non-accumulation of jump times from the uniform rate bound, and identify the martingale operator. The measurability statements and technical checks needed to make the construction rigorous are collected in Appendix~\ref{appsec:interlacing}.
		
		We first realize the jump kernels by measurable maps. Since \(E\) is a
		locally compact separable metric space, it is a standard Borel space. Hence,
		for every edge \(e\in\cl E\), there exists a Borel measurable map
		\(
		F_e:S_{\alpha(e)}\times(0,1)\to S_{\beta(e)}
		\)
		such that, if \(U\sim {\rm Unif}(0,1)\), then, for \(x\in S_{\alpha(e)}\), we have
		\(
		F_e(x,U)\sim K_e(x,\cdot).
		\)
		
		For each \(\alpha\in I\), fix a countable enumeration of the outgoing edges
		\(
		\{e\in\cl E:\alpha(e)=\alpha\}
		=
		\{e_{\alpha,1},e_{\alpha,2},\ldots\}.
		\)
		For $x \in S_\alpha$ and $0\le s\le \Lambda(x)$ define
		\(
		\Gamma_{\alpha,0}(x):=0, \ 
		\Gamma_{\alpha,n}(x):=
		\sum_{j=1}^n \lambda_{e_{\alpha,j}}(x),
		\)
		and the measurable edge
		selector
		\(
		e(x,s):=e_{\alpha,n}
		\text{ where }
		n:=\min\{j\ge 1:s\le \Gamma_{\alpha,j}(x)\}.
		\)
		For \(s>\Lambda(x)\), the value of \(e(x,s)\) is irrelevant and we fix it to $0$.
		
		We now construct the probability space. Let
		\(
		Z:=(0,\infty)\times(0,\overline\Lambda]\times(0,1),
		\)
		and denote by $N_Z$ the space of locally finite counting measures on $Z$. Endow $N_Z$ with the $\sigma$-algebra $\cl N_Z$ generated by the evaluation maps
		\(
		m\mapsto m(B)
		\) with $B \in \cl B(Z)$ relatively compact.
		Since \(Z\) is a locally compact separable metric space, this is a
		standard Borel space \cite[Ch. 2 and 6]{LastPenrose2017}.
		We consider the space
		\((H, \cl H) :=(D_E \times N_Z, \cl B(D_E) \otimes \cl N_Z).\)
		For \(x\in E\), let \(\alpha(x)\) be the unique stratum index such that
		\(x\in S_{\alpha(x)}\), and let \(\bb Q_x:= \bb Q_x^{\alpha(x)}\) denote the law on
		\( (D_E, \cl B(D_E))\) of the conservative intra-stratum diffusion on
		\(S_{\alpha(x)}\) started from \(x\). We view this path as an \(E\)-valued
		path through the inclusion \(S_{\alpha(x)}\subset E\). 
		
		Let \(\bb M\) be the law of a Poisson random measure on $(N_Z, \cl N_Z)$	with intensity \(dt\,ds\,du\). Define the one-step input kernel $\bb K: E \times \cl H \to [0,1]$ by
		\(
		\bb K (x;dh)=\bb K(x;d\eta,dm)
		:=
		\bb Q_x(d\eta)\,\bb M(dm),
		\)
		where \(\eta\in D_E\) is a diffusion path started from \(x\),
		\(m\) is an independent Poisson random measure, and $h=(\eta,m)$.
		
		Given \(x\in E\), \(\eta\in D_E\) and a point measure \(m\), define
		the first accepted time by
		\[
		\sigma(h)
		:=
		\inf\Bigl\{
		r>0:
		m\bigl(\{r\}\times(0,\Lambda(\eta(r-))]\times(0,1)\bigr)>0
		\Bigr\},
		\]
		with the convention \(\inf\varnothing=\infty\). If \(\sigma(h)<\infty\) and the first accepted atom is unique, write its marks as
		\((s,u)\).  Define the corresponding
		update map by
		\(
		\Psi(h)
		:=
		F_{e(\eta(\sigma(h)-),s)}(\eta(\sigma(h)-),u).
		\)
		If $\sigma=\infty$ or the accepted mark is not unique, define $\Psi$ arbitrarily, e.g. $\Psi(h)=x_\ast \in E$. The maps $\sigma:H \to (0 ,+\infty]$ and $\Psi:H \to E$ are measurable, as shown in Lemma~\ref{lem:s-j-meas} in Appendix~\ref{appsec:interlacing}.

		Now, select $x \in E$.	
		The Ionescu--Tulcea theorem (\cite[Thm. IV.4.7]{ccinlar2011probability}) gives a unique
		probability measure $\Pi_{x}$ on $(J,\cl J):=(H^{\bb N_0}, \cl H^{\otimes \bb N_0})$ such that
		conditionally on the previously constructed variables and on the
		starting state \(x\),
		\(
		h_n\sim \bb K(\Psi(h_{n-1});\cdot).
		\)
		Indeed, we have
		\begin{align}
			\label{eq:pi-x}
			\Pi_{x}(dh_0,dh_1,\dots)
			=\bb K(x;dh_0)\prod_{n \ge 1} \bb K(\Psi(h_{n-1});dh_n).
		\end{align}

		We now define the physical path \(X\) by concatenation. For an element $j=(h_0,h_1,\dots) \in J$, set
		\(
		\tau_0(j):=0, \text{ and } x_0(x,j):=x.
		\)
		Recursively, set
		\(
		\sigma_{n+1}(j):=\sigma(h_n), \  \tau_{n+1}(j):=\tau_n(j)+\sigma_{n+1}(j) \text{ and } x_{n+1}(x,j)
		:=
		\Psi(h_n).\)		
		We define the good set \(B_{\mathrm{good}}\subset E\times J\) as the set of all
		\((x,j)\) such that \(\eta_n(0)=x_n(x,j)\) for every \(n\ge0\) with
			\(\tau_n(j)<\infty\),  \(\sigma(h_n)>0\) for every \(n\ge0\) with \(\tau_n(j)<\infty\) and \(\tau_n(j)\to\infty\).
		On \(B_{\mathrm{good}}\), the path is given by the map $\Chi: E \times J \to D_E$ defined by
		\[
			\Chi^x_t(j)=\Chi(x,j)(t)
			:=
			\begin{cases}
				x_n(x,j), & t=\tau_n(j)\text{ for some }n,\\[1mm]
				\eta_n(t-\tau_n(j)), &
				\tau_n(j)<t<\tau_{n+1}(j).
			\end{cases}
		\]
		On \(E\times J\setminus B_{\mathrm{good}}\), define
		\(\Chi^x(j)\) to be an arbitrary fixed path in \(D_E\), e.g. $\Chi^x(j)\equiv x$. The map \(\Chi: E \times J \to D_E\) is Borel, see Lemma~\ref{lem:Chi-borel} in Appendix~\ref{appsec:interlacing}.

		Moreover, all the exceptional bad events in the construction, namely the presence of a multiple first accepted atom of $m_n$ and $B_{\mathrm{good}}^\complement$ are \(\Pi_x\)-null, see Lemma~\ref{lem:bad-sets} in Appendix~\ref{appsec:interlacing}.
		In particular, since each
		intra-stratum diffusion is conservative, the process is defined for all
		\(t\ge0\). Moreover, on every finite interval there are only finitely many
		jumps, and between jumps the path is continuous. Thus \(\Chi^x\) has c\`adl\`ag paths
		and is non-explosive.
		
		We equip $(D_E, \cl B(D_E))$ with the push-forward probability measure $\bb P_x:=(\Chi^x)_\# \Pi_x$, that is
		\begin{equation}
			\label{eq:Px-def-by-push-forw}
		\bb P_x(B):=\Pi_x\left(\{j \,|\, \Chi^x(j)\in B\}\right)=\int_J \one_B(\Chi^x(j))\Pi_x(dj), \quad B \in \cl B(D_E).
		\end{equation}
		If we define the canonical path on $D_E$ by \(X_t(\omega)=\omega(t),\)	then we have
		\(X_t \circ \Chi^x= \Chi_t^x,\)
		which gives, for all $B \in \cl B(E)$,
		\begin{equation}
			\label{eq:same-law}
			\bb P_x(X_t \in B)=\Pi_x(X_t \circ \Chi^x \in B)=\Pi_x(\Chi_t^x \in B).
		\end{equation}
		In other words, $X_t$ under $\bb P_x$ has the same law as $\Chi_t^x$ under $\Pi_x$.
		
		By construction, between \(\tau_n\) and \(\tau_{n+1}\) the process follows
		the diffusion path \(\eta_n\), whose law is that of the conservative
		diffusion on the stratum containing \(x_n=\Chi_{\tau_n}^x\). Hence, between jump
		times, \(\Chi^x\) evolves inside a single stratum according to the corresponding
		generator \(L_{\alpha}\). By \eqref{eq:same-law}, this and \(\tau_n \to +\infty\) give (ii)
		
		For the proof of (iii) and (iv), we write here the main points and leave details for Proposition~\ref{prop:iii-iv} in Appendix~\ref{appsec:interlacing}.
		
		Fix an edge \(e\in\cl E\) and a
		Borel set \(B\in\cl B(E)\). By the compensator formula for the thinned
		Poisson random measure, the counting process of jumps of type \(e\) landing
		in \(B\) has compensator
		\(
		\int_0^t \lambda_e(X_{s-})K_e(X_{s-},B)\,ds .
		\)
		In particular, jumps of type \(e\) occur with stochastic intensity
		\(\lambda_e(X_{t-})\), and conditional on such a jump from a pre-jump state
		\(z\), the post-jump position has law \(K_e(z,\cdot)\). Summing over the
		outgoing edges gives total jump intensity
		\(
		\Lambda(X_{t-}).
		\)
		This gives (iii)
		
		Let \(f\in\cl D_0\). On
		every inter-jump interval, \(X\) evolves inside a single stratum \(S_\alpha\).
		Hence, up to the next jump,
		\[
		f(X_t)-f(X_0)
		-
		\int_0^t
		L_\alpha(f|_{S_\alpha})(X_s)\,ds
		\]
		is a martingale. The compensated jump contribution is
		\[
		\sum_{0<s\leq t}
		\bigl(f(X_s)-f(X_{s-})\bigr)
		-
		\int_0^t
		Jf(X_s)\,ds .
		\]
		Adding the intra-stratum martingale part and the compensated jump part, and
		then summing over the finitely many inter-jump intervals before a fixed time
		\(T\), gives
		\[
		M_t^f
		:=
		f(X_t)-f(X_0)-\int_0^t Af(X_s)\,ds,
		\qquad 0\leq t\leq T,
		\]
		as a martingale. Since \(T>0\) is arbitrary, this holds for all \(t\geq 0\).
		Thus \(\P_x\) solves the martingale problem for \((A,\cl D_0)\).

		It remains to show that $x \mapsto \P_x$ is Borel. By \eqref{eq:Px-def-by-push-forw} and Lemma~\ref{lem:Chi-borel}, it suffices to prove that $x \mapsto \Pi_x$ is measurable. 
		To that purpose, we observe that the one-step input kernel
		\(
		\mathbb K(x;dh)
		=
		\mathbb Q_x(d\eta)\mathbb M(dm)
		\)
		is measurable in \(x\). This follows from Assumption~\ref{it:D1} and from the fact that
		the strata form a countable Borel partition of \(E\). Since
		\(
		\Psi:H\to E
		\)
		is Borel, the recursively defined kernels
		\(
		\mathbb K(x;dh_0)
		\)
		and, for \(n\ge1\),
		\(
		\mathbb K(\Psi(h_{n-1});dh_n)
		\)
		are measurable. This proves the claim.		
	\end{proof}

	\subsection{Strong Markov property}
	\label{subsec:strong-markov}
	
	We now show that the process constructed in Theorem~\ref{thm:existence} is strong Markov.
	
	\begin{theorem}
		\label{thm:strong-markov}
		The process $X$ constructed in Theorem~\ref{thm:existence} is strong Markov with respect to $\cl F^x=(\cl F_t^x)_{t\ge 0}$, that is, for every finite \(\cl F^x\)-stopping time \(T\) and bounded Borel functional $F$
		\[
		\mathbb E_x
		\left[
		F\big((X_{T+r})_{r\ge0}\big)
		\,\middle|\,
		\cl F_T^x
		\right]
		=
		\int_{D_E}F(\omega)\,\P_{X_T}(d\omega).
		\]
	\end{theorem}

	\begin{proof}
		The argument is the standard restart argument for interlaced Markov processes.
		We only give the main steps, and refer to
		Appendix~\ref{appsec:strong-markov} for details.
		
		First, the strong Markov property holds at the canonical jump times
		\((\tau_n)_{n\geq 0}\). Indeed, by the product representation
		\eqref{eq:pi-x}, conditionally on the input variables used before
		\(\tau_n\), the remaining input sequence is independent of the past and has
		the same law as a fresh sequence of inputs started from the post-jump state
		\(X_{\tau_n}\). Equivalently, for every bounded Borel functional
		\(F:D_E\to\mathbb R\),
		\[
		\mathbb E_x\!\left[
		F\bigl((X_{\tau_n+t})_{t\geq 0}\bigr)
		\,\middle|\,\cl F_{\tau_n}
		\right]
		=
		\int_{D_E}
		F(\omega)\,\P_{X_{\tau_n}}(d\omega),
		\qquad
		\P_x\text{-a.s. on }\{\tau_n<\infty\},
		\]
		
		Second, the same restart property holds at stopping times which occur before
		the next jump. More precisely, if \(S\) is a bounded stopping time such that
		\(S<\tau_1\), then the shifted first input consists of the shifted
		intra-stratum diffusion and the shifted Poisson random measure. By the strong
		Markov property of the intra-stratum diffusion and by the independent
		increments of the Poisson random measure, this shifted input has the law of a
		fresh input started from \(X_S\). Hence
		\[
		\mathbb E_x\!\left[
		F\bigl((X_{S+t})_{t\geq 0}\bigr)
		\,\middle|\,\cl F_S
		\right]
		=
		\int_{D_E}
		F(\omega)\,\P_{X_S}(d\omega),
		\qquad
		\P_x\text{-a.s. on }\{S<\tau_1\}.
		\]
		
		Let now \(T\) be a bounded stopping time. Since the jump times satisfy
		\(\tau_n\to\infty\) almost surely, the events
		\(
		A_n:=\{\tau_n\leq T<\tau_{n+1}\},
		\ n\geq 0,
		\)
		form a countable \(\cl F_T\)-measurable partition. On \(A_n\), applying
		the preceding pre-first-jump restart argument to the process shifted at \(\tau_n\)
		and to the stopping time \(T-\tau_n\), we obtain
		\[
		\mathbb E_x\!\left[
		F\bigl((X_{T+t})_{t\geq 0}\bigr)
		\,\middle|\,\cl F_T
		\right]
		=
		\int_{D_E}
		F(\omega)\,\P_{X_T}(d\omega)
		\qquad
		\P_x\text{-a.s. on }A_n .
		\]
		Summing over \(n\geq 0\) gives the same identity for every bounded stopping
		time \(T\). Finally, if \(T\) is finite but not necessarily bounded, the
		conclusion follows by applying the bounded case to \(T\wedge n\) and letting
		\(n\to\infty\). Therefore \((\P_x)_{x\in E}\) is a strong Markov
		family.
	\end{proof}

	\subsection{Uniqueness}
	\label{subsec:uniqueness}
	In this subsection we prove uniqueness of the martingale problem for
	\((A,\cl D_0)\). We use the disjoint-union
	space
	\[
	\widetilde E:=\bigsqcup_{\alpha\in I} S_\alpha,
	\]
	that is we endow the space $E$ with the disjoint-union topology instead of the topology induced by the distance $d$.
	We identify a function \(f:E\to\mathbb R\) with the
	corresponding function \(\tilde f:\widetilde E\to\mathbb R\) given by
	\(
	\tilde f(x)=f(x).
	\)
	Because $\widetilde{E}$ has the disjoint-union topology a function $f:\widetilde{E}\to \R$ is continuous if and only if its restriction $f_\alpha:= f|_{S_\alpha}$ is continuous for every $\alpha \in I$. Since $I$ is countable this gives
	\[
	C_0(\widetilde E)
	=
	\left\{
	f=(f_\alpha)_{\alpha\in I}:
	f_\alpha\in C_0(S_\alpha),
	\quad
	\forall\varepsilon>0,\ 
	\#\{\alpha:\|f_\alpha\|_\infty\geq\varepsilon\}<\infty
	\right\},
	\]
	with norm
	\(
	\|f\|_\infty=\sup_{\alpha\in I}\|f_\alpha\|_\infty.
	\)
	Moreover,
	\(\cl D_0 \subset C_0(\widetilde{E}).\)
	
	The auxiliary topology changes the continuous functions on the state space, but not its
	measurable structure. Indeed, the original topology of \(E\) fixes the Borel
	\(\sigma\)-field \(\cl B(E)\), and Definition~\ref{def:stratified-space}
	makes the stratification compatible with this measurable structure: each
	\(S_\alpha\) is a Borel subset of \(E\), and its manifold topology agrees with the
	topology induced by \(E\). Since \(I\) is countable, this implies
	\begin{equation}
		\label{eq:borel-equality}
		\cl B(\widetilde E)
		=
		\{A\subset E:\ A\cap S_\alpha\in\cl B(S_\alpha)
		\text{ for every }\alpha\in I\}
		=
		\cl B(E).
	\end{equation}
	Therefore the sets of finite signed Borel measures $\mathscr M(E)$ and  $\mathscr M (\widetilde{E})$, and the sets of bounded Borel functions $B_b(E)$ and $B_b(\widetilde E)$ can be canonically identified.

	The proof of uniqueness is based on three analytic lemmas, stated below. The first, namely the existence of a direct-sum Feller semigroup
	on the auxiliary disjoint-union space, is standard and does not rely on the geometry of the state space. Because of this, it is proved in
	Appendix~\ref{appsec:uniqueness}. We keep its statement in the main text because
	it contributes to the description of the structure of the uniqueness argument.
	
	We stress again that the uniqueness proof uses an auxiliary disjoint-union topology only as a
	functional-analytic device. The process remains \(E\)-valued, and the martingale
	problem is formulated on the original Borel state space; the auxiliary topology is used	to obtain a direct-sum Feller semigroup to which a bounded perturbation argument can
	be applied.
	
	\begin{lemma}\label{lem:direct-sum-feller}
		For
		\(f=(f_\alpha)_{\alpha\in I}\in C_0(\widetilde E)\), define
		\(
		(\widetilde P_t f)|_{S_\alpha}
		:=
		P_t^\alpha(f_\alpha).
		\)
		Then \((\widetilde P_t)_{t\geq0}: C_0(\widetilde E) \to C_0(\widetilde E)\) is a strongly continuous contraction semigroup.
		Its generator \(\widetilde L\) is given by
		\(
		\widetilde L f|_{S_\alpha}
		=
		L_\alpha(f_\alpha),
		\)
		with domain
		\[
		D(\widetilde L)
		=
		\left\{
		f\in C_0(\widetilde E):
		f_\alpha\in D( L_\alpha)\ \forall\alpha,
		\quad
		\bigl(L_\alpha f_\alpha\bigr)_\alpha
		\in C_0(\widetilde E)
		\right\}.
		\]
		Moreover, \(\cl D_0\) is a core for \(\widetilde L\).
	\end{lemma}

	\begin{lemma}
		\label{lem:resolvent-uniqueness}
		If \(\lambda>2\overline\Lambda\) and \(\rho\) is a finite signed Borel measure on
		\(E\) satisfying
		\begin{equation}
			\label{eq:resolvent-equal-zero}
		\int_E(\lambda-A)f\,d\rho=0,
		\qquad f\in\cl D_0,
		\end{equation}
		then \(\rho=0\).
	\end{lemma}

\begin{proof}
	Thanks to \eqref{eq:borel-equality}, we may regard \(\rho\) as a finite signed Borel measure
	on \(\widetilde E\) and every function in
	\(C_0(\widetilde E)\) may be regarded as a bounded Borel function on
	\(E\)
	
	The jump operator \(J\) is bounded from \(C_0(\widetilde E)\) into \(B_b(E)\). Indeed,
	for every \(f \in B_b(E)\),
	\[
	\begin{aligned}
		|Jf(x)|
		&\leq
		\int_E |f(y)-f(x)|\,\nu(x,dy) 
		\leq
		2\|f\|_\infty\nu(x,E)
		=
		2\|f\|_\infty\Lambda(x)
		\leq
		2\overline\Lambda\|f\|_\infty.
	\end{aligned}
	\]
	
	Let \(\lambda>2\overline\Lambda\), and suppose \eqref{eq:resolvent-equal-zero}, or
	equivalently,
	\begin{equation}
		\label{eq:signed-meas}
		\int_E(\lambda-\widetilde L)f\,d\rho
		=
		\int_E Bf\,d\rho,
		\qquad f\in\cl D_0.
	\end{equation}
	Equation \eqref{eq:signed-meas} can be extended to $D(\widetilde L)$, by selecting a sequence $(f_n)_{n \ge 0}$ that converges to $f$ so that
	\begin{gather*}
		\lambda\|f_n-f\|_\infty
		+
		\|\widetilde L f_n-\widetilde L f\|_\infty
		\to 0,\qquad
		\|Bf_n-Bf\|_\infty
		\le
		2\overline\Lambda\|f_n-f\|_\infty
		\to 0.
	\end{gather*}

	Let \(h\in C_0(\widetilde E)\). Denote by
	\(
	R_\lambda h
	:=
	\int_0^\infty e^{-\lambda t}\widetilde P_t h\,dt
	\)
	the resolvent of $\widetilde P_t$. Since \((\widetilde P_t)\) is a
	contraction semigroup, one has
	\(
	\|R_\lambda h\|_\infty
	\leq
	\frac1\lambda\|h\|_\infty.
	\)
	Moreover,
	\(
	R_\lambda h\in D(\widetilde L),
	\text{ and }
	(\lambda-\widetilde L)R_\lambda h=h.
	\)

	Therefore, using \eqref{eq:signed-meas} extended to $D(\widetilde L)$, we get, for \(\|h\|_\infty\leq1\),
	\[
	\begin{aligned}
		\left|\int_E h\,d\rho\right|
		&=
		\left|\int_E B R_\lambda h\,d\rho\right| 
		\leq
		\|\rho\|_{\mathrm{TV}}
		\,
		2\overline\Lambda
		\|R_\lambda h\|_\infty  \leq
		\|\rho\|_{\mathrm{TV}}
		\,
		\frac{2\overline\Lambda}{\lambda}
		\|h\|_\infty 
		\leq
		\frac{2\overline\Lambda}{\lambda}
		\|\rho\|_{\mathrm{TV}}.
	\end{aligned}
	\]
	Taking the supremum over all \(h\in C_0(\widetilde E)\) with
	\(\|h\|_\infty\leq1\), and using the Riesz-Markov-Kakutani theorem on the locally
	compact separable metric space \(\widetilde E\), gives
	\(
	\|\rho\|_{\mathrm{TV}}
	\leq
	(2\overline\Lambda
	\|\rho\|_{\mathrm{TV}})/\lambda.
	\)
	Since \(\lambda>2\overline\Lambda\), this forces
	\(
	\|\rho\|_{\mathrm{TV}}=0.
	\)
	Hence \(\rho=0\).
\end{proof}

	\begin{lemma}
		\label{lem:D0-measure-determining}
		The class \(\cl D_0\) is measure determining
		on \(E\). More precisely, if \(\mu\) and \(\nu\) are finite Borel measures on \(E\) and
		\begin{equation}
			\label{eq:measure-integral-identity}
			\int_E f\,d\mu=\int_E f\,d\nu,
			\qquad \text{for all }f\in\cl D_0,
		\end{equation}
		then \(\mu=\nu\).
	\end{lemma}
	
	\begin{proof}
		Assume \eqref{eq:measure-integral-identity} and fix \(\alpha\in I\). Since \(S_\alpha\) is a Borel subset of \(E\) and the inclusion is a homeomorphism, the restrictions \(\mu|_{S_\alpha}\) and \(\nu|_{S_\alpha}\)
		are finite Borel measures on \(S_\alpha\). Taking \(g \in \cl D_0\)
		supported only on \(S_\alpha\), we get
		\[
		\int_{S_\alpha} g\,d\mu
		=
		\int_{S_\alpha} g\,d\nu,
		\qquad g\in \cl C_\alpha.
		\]
		Since \(\cl C_\alpha\)
		is measure determining, one gets
		\(
		\mu|_{S_\alpha}=\nu|_{S_\alpha}.
		\)
		Because the strata form a countable Borel partition we conclude that \(\mu=\nu\) on \(E\).
	\end{proof}

	\begin{theorem}
		\label{thm:uniqueness}
		Suppose Assumptions \assref{ass:diffusions} and
		\assref{ass:jumps} hold. Then, for every \(x\in E\), the
		martingale problem for \((A,\cl D_0)\) has at most one solution on
		\(D_E\) with initial condition \(x\). Consequently, together with the
		existence result of Theorem~\ref{thm:existence}, the martingale problem is well-posed.
	\end{theorem}
	
	\begin{proof}
		Let \(\bb Q_x\) be any solution of the
		martingale problem for \((A,\cl D_0)\) with initial condition \(x\). Let
		\(Z_t\) denote the canonical process under \(\bb Q_x\). For $B \in \cl B(E)$, set
		\(
		Q_t(x,B):=\bb Q_x(Z_t\in B)
		\)
		and 
		\(
		P_t(x,B):=\P_x(X_t\in B),
		\) where $\P_x$ denotes the law on $D_E$ constructed in Theorem \ref{thm:existence}.
		We will show that the one-dimensional marginal distributions coincide, that is
		\begin{equation}
			\label{eq:equality-one-dim-marginals}
		Q_t(x,\cdot)=P_t(x,\cdot),
		\qquad t\geq0.
		\end{equation}
		To get uniqueness of the
		whole law, it will then suffice to apply \cite[Thm. IV.4.2]{ethier2009markov}.
		
		Fix \(f\in\cl D_0\). Since \(\bb Q_x\) solves the martingale problem,
		\[
		M_t^f
		:=
		f(Z_t)-f(x)-\int_0^t Af(Z_s)\,ds
		\]
		is a \(\bb Q_x\)-martingale. Taking expectations at deterministic times gives
		\[
		\int_E f\,Q_t(x,dy)
		=
		f(x)+\int_0^t\int_E Af\,Q_s(x,dy)\,ds.
		\]
		Equivalently, the function
		\(
		m_f^x(t):=\int_E f\,Q_t(x,dy)
		\)
		is absolutely continuous and satisfies
		\begin{equation}
			\label{eq:m-derivative}
			\frac{d}{dt}m_f^x(t)
			=
			\int_E Af\,Q_t(x,dy)
		\end{equation}
		for Lebesgue-a.e. \(t\geq0\).
		
		For \(\lambda>0\) and \(B \in \cl B(E)\), define the Laplace resolvent measure of the family
		\((Q_t(x,\cdot))_{t\geq0}\) by
		\(
		U_\lambda^{\bb Q}(x,B)
		:=
		\int_0^\infty e^{-\lambda t}Q_t(x,B)\,dt.
		\)
		This is a finite measure with total mass \(1/\lambda\). We claim that
		\begin{equation}
			\label{eq:resolvent-identity}
			\int_E(\lambda-A)f\,U_\lambda^{\bb Q}(x,dy)
			=
			f(x),
			\qquad f\in\cl D_0.
		\end{equation}
		Indeed,
		\[
		\begin{aligned}
			\int_E(\lambda-A)f\,U_\lambda^{\bb Q}(x,dy)
			&=
			\lambda\int_0^\infty e^{-\lambda t}m_f^x(t)\,dt
			-
			\int_0^\infty e^{-\lambda t}\int_E Af\,Q_t(x,dy)\,dt .
		\end{aligned}
		\]
		Since \(m_x^f\) is absolutely continuous and bounded, integration by parts and \eqref{eq:m-derivative} give
		\[
		\int_0^\infty e^{-\lambda t}\int_E Af\,Q_t(x,dy)\,dt
		=
		\int_0^\infty e^{-\lambda t}\frac{d}{dt}m_f^x(t)\,dt
		=
		-m_f^x(0)+\lambda\int_0^\infty e^{-\lambda t}m_f^x(t)\,dt.
		\]
		Since \(m_f^x(0)=f(x)\), we obtain \eqref{eq:resolvent-identity}.
		
		The constructed solution \(\P_x\) also solves the martingale problem. Therefore its
		resolvent measure
		\(
		U_\lambda^{\P}(x,B)
		:=
		\int_0^\infty e^{-\lambda t}P_t(x,B)\,dt
		\)
		satisfies \eqref{eq:resolvent-identity} as well.
		Consequently, the finite signed measure
		\(
		\rho_\lambda(\cdot)
		:=
		U_\lambda^{\bb Q}(x,\cdot)-U_\lambda^{\P}(x,\cdot)
		\)
		satisfies \eqref{eq:resolvent-equal-zero}.
		Lemma~\ref{lem:resolvent-uniqueness} gives, for every \(\lambda>2 \overline{\Lambda}\),
		\(
		\rho_\lambda=0
		\)
		and hence
		\(
		U_\lambda^{\bb Q}(x,\cdot)=U_\lambda^{\P}(x,\cdot).
		\)
		
		Now fix \(g\in\cl D_0\). From equality of the resolvent measures we get
		\[
		\int_0^\infty e^{-\lambda t}
		\left[
		\int_E g\,Q_t(x,dy)
		-
		P_tg(x)
		\right]dt
		=
		0,
		\qquad \lambda>2 \overline{\Lambda}.
		\]
		The function
		\(
		h_g(t)
		:=
		\int_E g\,Q_t(x,dy)-P_tg(x)
		\)
		is bounded and continuous. Indeed, both terms are absolutely
		continuous in \(t\), because both \(\bb Q_x\) and \(\P_x\) solve the martingale problem at
		deterministic times. Since the Laplace transform of \(h_g\) vanishes for all
		\(\lambda>2 \overline{\Lambda}\), the uniqueness of the Laplace transform gives
		\(
		h_g(t)=0
		\) for all \(t\geq0.\)
		Therefore, for $g \in \cl D_0$ and $t \ge 0$, 
		\(
		\int_E g\,Q_t(x,dy)
		=
		P_tg(x),
		\)
		and \eqref{eq:equality-one-dim-marginals} follows by Lemma~\ref{lem:D0-measure-determining}.
	\end{proof}

	\subsection{Ergodic theorems}
	\label{subsec:harris}
	
	We now record a Lyapunov-type sufficient condition for the existence and
	uniqueness of an invariant probability measure, together with positive Harris
	recurrence, for the stratified jump-diffusion $X$ constructed in
	Theorem~\ref{thm:existence}. The theory presented here is essentially taken from Meyn and Tweedie's works on Markov processes and their stability \cite{meyn1993stability,meyn1993stability3,meyn2012markov}.
	
	Let $(O_m)_{m \ge 0}$ be an open relatively compact exhaustion of $E$, and define 
	\( T_m:=\inf\{t\ge0:X_t\notin O_m\} \text{ and } X_t^{(m)}:=X_{t\wedge T_m}. \)
	We recall some important definitions from \cite{meyn1993stability,meyn1993stability3}.

	\begin{definition}[$\psi$-irreducibility]
		The process $X$ is called $\psi$-irreducible, for a non-zero $\sigma$-finite measure $\psi$ on $E$, if
		\(\psi(B)>0\) implies \(\E_x(\int_0^\infty \mathbf{1}_{\{X_t \in B\}}\, dt)>0\) for all \( x \in E.\)
	\end{definition}
	
	\begin{definition}[Aperiodicity]
		Let $X$ be $\psi$-irreducuble.
		We will say that $X$ is aperiodic if a suitable $D$-skeleton is strongly aperiodic, that is there exist $D,\,\epsilon>0$, a compact set $K\subset E$ such that $\psi(K)>0$ and a probability measure $\rho$, with $\rho(K)>0$, such that, for all $x \in K$, 
		\(P_D(x,\cdot) \ge \varepsilon \rho(\cdot). \)
	\end{definition}
	
	\begin{definition}[Harris recurrence]
		The process \(X\) is called Harris recurrent if there exists a non-zero $\sigma$-finite measure $\gamma$ such that for every set $A \in \cl B(E)$ with $\gamma(A)>0$, denoting $\tau_A:=\inf\{t \ge 0, X_t \in A\}$ one has
		\(\P_x(\tau_A < +\infty)=1,\)
		for all $x \in E$.
	\end{definition}
	
	If the process is Harris recurrent then there exists an essentially unique invariant measure $\pi$. If this invariant measure is finite then $X$ is called positive Harris recurrent.

	\begin{definition}[Truncated extended generator]
		\label{def:truncated-extended-generator}
		A measurable function $V:E\to\mathbb R$ is said to belong to the domain
		$D(A^\ast_m)$ if there exists a measurable function
		$G_m:E\to\mathbb R$ such that
		\[
			\mathbb E_x[V(X_t^{(m)})]
			=V(x)+\mathbb E_x\!\left[\int_0^{t \wedge T_m} G_m(X_s)\,ds\right],
			\qquad t\ge 0,\ x\in E
		\]
		and $\E_x [\int_{0}^{t\wedge T_m}|G_m(X_s)|\, ds]<\infty$
		for all $t\ge 0$.
		In that case we write $A^\ast_m V:=G_m$ on $O_m$.
	\end{definition}

	\begin{definition}[Petite set]
		A Borel set $C \in \cl B(E)$ is called petite for $X$ if there
		exist a probability measure $a$ on $(0,\infty)$ and a non--zero measure
		$\nu$ on $(E,\cl B(E))$ such that
		\[
			\int_0^\infty P_t(x,B)\,a(dt)
			\;\ge\; \nu(B)
			\quad\text{for all }x\in C,\ B\in\cl B(E).
		\]
	\end{definition}

	We now impose the following classical stability condition used for the long-time analysis. It is a Foster--Lyapunov drift condition for the stopped extended generators: outside a petite set, the expected infinitesimal drift of \(V\) is negative and pulls the process back toward a controlled region of the state space. The petite set provides the local recurrence/minorization ingredient needed to apply the standard Harris--Meyn--Tweedie theory.

	\begin{assumption}{L}[Lyapunov drift]
		\label{ass:Lyapunov}
		There exist \(V:E\to[1,\infty)\),
		a closed petite set \(C\), and constants \(\lambda,b>0\) such that \(V\) is bounded
		on \(C\), and, for all $m \ge 1$, \(V \in D(A^\ast_m)\) and
		\[
		A_m^\ast V(x)\le -\lambda V(x)+b\mathbf 1_C(x),
		\qquad x\in O_m.
		\]		
	\end{assumption}

	We can now state a Harris-type ergodic theorem and a Corollary that we shall use as
	black boxes in the covariance application. 
	
	\begin{theorem}
		\label{thm:harris}
		Suppose Assumptions \assref{ass:diffusions}, \assref{ass:jumps} and \assref{ass:Lyapunov} hold.
		Then the stratified jump-diffusion \(X\) is positive Harris recurrent. It has a unique invariant probability measure $\pi$ such that
		
		\begin{itemize}
			\item[(i)] \(X\) is Harris recurrent with respect to \(\pi\);
			\item[(ii)] $X$ is $\pi$-irreducible;
			\item[(iii)] 
			\(
			\int_E V(x)\,\pi(dx)<\infty;
			\)
			\item[(iv)] for every initial distribution \(\mu\) on \(E\),
			\[
			\frac1T\int_0^T \mu P_t\,dt \Longrightarrow \pi
			\qquad\text{as }T\to\infty,
			\]
			and in particular
			\[
			\frac1T\int_0^T f(X_t)\,dt \xrightarrow[T\to\infty]{a.s.} \int_E f\,d\pi
			\]
			for every bounded measurable \(f:E\to\mathbb R\).
		\end{itemize}
	\end{theorem}
	
	\begin{proof}
		By Theorems~\ref{thm:existence} and \ref{thm:strong-markov}, \(X\) is a conservative Borel right process on the
		locally compact separable metric space \(E\).
		Together with Assumption \assref{ass:Lyapunov}, the hypotheses of 
		\cite[Theorem~4.2]{meyn1993stability3} are satisfied. (i)-(iii) follow. (iv) is a standard consequence of positive
		Harris recurrence \cite{glynn1992uniform}.
	\end{proof}	
	\begin{corollary}
		\label{cor:geometric-ergodicity}
		In addition to Theorem~\ref{thm:harris}, assume that there exists a measurable function $G:E \to \R$ such that $A_m^\ast V=G$ on \(O_m\) for all $m \ge 1$. If, furthermore, $X$ is aperiodic then it is $V$--uniformly geometrically ergodic:
		there exist constants $M<\infty$ and $c>0$ such that
		\[
		\|P_t(x,\cdot)-\pi\|_V
		:=\sup_{|f|\le V}\left|P_t f(x)-\pi(f)\right|
		\le M V(x)e^{-c t},
		\qquad x\in E,\ t\ge0.
		\]
	\end{corollary}
	
	\begin{proof}
		This is Theorem 20.3.2 in \cite{meyn2012markov}. The extended generator definition assumed therein, which is in local martingale terms, follows from $V \in D(A_m^\ast)$ for every $m \ge 1$ and the existence of a single function $G:E \to \R$ such that $A_m^\ast V=G$ on \(O_m\).
	\end{proof}

	\begin{remark}[Stratified verification of \assref{ass:Lyapunov}]
		\label{rem:stratified-CD2-verification}	
		Assume that, for each stratum \(S_\alpha\), the restriction
		\(V_\alpha:=V|_{S_\alpha}\) is regular enough for the intra-stratum expression
		\(L_\alpha V_\alpha\) to be well defined. Define the formal stratified drift
		\(
		G(x)
		:=
		L_\alpha V_\alpha(x)
		+
		JV(x),
		\text{ for } x\in S_\alpha .
		\)
		
		First, for every \(m\ge1\), verify the local boundedness conditions
		\(
		\sup_{x\in O_m}
		|L_{\alpha(x)}V_{\alpha(x)}(x)|
		<\infty,
		\)
		and 
		\[
		\sup_{x\in O_m}
		\int_E \left|V(x)-V(y)\right|\nu(x,dy)
		<\infty.
		\]
		These estimates ensure that the drift and jump compensator are integrable up to
		the stopping time \(T_m\). Second, verify the stratumwise Dynkin's formula along
		each intra-stratum diffusion: for every \(\alpha\), every \(x\in S_\alpha\), and
		every bounded stopping time \(\tau\) before the diffusion leaves \(O_m\cap S_\alpha\),
		\[
		V_\alpha(\eta_\tau)-V_\alpha(\eta_0)
		-
		\int_0^\tau L_\alpha V_\alpha(\eta_s)\,ds
		\]
		is a martingale. Combining this stratumwise martingale identity with the
		compensator of the jump measure gives
		\[
		V(X_t^{(m)})-V(X_0)
		-
		\int_0^t
		\mathbf 1_{O_m}(X_s^{(m)})G(X_s^{(m)})\,ds
		\]
		as a martingale. Hence
		\(
		V\in D(A_m^*) \text{ and }
		A_m^*V(x)=G(x),
		\text{ for } x\in O_m.
		\)
		Thus the Lyapunov drift can
		be checked by estimating separately the intra-stratum contribution \(L_\alpha V\)
		and the jump contribution
		\(
		JV(x).
		\)
	\end{remark}
	
	\section{A rank-switching diffusion on the covariance cone}
	\label{sec:cov-example}
	
	In this section we construct a rank-switching diffusion on the cone of
	positive semidefinite covariance matrices. Matrix-valued stochastic processes on positive definite or correlation matrices have been studied in several directions, including positive definite jump diffusions \cite{MayerhoferPfaffelStelzer2011} and mean-reverting diffusions on correlation matrices \cite{AhdidaAlfonsi2013}. Our construction is different in that the full positive semidefinite cone is treated as a rank-stratified space, with continuous fixed-rank dynamics and stochastic transitions between different ranks.
	
	The covariance-cone model is a natural example of the abstract framework.
	The state variable is a positive semidefinite matrix, and its rank is an
	intrinsic geometric property rather than an external discrete label. The
	fixed-rank components are smooth manifolds of different dimensions, while
	the full cone carries the ambient topology inherited from the Euclidean space
	of symmetric matrices. Rank-switching is therefore a transition between strata of a singular geometric state space.

	\subsection{State space}
	In this Section, we introduce $\Cov(n)$ as a stratified space and describe the geometry of the fixed-rank covariance strata in coordinates
	adapted to the later stochastic construction. The point is to remove the non-uniqueness of matrix factorizations.
	A rank-\(r\) covariance matrix can be written as \(Q\exp(H)Q^\top\), where
	\(Q\in\St(n,r)\) describes its image subspace and \(H\in\Sym(r)\) describes the
	logarithmic covariance on that subspace. This representation is not unique:
	changing \(Q\) by a right orthogonal transformation changes \(H\) by conjugation.
	The associated-bundle quotient below records exactly this equivalence.

	Let
	\(
	\Sym(n)
	\)
	denote the Euclidean space of real symmetric \(n\times n\) matrices, endowed
	with the Frobenius norm \(\|\cdot\|_F\). Define
	\[
	\Cov(n)
	:=
	\{\Sigma\in \Sym(n):\Sigma\succeq 0\}.
	\]
	We endow \(\Cov(n)\) with the metric inherited from
	\(\Sym(n)\), that is
	\(
	d(\Sigma,\Lambda):=\|\Sigma-\Lambda\|_F.
	\)
	Since \(\Cov(n)\) is a closed convex cone in the finite-dimensional
	Euclidean space \(\Sym(n)\), it is locally compact and separable.
	
	For \(r=0,\ldots,n\), set
	\(
	S_r
	:=
	\{\Sigma\in\Cov(n):\rk(\Sigma)=r\}.
	\)
	Then \(\Cov(n)=\bigsqcup_{r=0}^n S_r.\)

	The stratum \(S_0\) is the singleton \(\{0\}\), equipped with its trivial
	smooth structure and trivial Riemannian metric. Fix \( r=1,\dots,n\). Let \(\St(n,r)\) denote the Stiefel manifold of orthonormal
	\(r\)-frames in \(\mathbb R^n\) and \(O(r)\) denote the group of $r$-dimensional orthogonal matrices. Define the right $O(r)$-action on
	\(\St(n,r) \times \Sym(r)\) by
	\begin{equation}
		\label{eq:right-action-log}
		R_G(Q,H):=(Q,H)\cdot G
		:=
		\bigl(QG,G^\top H G\bigr), \quad Q \in \St(n,r), H \in \Sym(r), G \in O(r).	
	\end{equation}
	
	We define the associated fiber bundle
	\(
		M_r
		:=
		\St(n,r)\times_{O(r)}\Sym(r)
		:=
		(\St(n,r)\times \Sym(r))/O(r),
	\)
	where the right action in \eqref{eq:right-action-log} is quotiented out. \(M_r\) is a fiber bundle associated with the principal \(O(r)\)-bundle
	\(
	\pi_{\Gr}:\St(n,r)\to \Gr(r,n),
	\ 
	Q \mapsto \Span(Q),
	\)
	where $\Gr(r,n)$ is the Grassmann manifold of $r$-dimensional subspaces of $\R^n$, \cite{bendokat2024grassmann,bonnabel2010riemannian}.
	
	We denote by \(q_r: \St(n,r)\times \Sym(r)\to M_r\) the quotient map and write equivalence classes as \([Q,H]=q_r(Q,H)\). Let
	\(
	\phi_r: M_r \to S_r, \
	[Q,H]\mapsto Q\exp(H)Q^\top.
	\)
	The map \(\phi_r\) is well defined because the matrix exponential is equivariant under orthogonal conjugation, and is a diffeomorphism.
	
	\begin{lemma}
		\label{lem:phi-diffeomorphism}
		The map \(
		\phi_r: M_r \to S_r, \
		[Q,H]\mapsto Q\exp(H)Q^\top .
		\) is a diffeomorphism
	\end{lemma}
	The proof is standard and moved to Appendix~\ref{appsec:cov-application}.

	Equip
	\(\Sym(r)\), \(\St(n,r)\) and \(\Gr(r,n)\) with the Frobenius metric.
	The product metric on
	\(\St(n,r)\times\Sym(r)\) is invariant under the
	right \(O(r)\)-action in \ref{eq:right-action-log}. Hence it descends to a Riemannian metric
	\(\tilde g_r\) on
	\(
	M_r.
	\)
	We define the Riemannian metric \(g_r\) on \(S_r\) by transport through
	\(\phi_r\), that is
	\(
	g_r
	:= \tilde g_r \circ \phi_r^{-1}.
	\)
	
	\begin{proposition}
		\label{prop:cov-state-space}
		The space
		\(\Cov(n)=\bigsqcup_{r=0}^n S_r\)
		satisfies Definition \ref{def:stratified-space}.
	\end{proposition}
	
	The proof is standard and moved to Appendix~\ref{appsec:cov-application}.

	\subsection{Intra-stratum diffusions}
	
	We now define the intra-stratum diffusion on each \(S_r\).	
	For \(r=0\), the process is constant:
	\(
	\Sigma_t\equiv 0.
	\)
	Its generator is
	\(
	L_0 f(0)=0.
	\)
	
	Let \(1\le r\le n\). We define the diffusion in logarithmic associated-bundle
	coordinates. The construction is performed in three steps. First, we define a
	horizontal Stiefel component. This component is defined using the horizontal distribution of the principal
	bundle \(\St(n,r)\to\Gr(r,n)\), that is 
	\(\cl H_Q=\{V\in T_Q\St(n,r):Q^\top V=0\}\) \cite{bendokat2024grassmann}, and describes the stochastic evolution
	of the \(r\)-dimensional image subspace.
	
	Second, we define a fiber component on
	\(\Sym(r)\), given by an Ornstein--Uhlenbeck dynamics. Finally, we project the product process
	\((Q_t,H_t)\) to the associated bundle
	\(M_r\)
	and transport it to the covariance stratum \(S_r\) through the diffeomorphism
	\(\phi_r\).
	
	On the compact manifold \(\St(n,r)\), let \(Q_t\) be the
	horizontal lift of Brownian motion on \(\Gr(r,n)\) defined as follows. Let $E_{ij}$ be the standard basis of \(\R^{n\times r}\) and define
	\(X_{ij}(Q):=\Pi_Q^\perp E_{ij}:=(I_n-Q Q^\top)E_{ij}.\)
	The family \((X_{ij})_{i,j}\) is the projection of $(E_{ij})_{i,j}$ onto the horizontal bundle and
	therefore it is a global Parseval frame of \(\cl H\). We define the Stratonovich SDE
	\begin{equation}
		\label{eq:q-sde}
		dQ_t
		=
		\sum_{i=1}^n\sum_{j=1}^r X_{ij}(Q_t)\circ dB_t^{ij},
	\end{equation}
	or equivalently
	\(dQ_t
	=
	(I_n-Q_tQ_t^\top)\circ dB_t,\)
	where $B_t=(B_t^{ij})$ is a Brownian motion in $\R^{}$.
	The corresponding process remains on the Stiefel manifold. Indeed, using the Stratonovich chain rule,
	\(
	d(Q_t^\top Q_t)
	=
	dQ_t^\top Q_t+Q_t^\top dQ_t=0.
	\)
	We call $(P_t^{\St})$ the transition semigroup of \(Q_t\) on \(C(\St(n,r))\), and \((L_Q,D(L_Q))\) its closed generator.

	For \(f\in C^\infty(\St(n,r))\), define
	\(
	\Delta_Q^{\hor}f
	:=
	\sum_{i=1}^n\sum_{j=1}^r X_{ij}^2f,
	\)
	where \(X_{ij}^2u\) means \(X_{ij}(X_{ij}u)\).
	Then,
	\begin{equation}
		\label{eq:c2-subset-domain}
		C^{\infty}(\St(n,r))\subset D(L_Q), \quad
		L_Qf=\frac12\Delta_Q^{\hor}f, \quad f \in C^\infty(\St(n,r)).
	\end{equation}
	Indeed, by Itô's formula,
	\[
	P_t^{\St}f-f
	=
	\frac12\int_0^t P_s^{\St}(\Delta_Q^{\hor}f)\,ds.
	\]
	Since \(\Delta_Q^{\hor}f\in C(\St(n,r))\) and \(P_t^{\St}\) is strongly
	continuous on \(C(\St(n,r))\), it follows that
	\(
	(P_t^{\St}f-f)/t
	\to
	(1/2)\Delta_Q^{\hor}f\)
	 in \(C(\St(n,r)).
	\)

	Independently, let \(H_t\) be the Ornstein--Uhlenbeck process on the Euclidean
	space \(\Sym(r)\):
	\begin{equation}
		\label{eq:h-sde}
		dH_t
		=
		-\theta_r H_t\,dt+\sigma_r\,dW_t,
		\qquad
		\theta_r>0,\quad \sigma_r>0,
	\end{equation}
	where \(W_t^{(r)}\) is Brownian motion in
	\(\Sym(r)\) with respect to the Frobenius inner product.
	Let \(P_t^{\Sym}\) be the corresponding transition semigroup on \(C_0(\Sym(r))\), and let \((L_H,D(L_H))\) be its closed generator. Again, using It\^o's formula, we can show \( C_c^\infty(\Sym(r))\subset D(L_H), \) and 
	\[ L_H g = -\theta_r\langle H,\nabla_Hg\rangle_F +\frac{\sigma_r^2}{2}\Delta_H g, \quad g \in C_c^\infty(\Sym(r)), \]
	where $\Delta_H$ is the Euclidean Laplacian on $\Sym(r)$.
	
	Let \((\widetilde P_t^r)_{t\ge0}\) be the transition semigroup of the product
	process \((Q_t,H_t)\) on \(\St(n,r)\times\Sym(r)\), and let
	\((\widetilde L_r,D(\widetilde L_r))\) be its generator on
	\(C_0(\St(n,r)\times\Sym(r))\). Since \(Q_t\) and \(H_t\) are independent, we have
	\(\widetilde P_t^r(f \otimes g)=(P_t^{\St}f) \otimes (P_t^{\Sym}g)\) for \( f \in C(\St(n,r))\) and \(g \in C_0(\Sym(r)).\)
	Moreover, on separated functions \(f\otimes g\), with
	\(f\in D(L_Q)\) and \(g\in D(L_H)\),
	\(
	\widetilde L_r(f\otimes g)
	=
	(L_Qf)\otimes g+f\otimes(L_Hg).
	\)

	Suppose that \((Q_t,H_t)\) solves the product SDE on the product space $\St(n,r)\times \Sym(r)$ with initial condition \((Q,H)\) and call \(
	(\widehat Q_t,\widehat H_t):=(Q_tG,G^\top H_tG)
	\)
	Then
	\[
	d\widehat Q_t
	=
	dQ_t\,G
	=
	(I_n-Q_tQ_t^\top)\circ dB_t\,G=(I_n -\widehat Q_t\widehat Q_t)\circ d(\widehat B_t) .
	\]
	The process \(\widehat B_t:= B_tG\) is again a standard Brownian motion in
	\(\mathbb R^{n\times r}\), because right multiplication by \(G\) is an orthogonal
	transformation of \(\mathbb R^{n\times r}\).
	
	Similarly,
	\[
	\begin{aligned}
		d\widehat H_t
		=
		G^\top dH_tG
		=
		-\theta_rG^\top H_tG\,dt
		+
		\sigma_r\,d(G^\top W_t^{(r)}G)=
		-\theta_r\widehat H_t\,dt
		+
		\sigma_r\,d\widehat W_t^{(r)},
	\end{aligned}
	\]
	where \(\widehat W_t^{(r)}\) is again a Brownian motion in
	\(\Sym(r)\), because conjugation is orthogonal with respect to the Frobenius metric.
	
	It follows that the process \((Q_t,H_t)\) is \(O(r)\)-equivariant, that is
	\begin{equation}
		\label{eq:invariance}\cl L_{R_G(Q,H)}((Q_t,H_t))=\cl L_{(Q,H)}(R_G(Q_t,H_t)),\end{equation}
	and, consequently, it descends to a diffusion \([Q_t,H_t]\) on
	\(
	M_r.
	\)
	Transporting through \(\phi_r\), we obtain an \(S_r\)-valued diffusion
	\(
	\Sigma_t^{(r)}
	=
	Q_t \exp(H_t) Q_t^\top.
	\) We denote its transition semigroup by \((P_t^r)_{t\ge0}\) and its generator by
	\((L_r,D(L_r))\).

	\begin{proposition}
		\label{prop:cov-intrastratum-diffusions}
		For every \(0\le r\le n\), the fixed-rank covariance dynamics $\Sigma_t$ defined above
		gives a conservative Markov family on \(S_r\) that satisfies Assumption \assref{ass:diffusions}. In particular, its transition semigroup \((P_t^r)_{t\ge0}\) is a strongly continuous contraction
		semigroup on \(C_0(S_r)\). Moreover,
		\(
		C_c^\infty(S_r)
		\)
		is a measure-determining core for \(L_r\). For \(r=0\), \(L_0f(0)=0\). For \(1\le r\le n\), \(L_r\) is identified
		in associated-bundle coordinates by
		\[
		(L_r f)\circ\phi_r\circ q_r
		=
		\left(
		\frac12\Delta_Q^{\mathrm{hor}}
		-\theta_r\langle H,\nabla_H\rangle_F
		+
		\frac{\sigma_r^2}{2}\Delta_H
		\right)
		(f\circ\phi_r\circ q_r),
		\qquad f\in C_c^\infty(S_r).
		\]
	\end{proposition}

	The proof works first on the product space $\St(n,r)\times \Sym(r)$ and proves the desired properties separately for the $Q$- and $H$-components, before transporting them through the quotient map $q_r$ and the diffeomorphism $\phi_r$. Technical details are given in Appendix~\ref{appsec:cov-application}.

	\subsection{Rank-changing jumps}
	
	Here, we define the rank-changing jump mechanism. For the sake of simplicity, we use a state-independent refreshment-type mechanism: when a jump from
	rank \(r\) to rank \(s\) occurs, the post-jump covariance matrix is sampled from a
	prescribed probability law on \(S_s\).
	
	The vertex set is
	\(
	I=\{0,1,\ldots,n\}.
	\)
	For \(0\le r<n\), we include an upward edge
	\(
	e_r^+:r\to r+1,
	\)
	and for \(0 < r\le n\), we include a downward edge
	\(
	e_r^-:r\to r-1.
	\)	
	Fix constant jump rates
	\(
	\lambda_r^+\equiv c_r^+>0\) for \(
	r<n\), \(\lambda_n^+ \equiv 0\), and
	\(
	\lambda_r^- \equiv c_r^->0\) for 
	\(r>0,\) \(\lambda_0^-\equiv 0.
	\)

	Choose constants
	\(
	0<\underline \ell<\overline \ell<\infty
	\)
	and, for \(r=1,\ldots,n\), set
	\(
	A_r
	:=
	\{
	\Sigma\in S_r:
	\underline \ell\le \lambda_i^+(\Sigma)\le \overline \ell,
	\ i=1,\ldots,r
	\},
	\)
	where
	\(\lambda_1^+(\Sigma),\ldots,\lambda_r^+(\Sigma)\) are the positive
	eigenvalues of \(\Sigma\). 
	Thus \(A_r\) is compact in \(S_r\). Let \(R_r\) be a probability measure
	supported on \(A_r\) defined by 
	\(
	R_r(d\Sigma)=\rho_r(\Sigma)\vol_r(d\Sigma),
	\) with \(\rho_r(\Sigma)\ge a_r>0\text{ on }A_r.
	\)
	For \(r=0\), set
	\(
	A_0:=\{0\}
	\text{ and }
	R_0:=\delta_0.
	\)

	For $\Sigma \in S_r$, define the jump kernels
	\begin{gather*}
	K_r^+(\Sigma,\cdot)
	:=
	R_{r+1}(\cdot),
	\qquad 0 \le r < n,\\
	K_r^-(\Sigma,\cdot)
	:=
	R_{r-1}(\cdot),
	\qquad 0< r\le n.
	\end{gather*}
	
	The generator of the full rank-switching process acts on
	\(
	\cl D_0
	:=
	\bigoplus_{r=0}^n C_c^\infty(S_r)
	\)
	as follows. 
	For $\Sigma \in S_{r}$, let
	\(
	\nu(\Sigma,d\Xi)
	=
	\lambda_r^+R_{r+1}(d\Xi)+\lambda_r^-R_{r-1}(d\Xi).
	\)
	Define, for $f \in \cl D_0$,
	\[
	\begin{aligned}
		Jf(\Sigma)
		:=&
		\int_{\Cov(n)}
		\bigl(f(\Xi)-f(\Sigma)\bigr)\nu(\Sigma,d\Xi).
	\end{aligned}
	\]
	Finally set
	\(Af(\Sigma)
		:=
		L_r(f|_{S_r})(\Sigma)
		+Jf(\Sigma).
	\)

	The graph has finitely many vertices and finitely
	many edges. The kernels \(K_r^\pm\) are Borel probability kernels. Moreover,
	\(
	K_r^+(\Sigma,S_{r+1})=1,
	\text{ and }
	K_r^-(\Sigma,S_{r-1})=1.
	\)
	Since the target stratum has different rank from the starting stratum,
	\(
	K_r^\pm(\Sigma,\{\Sigma\})=0.
	\)
	The rates are constant, hence Borel. Finally, the total jump rate is uniformly
	bounded. Indeed
	\(
	\sup_{\Sigma\in\Cov(n)}\Lambda(\Sigma)
	\le
	\max_{0\le r\le n}(c_r^+ + c_r^-)
	<\infty.
	\)
	Thus Assumption \assref{ass:jumps} holds.
	
	We have therefore verified Definition \ref{def:stratified-space} and Assumptions \assref{ass:diffusions}, \assref{ass:jumps}. Consequently Theorems \ref{thm:existence}, \ref{thm:strong-markov} and \ref{thm:uniqueness} apply. In particular, the martingale problem for \((A,\cl D_0)\) is well-posed and there exists a unique conservative strong Markov solution process, denoted by \((\Sigma_t)_{t\ge 0}\).

	\subsection{Lyapunov drift, petite sets and aperiodicity}
	\label{subsec:cov-lyapunov}

	Define the function
	\(
	V:\Cov(n)\to[1,\infty),
	\ 
	\Sigma \mapsto 1+\Tr(\Sigma).
	\)
	It is continuous on \(\Cov(n)\) and smooth on each
	stratum \(S_r\). 	
	Moreover, define
	\(
	O_m:=\{\Sigma\in\Cov(n):V(\Sigma)<m\}.
	\)
	Each \(O_m\) is open in \(\Cov(n)\), while, in view of the fact that \(\|\Sigma\|_F \le \Tr(\Sigma)\), \(C_m:= \overline O_m=\{\Sigma\in\Cov(n):V(\Sigma) \le m\}\) is closed and bounded in $\Sym(n)$ and hence compact. Consequently, it is compact in $\Cov(n)$.
	Therefore
	\((O_m)_{m\ge2}\) is an open relatively compact exhaustion of
	\(\Cov(n)\).
	Finally, let
	\(
	T_m:=\inf\{t\ge0:\Sigma_t\notin O_m\} \text{ and }
	\Sigma_t^{(m)}:=\Sigma_{t\wedge T_m}.
	\)

	\begin{proposition}
		\label{prop:cov-lyapunov-drift}
		The function
		\(V\)
		belongs to \(D(A^*_m)\) for all $m \ge 2$. Moreover, there exist
		constants \(\lambda>0\), \(b<\infty\), and \(k<\infty\) such that
		one has, for $k \ge b/\lambda$,
		\[
		A^*_mV(\Sigma)
		\le
		-\lambda V(\Sigma)+b\mathbf 1_{C_k}(\Sigma),
		\qquad
		\Sigma\in O_m.
		\]
	\end{proposition}
	
	The proof of Proposition~\ref{prop:cov-lyapunov-drift} uses the strategy outlined in Remark \ref{rem:stratified-CD2-verification}. In associated-bundle coordinates, $V$ is a function only of $H$, while independent of $Q$. The Lyapunov estimate is, then, driven by the Ornstein--Uhlenbeck drift in the $H$-component: for large positive eigenvalues, the term \(-\theta_r H_t\) dominates the diffusion contribution and forces a negative drift for \(V(\Sigma)=1+\Tr(\Sigma)\). The jump contribution is harmless for this estimate because the post-jump distributions are supported in fixed compact spectral regions. Details are given in Appendix~\ref{appsec:cov-application}.

	\begin{proposition}
		\label{prop:CR-petite}
		Fix \(h>0\), \(k >0\) and
		\(t_0>0\), and set
		\(
		D:=2nh+t_0 .
		\)
		Then there exist \(\varepsilon>0\) and a non-zero finite measure
		\(\rho\) such that, for \(\Sigma\in C_k \text{ and }
		B\in\cl B(\Cov(n))\), 
		\(
		P_{D}(\Sigma,B)
		\ge
		\varepsilon\rho(B).
		\)
		In particular \(C_k\) is petite.
	\end{proposition}

	\begin{proof}
		Choose a compact set
		\begin{equation}
			\label{eq:K-def}
			K\Subset A_n\Subset S_n
		\end{equation}
		with \(\vol_n(K)>0\). The top stratum \(S_n\) is identified with
		\(\Sym(n)\) through \(H\mapsto \exp(H)\). Hence the top-rank diffusion is a
		non-degenerate Ornstein--Uhlenbeck diffusion in logarithmic coordinates. Its
		transition density \(p^{(n)}_t(\Xi,\Upsilon)\), with respect to \(\vol_n\), is smooth
		and strictly positive for every \(t>0\).
		Define the finite nonzero measure
		\(
		\rho(B):=\vol_n(B\cap K).
		\)
		
		Fix an initial point \(\Sigma\in C_k\cap S_r\). We force the rank path
		\[
		r\to r-1\to\cdots\to0\to1\to\cdots\to n.
		\]
		This path has
		\(
		N_r:=r+n\le 2n
		\)
		jumps. 
		Define
		\(
		a_r
		:=
		(\prod_{k=1}^r c_k^-)
		(\prod_{k=0}^{n-1} c_k^+),
		\)
		where the empty product is interpreted as \(1\). Since all the involved rates are
		strictly positive, \(a_r\) is strictly positive as well.
		
		We divide the time interval \([0,2nh]\) into \(2n\) subintervals of length \(h\).
		On the first \(N_r\) subintervals, require that exactly one jump occurs, that this jump
		is the prescribed one in the above rank path, and that no other jump occurs in that
		subinterval. On the remaining time interval, require that no further rank-changing
		jump occurs. Since all total jump rates are bounded by \(\overline\Lambda\), the
		probability of this event is bounded from below by
		\(
		a_r h^{N_r} e^{-\overline\Lambda D}.
		\)
		
		After the last prescribed upward jump, the process lies in \(S_n\), and the landing law
		is \(R_n\), which is supported on \(A_n\). From that last jump time until time
		\(D\), no further jump occurs on the event just described. The remaining
		top-stratum diffusion time belongs to the interval \([t_0,t_0+2nh]\). Therefore, defining
		\[
		\delta_{h,t_0,K}
		:=
		\inf_{\substack{
				s\in[t_0,t_0+2nh]\\
				\Xi\in A_n,\ \Upsilon\in K}}
		p^{(n)}_s(\Xi,\Upsilon)
		>0,
		\]
		we get
		\(
		P_{D}(\Sigma,B)
		\ge
		a_r h^{N_r} e^{-\overline\Lambda D}
		\delta_{h,t_0,K}\,
		\rho(B).
		\)
		Taking the minimum over the finitely many initial ranks gives the desired inequality with
		\(
		\varepsilon
		:=
		\delta_{h,t_0,K}
		e^{-\overline\Lambda D}
		\min_{0\le r\le n}
		[
		h^{N_r}
		a_r
		]
		>0 .
		\)
		Since a one-time minorization is a petite-set minorization with \(a=\delta_{D}\),
		the set \(C_k\) is petite.
	\end{proof}

	By the Lyapunov drift condition and the petite-set property, the process $\Sigma_t$ is
	positive Harris recurrent and admits a unique invariant probability measure
	\(\pi\). In particular, it is \(\pi\)-irreducible.
	
	\begin{proposition}
		\label{prop:cov-aperiodicity}
		The \(D\)-skeleton is strongly aperiodic in the
		\(\pi\)-irreducible sense. Consequently the continuous-time rank-switching
		covariance process is aperiodic.
	\end{proposition}

	\begin{proof}
	Define the compact set $K$ as in \eqref{eq:K-def}.
	Choose \(k\) large enough so that
	\(
	\pi(C_k)>0
	\)
	and
	\(
	K\subset C_k.
	\)
	The minorizing measure \(\rho\) can be chosen to charge \(K\), as shown in the proof of Proposition~\ref{prop:CR-petite}. Therefore
	\(
	\rho(C_k)\ge \rho(K)>0.
	\)
	Thus, for the \(D\)-skeleton, for all $\Sigma \in C_k$
	\(
	P_{D}(\Sigma,\cdot)
	\ge
	\varepsilon\rho(\cdot),
	\)
	with
	\(
	\pi(C_k)>0 \text{ and }
	\rho(C_k)>0.
	\)
\end{proof}

	We have therefore verified Assumption \assref{ass:Lyapunov}, as well
	as aperiodicity. Hence the process $\Sigma_t$ is positive Harris recurrent and 
	\(V\)-uniformly geometrically ergodic with
	\(
	V(\Sigma)=1+\Tr(\Sigma).
	\)
	
	\appendix
	
	\section{Interlacing construction}
	\label{appsec:interlacing}
	Here we present the technical parts needed for the interlacing construction and Theorem~\ref{thm:existence}.
	
	\begin{lemma}
		\label{lem:s-j-meas}
		The maps $\sigma:H \to (0 ,+\infty]$ and $\Psi:H \to E$ are measurable.
	\end{lemma}
	\begin{proof}
		We start by showing that the map $D_E\times (0,+\infty)\ni(\eta,r) \mapsto \eta(r-)$ is Borel. Indeed, defining
		\(
		q_n(r):=2^{-n}(\lceil 2^n r\rceil-1)\vee 0, 
		\)
		we have \(q_n(r)<r\) and \(q_n(r)\to r\). Since \(\eta\) is càdlàg,
		\(
		\eta(q_n(r))\to \eta(r-).
		\)
		For fixed \(n\), the map \(q_n\) is a countably-valued Borel function. Hence, for every
		open \(O\subset E\),
		\[
		\{(\eta,r):\eta(q_n(r))\in O\}
		=
		\bigcup_{k\ge0}
		\{r:q_n(r)=k2^{-n}\}
		\times
		\{\eta:\eta(k2^{-n})\in O\},
		\]
		which is Borel because fixed-time evaluations on \(D_E\) are Borel. Thus
		\((\eta,r)\mapsto \eta(q_n(r))\) is Borel for every \(n\). Since \(E\) is metric and
		\(
		\eta(r-)=\lim_{n\to\infty}\eta(q_n(r)),
		\)
		the map \((\eta,r)\mapsto \eta(r-)\) is Borel.

		Since \(Z\) is locally compact and separable, the space \(N_Z\) of locally
		finite counting measures admits measurable enumerations of atoms (see \cite[Lem. 9.1.XIII]{daley2008introductionII}). Let
		\(a_i(m)=(r_i(m),s_i(m),u_i(m))\), \(i\geq1\), be such an enumeration, with a
		cemetery value $\partial$ for absent atoms. Define
		\[
		T_i(h):=
		\begin{cases}
			r_i(m), & a_i(m)\neq\partial,\ 
			0<s_i(m)\leq \Lambda(\eta(r_i(m)-)),\\
			+\infty, & \text{otherwise}.
		\end{cases}
		\]
		Since \((\eta,r)\mapsto \eta(r-)\) is Borel and \(\Lambda\) is Borel, each \(T_i\)
		is measurable. Hence
		\(
		\sigma(h)=\inf_{i\geq1}T_i(h)
		\)
		is measurable. 
		
		On the set where \(0<\sigma(h)<\infty\) and the first accepted atom is unique, let
		\(I(h)\) be the unique index such that
		\(T_I(h)=\sigma(h)\). Then \(I\), and therefore the marks
		\(
		S(h):=s_I(m)\text{ and } U(h):=u_I(m),
		\)
		are measurable. Since the edge selector \((z,s)\mapsto e(z,s)\) is Borel and the maps
		\(F_e\) are Borel, the landing map \((z,s,u) \mapsto F_{e(z,s)}(z,u)\) is Borel. Thus, on the good set,
		\(
		\Psi(h)
		=
		F_{e(\eta(\sigma(h)-),S(h))}(\eta(\sigma(h)-),U(h))
		\)
		is measurable. On the complementary set, which is Borel, we set
		\(\Psi(h)=x_\ast\). Hence \(\Psi\) is measurable and the claim is proven.
	\end{proof}

	\begin{lemma}
		\label{lem:Chi-borel}
		The map \(\Chi: E \times J \to D_E\) is measurable.
	\end{lemma} 
	\begin{proof}
					
			Since \(E\) is a separable
			metric space, the Borel \(\sigma\)-field on \(D_E\) is
			generated by the coordinate maps
			\(
			\pi_t:D_E\to E,\  \pi_t(\omega)=\omega(t)\text{ for }t\ge 0.
			\)
			Consequently, it suffices to prove that
			\(
			\pi_t\circ \Chi=\Chi_t: E \times J \to E
			\)
			is measurable for every fixed \(t\ge0\).
			We shall use the fact that the variable-time evaluation map
			\(
			(\eta,r) \to \eta(r),
			\)
			is Borel. 
			This can be proved using the same strategy we used for the map $(\eta,r)\to \eta(r-)$ substituting left approximation with right approximation, namely \(p_k(r):=2^{-k}\lceil 2^k r\rceil\)

			We now turn to the quantities entering the definition of
			\(\cl X\). The coordinate maps
			\(
			J\ni j\mapsto h_n\in H
			\)
			are measurable by definition of the product \(\sigma\)-field on
			\(J=H^{\mathbb N_0}\). By Lemma~\ref{lem:s-j-meas}, both
			\(
			\sigma:H\to(0,\infty]\) and
			\(\Psi:H\to E
			\)
			are measurable. Therefore, for every \(n\ge0\),
			\(
			\sigma_{n+1}(j)=\sigma(h_n)
			\)
			is measurable. Hence, the recursively
			defined jump times
			are measurable maps
			\(
			\tau_n:J\to[0,\infty].
			\)
			Moreover, every map
			\(
			(x,j)\mapsto x_n(x,j)
			\)
			is measurable.
			
			Next we prove that the good set \(B_{\mathrm{good}}\) is measurable. 
			For fixed \(n\), the map
			\(
			(x,j)\mapsto \eta_n(0)
			\)
			is measurable, because it is the composition of the coordinate projection
			\(j\mapsto h_n\), the projection \(h_n\mapsto \eta_n\), and
			the fixed-time evaluation \(\eta_n\mapsto \eta_n(0)\). Since \(E\) is
			metric, the diagonal
			\(
			\{(z,z):z\in E\}
			\)
			is closed in \(E\times E\), hence Borel. Thus
			\(
			\{(x,j):\eta_n(0)=x_n(x,j)\}
			\)
			is measurable. Therefore
			\(
			\{(x,j):\tau_n(j)=\infty\}
			\cup
			\{(x,j):\eta_n(0)=x_n(x,j)\}
			\)
			is measurable as well. Taking the countable intersection over \(n\), we obtain the
			measurability of the first defining condition of
			\(B_{\mathrm{good}}\).
			
			The second condition is Borel because
			\(
			\{(x,j):\tau_n(j)=\infty\}
			\cup
			\{(x,j):\sigma(h_n)>0\}
			\)
			is measurable for every \(n\), and we again take a countable intersection.			
			Finally,
			\[
			\{\tau_n(j)\to\infty\}
			=
			\bigcap_{M\in\mathbb N}
			\bigcup_{n\ge0}
			\{j:\tau_n(j)>M\},
			\]
			because the sequence \((\tau_n)_{n\ge0}\) is nondecreasing. The right-hand
			side is measurable. Hence \(B_{\mathrm{good}}\in
			\cl B(E)\otimes\cl J\).
			
			We now show that, on \(B_{\mathrm{good}}\), the formula defining
			\(\cl X\) really gives an element of \(D_E\). Fix
			\((x,j)\in B_{\mathrm{good}}\). Since \(\tau_n(j)\to\infty\), every
			compact time interval \([0,T]\) intersects only finitely many intervals
			\(
			[\tau_n(j),\tau_{n+1}(j)).
			\)
			Moreover, by the condition \(\sigma(h_n)>0\) whenever \(\tau_n(j)<\infty\),
			the finite jump times are strictly increasing. On each open interval
			\(
			(\tau_n(j),\tau_{n+1}(j))
			\)
			the path is given by
			\(
			t\mapsto \eta_n(t-\tau_n(j)),
			\)
			and is therefore càdlàg.
			
			At
			concatenation times $\tau_n(j), \, n \ge 0$, one has
			\(
			\cl X(x,j)(\tau_n(j))=\eta_n(0),
			\)
			so right-continuity follows from right-continuity of \(\eta_n\).
			The
			left limit at \(\tau_n(j)\) is
			\[
			\lim_{t\to \tau_n(j)}
			\cl X(x,j)(t)
			=
			\lim_{s\to \sigma_n(j)}
			\eta_{n-1}(s)
			=
			\eta_{n-1}(\sigma_n(j)-),
			\]
			which exists because \(\eta_{n-1}\) is càdlàg. 
			Thus \(\cl X(x,j)\in D_E\) on \(B_{\mathrm{good}}\). On
			\((E\times J)\setminus B_{\mathrm{good}}\) we define
			\(\cl X(x,j)\equiv x\),
			which is again an element of \(D_E\). Therefore \(\cl X\) is a
			well-defined map from \(E\times J\) to \(D_E\).
			
			Now, fix \(t\ge0\). 			
			For \(n\ge0\), define the measurable sets
			\(
			A_{n,t}:=
			B_{\mathrm{good}}\cap\{(x,j):\tau_n(j)=t\},
			\)
			and
			\(
			C_{n,t}:=
			B_{\mathrm{good}}\cap\{(x,j):\tau_n(j)<t<\tau_{n+1}(j)\}.
			\)
			For fixed \(t\), the family
			\(
			\{B_{\mathrm{good}}^\complement, (A_{n,t})_{n \ge 0},  (C_{n,t})_{n\ge0}\}
			\)
			is a countable measurable partition of \(E \times J\). 
			On \(B_{\mathrm{good}}^\complement\), we have
			\(
			\cl X_t(x,j)=x,
			\)
			which is Borel. On \(A_{n,t}\), we have
			\(
			\cl X_t(x,j)=x_n(x,j),
			\)
			which is Borel, as shown above.
			
			It remains to treat \(C_{n,t}\). Define
			\[
			r_{n,t}(j):=
			\begin{cases}
				t-\tau_n(j), & \tau_n(j)<\infty,\\
				0, & \tau_n(j)=\infty .
			\end{cases}
			\]
			This is a measurable map from \(J\) to \([0,\infty)\). It follows that the map
			\(
			(x,j) \mapsto \eta_n(r_{n,t}(j))
			\)
			is Borel.
			On \(C_{n,t}\), this map is exactly
			\(
			\cl X_t(x,j)=\eta_n(t-\tau_n(j)).
			\)
			Therefore \(\cl X_t\) is obtained by gluing countably many Borel maps
			along a countable Borel partition, and hence is Borel.
	\end{proof}

	\begin{lemma}
		\label{lem:bad-sets}
		The sets
		\begin{gather*}
			B_{\mathrm{mark}}:=\{\text{for some }n \text{ the first accepted atom of }m_n\text{ is not unique}\},\\
			B_{\sigma}
			:=
			\left\{
			\exists n\geq 0 \text{ such that }
			\tau_n(j)<\infty
			\text{ and }
			\sigma(h_n)=0
			\right\},\\
			B_{\eta}
			:=
			\left\{
			\exists n\geq 0 \text{ such that }
			\tau_n(j)<\infty
			\text{ and }
			\eta_n(0)\neq x_n(x,j)
			\right\},\\
			B_{\mathrm{acc}}:=\{\tau_n \not \to \infty\}
		\end{gather*}
		are $\Pi_{x}$ null.
	\end{lemma}
	\begin{proof}
		
		Fix \(n\). Conditionally on \(h_0,\ldots,h_{n-1}\), the random measure
		\(m_n\) is a Poisson random measure on
		\(
		Z
		\)
		with intensity \(dt\,da\,du\). Its projection onto the time coordinate is a
		Poisson point process on \((0,\infty)\) with intensity \(\overline\Lambda\,dt\).
		Since this intensity is nonatomic, almost surely no two atoms have the same
		time coordinate. Hence, whenever the first accepted atom exists, it is almost surely
		unique. Taking the countable union over \(n\), we obtain
		\(
		\Pi_x(B_{\mathrm{mark}})=0.
		\)

		Moreover, for every \(\varepsilon>0\),
		\[
		\{\sigma(h_n)=0\}
		\subset
		\left\{
		m_n\left((0,\varepsilon]\times(0,\overline\Lambda]\times(0,1)\right)\ge1
		\right\}.
		\]
		Consequently,
		\(
		\Pi_x(\sigma(h_n)=0\mid h_0,\ldots,h_{n-1})
		\le 1-e^{-\overline\Lambda\varepsilon}.
		\)
		Letting \(\varepsilon\to0\), we obtain \(\Pi_x(\sigma(h_n)=0)=0\) for every \(n\), and by countable
		subadditivity,
		\(
		\Pi_x(\exists n\ge0:\sigma(h_n)=0)=\Pi_x(B_{\sigma})=0.
		\)
		
		Since \(h_0\) has law \(\mathbb K(x;dh)=\mathbb Q_x(d\eta)\mathbb M(dm)\), assumption \ref{it:D2} gives
		\(
		\Pi_x(\eta_0(0)=x)=1.
		\)
		Moreover, for \(n\ge1\), conditionally on \(h_0,\ldots,h_{n-1}\), the
		input \(h_n\) has law
		\(
		\mathbb K(\Psi(h_{n-1});dh).
		\)
		Therefore
		\(
		\Pi_x\bigl(\eta_n(0)=J(h_{n-1})\mid h_0,\ldots,h_{n-1}\bigr)=1.
		\)
		Since \(x_n(x,j)=J(h_{n-1})\) for \(n\ge1\), it follows that
		\(
		\Pi_x(
		\eta_n(0)=x_n(x,j)\text{ for every }n\ge0
		)=1.
		\)
		Equivalently, \(\Pi_x(B_{\eta})=0\).
		
		For each \(n\), let
		\(\xi_{n+1}\) be the first atom time of \(m_n\) in the unrestricted space $Z$.
		Then \(\xi_{n+1}\) has exponential law with parameter \(\overline\Lambda\),
		and the variables \(\xi_1,\xi_2,\ldots\) are independent under $\Pi_x$. Since accepted atoms are a subset of all atoms,
		\(
		\sigma_{n+1}\ge \xi_{n+1}.
		\)
		Therefore
		\(\tau_n=\sum_{j=1}^n \sigma_j \ge \sum_{j=1}^n \xi_j \to \infty\),
		\(\Pi_x\)-a.s., which gives \(\Pi_x(B_{\mathrm{acc}})=0\).
	\end{proof}

	\begin{proposition}
		\label{prop:iii-iv}
		Let $\P_x$ be the solution to $(A,\cl D_0)$ constructed in Theorem~\ref{thm:existence} and call $X$ the corresponding process on $E$. Then:
		\begin{itemize}
			\item[(iii)] the jumps of $X$ are governed by the structural kernels and rates:
			whenever $X_{t-}=x\in S_\alpha$, jumps of type $e\in\cl E$ with $\alpha(e)=\alpha$
			occur with conditional intensity $\lambda_e(x)$, and conditional on such a jump of type $e$,
			the post-jump location has law $K_e(x,\cdot)$;
			\item[(iv)] for every $f\in\cl D_0$, the process
			\[
			M_t^f=f(X_t)-f(X_0)-\int_0^t A f(X_s)\,ds
			\]
			is an $(\cl F_t^x)$--martingale. In particular, the law of $X$
			on $D_E$ is a solution of the martingale problem for
			$(A,\cl D_0)$ with initial condition $x$.
		\end{itemize}
	\end{proposition}
	\begin{proof}
		We define 
		\(\cl F_t^0=\sigma(X_s, 0\le s \le t)\)
		and $\cl F_t^x$ its right-continuous completed version under $\bb P_x$.
		On $(J,\cl J, \Pi_x)$, we denote
		\(\widehat{\cl F}_t^0:= (\Chi^x)^{-1}(\cl F_t^0)=\sigma (X_s \circ \Chi^x, 0\le s \le t)=\sigma (\Chi_s^x, 0\le s \le t)\)
		and call $\widehat{\cl F}_t^x$ its right-continuous completion under $\Pi_x$. If $N \in \cl F_t^x$ is $\bb P_x$-null, then $(\Chi^x)^{-1}N$ is $\Pi_x$-null. This gives
		\((\Chi^x)^{-1} \cl F_t^x \subset \widehat{\cl F}_t^x.\)
		
		Moreover, on $(J,\cl J, \Pi_x)$, we denote by \(\cl G^x=(\cl G_t^x)_{t\ge0}\)  the completed, right-continuous
		construction filtration generated non-anticipatively by the recursive
		construction: it contains the Poisson atoms and marks used up to physical time
		\(t\), the diffusion segments up to the physical times at which they have
		been used, and the post-jump states already produced.
		Formally, for \(t\ge0\), set
		\(
		\theta_n(t):=((t-\tau_n)^+)\wedge\sigma_{n+1}.
		\)
		Let \(\mathcal G_t^{0,x}\) be the \(\sigma\)-field generated by the random variables
		\(
		\mathbf 1_{\{\tau_n\le t\}}\eta_n(s\wedge\theta_n(t))
		\)
		and
		\(
		\mathbf 1_{\{\tau_n\le t\}}
		m_n((0,s\wedge\theta_n(t)]\times C)
		\), where \( s\ge0,\ C\in\cl B((0,\overline\Lambda]\times(0,1)) \text{ and } n\ge0,\)
		together with the already produced post-jump states
		\(
		\mathbf 1_{\{\tau_n\le t\}}x_n,
		\) for \(n \ge 0\),
		where arbitrary cemetery values are used on the complements of the events
		\(\{\tau_n\le t\}\). Let \((\mathcal G_t^x)_{t\ge0}\) then be
		right-continuous completion of \((\mathcal G_t^{0,x})_{t\ge0}\) under \(\Pi_x\).

		By construction, the lifted process
		\(
		X^x_t=X_t\circ\Chi^x
		\)
		is \(\cl G^x_t\)-adapted, and the structural jump times
		\((\tau_n)_{n\ge0}\) are \((\cl G^x_t)\)-stopping times. Moreover, if
		\(
		\widehat{\cl F}^{x}_t
		:=
		(\Chi^x)^{-1}\cl F^x_t
		\)
		denotes the lifted completed canonical filtration, then
		\(
		\widehat{\cl F}^{x}_t\subset \cl G^x_t.
		\)
		By construction one has
		\((\Chi^x)^{-1} \cl F_t^x \subset \widehat{\cl F}_t^x \subset \cl G_t^x.\)
		
		We first prove (iii). 
		Fix an edge \(e=e_{\alpha,k}\), with $\alpha(e)=\alpha$, and a Borel set \(B\in\cl B(E)\). On the
		\(n\)-th segment, we call $\alpha_n$ the index corresponding to $\eta_n$ and define the local counting process of type-\(e\) jumps landing
		in \(B\) by
		\[
		\begin{aligned}
			\mathsf N^{e,B,n}_u
			:=
			\int_0^{u}
			\int_0^{\overline\Lambda}
			\int_0^1
			&\mathbf 1_{\{\alpha\}}(\alpha_n)\,
			\mathbf 1_{(\Gamma_{\alpha,k-1}(\eta_n(q-)),\,
				\Gamma_{\alpha,k}(\eta_n(q-))]}(a) 
			\mathbf 1_B\!\left(F_e(\eta_n(q-),b)\right)
			\,m_n(dq,da,db).
		\end{aligned}
		\]
		Conditionally on \(\cl G^x_{\tau_n}\), the integrand is predictable in the variables \((q,j)\)
		with respect to the local filtration generated by the \(n\)-th diffusion
		segment and Poisson measure, and is Borel
		in the mark variables \((a,b)\).
		Conditionally on \(\cl G_{\tau_n}^x\), the measure \(m_n\) is
		a fresh Poisson random measure with intensity \(dq\,da\,db\). Hence \(N^{e,B,n}\) has compensator
		\[
		\begin{aligned}
			\mathsf{C}^{e,B,n}_u:=&\int_0^{u}
			\mathbf 1_{\{\alpha\}}(\alpha_n)
			\int_{\Gamma_{\alpha,k-1}(\eta_n(q-))}^{\Gamma_{\alpha,k}(\eta_n(q-))}
			\int_0^1
			\mathbf 1_B\!\left(F_e(\eta_n(q-),b)\right)
			\,db\,da\,dq 
			=\int_0^{u}
			\lambda_e(\eta_n(q-))K_e(\eta_n(q-),B)\,dq,
		\end{aligned}
		\]
		where $\lambda_e$ is extended by $0$ outside $S_{\alpha(e)}$.
		The process
		\(\mathsf{JM}^{e,B,n}_u=\mathsf N^{e,B,n}_u-\mathsf C^{e,B,n}_u\)
		is a martingale, and so is the stopped process $\mathsf{JM}^{e,B,n}_{u \wedge \sigma_{n+1}}$.
		
		Now pass to physical time. Set
		\(
		u_n(t):=((t-\tau_n)^+)\wedge\sigma_{n+1},
		\)
		and define
		\(
		\mathsf{JM}^{e,B,x}_t
		:=
		\sum_{n\ge0}\mathsf{JM}^{e,B,n}_{u_n(t)}.
		\)
		Since \(\tau_n\to\infty\), only finitely many terms contribute on compact
		time intervals. Therefore
		\(
		\mathsf{JM}^{e,B,x}_t=\mathsf N^{e,B,x}_t
		-
		\int_0^t
		\lambda_e(\Chi^x_{r-})K_e(\Chi^x_{r-},B)\,dr
		\)
		is a \(\cl G^x\)-martingale.
		
		Taking \(B=S_{\beta(e)}\), and using that
		\(
		K_e(z,S_{\beta(e)})=1 \text{ for } z\in S_{\alpha(e)},
		\)
		we obtain the type-\(e\) counting process
		\(
		\mathsf N^{e,x}_t:=\mathsf N^{e,S_{\beta(e)},x}_t,
		\)
		with compensator
		\(
		\int_0^t \lambda_e(\Chi^x_{r-})\,dr.
		\)
		Thus jumps of type \(e\) have stochastic intensity
		\(
		\mathsf C^{e,x}_t:=\lambda_e(\Chi^x_{t-}).
		\)
		
		Finally, on the good set of the construction, every accepted jump has a unique
		edge type. Hence the process
		\(
		\mathsf N^{x}_t
		=
		\sum_{e\in\cl E}\mathsf N^{e,x}_t.
		\)
		has compensator
		\[
			\mathsf C^x_t:=\sum_{e\in\cl E}
			\int_0^t \lambda_e(\Chi^x_{r-})\,dr
			=
			\int_0^t
			\sum_{e:\alpha(e)=\alpha(\Chi^x_{r-})}
			\lambda_e(\Chi^x_{r-})\,dr
			=
			\int_0^t \Lambda(\Chi^x_{r-})\,dr,
		\]
		where the countable sum is justified by monotone convergence and by the uniform
		bound \(\Lambda\le\overline\Lambda\). Therefore the accepted jump process has
		stochastic intensity \(\Lambda(\Chi^x_{t-})\).
		
		The compensator formulas for \(\mathsf C^{e,B,x}\) and \(\mathsf C^{e,x}\) also show that, conditional on a
		type-\(e\) jump from a pre-jump state \(z\), the landing point has law
		\(K_e(z,\cdot)\). 
		This proves (iii) for the constructed process on the construction space.
		Since \(\P_x=(\Chi^x)_\#\Pi_x\), the corresponding jump mechanism is carried by
		the canonical process on \(D_E\) under \(\P_x\).
		
		It remains to verify (iv). For \(f\in\cl D_0\), define
		\(
		\mathscr M^{f,x}_t
		:=
		f(\Chi^x_t)-f(x)-\int_0^t Af(\Chi^x_r)\,dr.
		\)
		We first show that \(\mathscr M^{f,x}\) is a
		\((\Pi_x,\cl G^x_t)\)-martingale and then push the property forward through $\Chi^x$.
		
		To that purpose, fix $n \ge 0$. Conditionally on
		\(\cl G_{\tau_n}^x\), the stopped process
		\begin{equation}
			\label{eq:diffusion-local-martinale}
			\mathscr D_{u \wedge \sigma_{n+1}}^{f,n}
			:=
			f(\eta_n(u\wedge\sigma_{n+1}))
			-
			f(\eta_n(0))
			-
			\int_0^{u\wedge\sigma_{n+1}}
			L_{\alpha_n}f_{\alpha_n}(\eta_n(q))\,dq
		\end{equation}
		is a martingale in the local time variable \(u\).	
		For the jump part, define the accepted jumps process
		\[
		\begin{aligned}
			\mathscr N^{f,n}_{u}
			:=&
			\int_0^{u}\int_0^{\overline \Lambda}\int_0^1\mathbf 1_{\{a\le \Lambda(\eta_n(q-))\}}
			\left[
			f\!\left(F_{e(\eta_n(q-),a)}(\eta_n(q-),b)\right)
			-
			f(\eta_n(q-))
			\right]m_n(dq,da,db).
		\end{aligned}
		\]
		Again, conditional on $\cl G_{\tau_n}^x$, the integrand is predictable in the variables $(q,j)$ and Borel in $(a,b)$.	
		Its compensator is
		\(
		\mathscr C^{f,n}_{u}:=\int_0^{u}
		Jf(\eta_n(q))\,dq.
		\)		
		This implies that the process
		\(\mathscr{J}^{f,n}_u =\mathscr N^{f,n}_{u}-\mathscr C^{f,n}_{u}\)
		is a martingale. Now stopping at $\sigma_{n+1}$ we get only the first jump, that is
		\begin{equation}
			\label{eq:jump-local-martingale}
			\begin{aligned}
				\mathscr{J}^{f,n}_{u\wedge\sigma_{n+1}} &=\mathscr N^{f,n}_{u\wedge\sigma_{n+1}}-\mathscr C^{f,n}_{u\wedge\sigma_{n+1}}=\one_{\{\sigma_{n+1}\le u\}}\left[f(x_{n+1})-f(\eta_n(\sigma_{n+1}-))\right]-\int_0^{u\wedge \sigma_{n+1}}
				Jf(\eta_n(q))\,dq
			\end{aligned}
		\end{equation}
		is a martingale in the local time variable \(u\) conditional on $\cl G_{\tau_n}^x$. Putting \eqref{eq:diffusion-local-martinale} and \eqref{eq:jump-local-martingale} together, we obtain that, conditional on $\cl G_{\tau_n}^x$
		\[f(\eta_n(u\wedge\sigma_{n+1}))
		-
		f(\eta_n(0))
		+\one_{\{\sigma_{n+1}\le u\}}\left[f(x_{n+1})-f(\eta_n(\sigma_{n+1}-))\right]-\int_0^{u\wedge\sigma_{n+1}}
		Af(\eta_n(q))\,dq\]
		is a martingale in local time.		
		
		Passing to physical time through \(u_n(t):=((t-\tau_n)^+)\wedge\sigma_{n+1}\) and stopping at $\tau_N, \, N \ge 0$, telescopic sums give the stopped martingale
		\(\mathscr M_{t \wedge \tau_N}^{f,x}=f(\Chi_{t \wedge \tau_N}^x)-f(x)-\int_0^{t\wedge\tau_{N}}
		Af(\Chi_{r}^x)\,dr.\)
		Since
		\(\tau_N\to\infty\), \(\Pi_x\)-a.s., and since, for fixed
		\(T>0\),
		\(
		\sup_{0\le t\le T}
		|\mathscr M_{t\wedge\tau_N}^{f,x}|\le
		2\|f\|_\infty+T\|Af\|_\infty,
		\) dominated convergence allows us to establish that the full process
		\(\mathscr M_t^{f,x}\) is a \(\Pi_x,\cl G_t^x\)-martingale.

		Moreover \(\mathscr M^{f,x}_t\) is a functional of the physical path up to time \(t\), hence
		it is \(\widehat{\cl F}^x_t\)-measurable. Since $\widehat{\cl F}^{x}_s\subset \cl G^{x}_s,
		$ the tower property gives, for \(s\le t\),
		\(
		\mathbb E_{\Pi_x}
		[\mathscr M^{f,x}_t\mid \widehat{\cl F}^{x}_s]
		=
		\mathbb E_{\Pi_x}[\mathbb E_{\Pi_x}[
		\mathscr M^{f,x}_t\mid \cl G^{x}_s]
		\mid \widehat{\cl F}^{x}_s]
		=
		\mathscr M^{f,x}_s.
		\)
		Thus \(\mathscr M^{f,x}\) is also a
		\((\Pi_x,\widehat{\cl F}^{x}_t)\)-martingale.
		
		Finally, 
		since
		\(
		\mathscr M^{f,x}_t= M^f_t\circ\Chi^x
		\)
		and \(\P_x=(\Chi^x)_\#\Pi_x\), we have, for every \(A\in\cl F^x_s\),
		\[
		\mathbb E_{\P_x}\left[\mathbf 1_A M^f_t\right]
		=
		\mathbb E_{\Pi_x}
		\left[
		\mathbf 1_{(\Chi^x)^{-1}(A)}\mathscr M^{f,x}_t
		\right]
		=
		\mathbb E_{\Pi_x}
		\left[
		\mathbf 1_{(\Chi^x)^{-1}(A)}\mathscr M^{f,x}_s
		\right]
		=
		\mathbb E_{\P_x}\left[\mathbf 1_A M^f_s\right].
		\]
		Therefore \(M^f\) is a martingale under \(\P_x\) with respect to the
		completed natural filtration \((\cl F_t^x)_{t \ge 0}\). Hence \(\P_x\) solves the martingale problem for
		\((A,\cl D_0)\). This concludes the proof.
	\end{proof}
	
	\section{Strong-Markov property}
	\label{appsec:strong-markov}
	Here we gather all technical results and proofs needed for the proof of Theorem~\ref{thm:strong-markov}.

	\begin{lemma}
		\label{lem:restart-at-jump-times}
		Fix \(x\in E\). Let \((\tau_k)_{k\ge0}\) be the canonical jump times of \(X\),
		with \(\tau_0=0\). Then, for every \(k\ge0\) and every \(F \in B_b(D_E)\),
		\[
		\mathbb E_x
		\left[
		F\big((X_{\tau_k+s})_{s\ge0}\big)
		\,\middle|\,
		\cl F_{\tau_k}^x
		\right]
		=
		\int_{D_E}F(\omega)\,\P_{X_{\tau_k}}(d\omega)
		\qquad
		\P_x\text{-a.s. on }\{\tau_k<\infty\}.
		\]
	\end{lemma}
	
	\begin{proof}
		We prove the identity by lifting it to \((J,\cl J,\Pi_x)\) and then
		pushing it forward.		
		Continuity of the intra-stratum diffusions and \eqref{eq:no-jump-in-place} imply that canonical jump times of $X$, i.e. discontinuity times of the canonical path, coincide $\P_x$-a.s. with structural jump times of the construction in Theorem~\ref{thm:existence}.		
		Let $I_k$ be defined as in the proof of Proposition~\ref{prop:iii-iv}. We observe that
		\(
		\widehat{\cl F}_{\tau_k}^x
		\subset \cl G_{\tau_k}^x \subset
		\cl I_k,
		\text{ }
		\Pi_x\text{-modulo null sets}.
		\)
		Indeed, up to time \(\tau_k\), the construction has used only the first \(k\)
		inputs \(h_0,\ldots,h_{k-1}\), and the next input \(h_k\) is not revealed
		before the process restarts from
		\(
		\Chi^x_{\tau_k}=\Psi(h_{k-1}).
		\)
		
		By \eqref{eq:pi-x}, denoting $\theta_{\tau_k} j:=
		(h_k,h_{k+1},\ldots)$, we get
		\(
		\cl L_{\Pi_x}
		(\theta_{\tau_k}j
		\,|\,
		\cl I_k)=
		\Pi_{x_k}.
		\)		
		Applying the measurable construction map, we get
		\(
		(\Chi^x_{\tau_k+s}(j))_{s\ge0}
		=
		\Chi^{x_k}(\theta_{\tau_k} j)
		\text{ } \Pi_x\text{-a.s. on }\{\tau_k<\infty\}.
		\)
		Thus, for every bounded Borel \(F:D_E\to\mathbb R\),
		\begin{equation}
			\label{eq:cond-exp}
			\begin{aligned}
				\mathbb E_{\Pi_x}
				\left[
				F\big((\Chi^x_{\tau_k+s})_{s\ge0}\big)
				\,\middle|\,
				\cl I_k
				\right]
				&=
				\int_{\cl J}
				F(\Chi^{x_k}(y))
				\,\Pi_{x_k}(d y)                                      
				&=
				\int_{D_E}
				F(\omega)
				\,\P_{x_k}(d\omega)&=\int_{D_E}
				F(\omega)
				\,\P_{\Chi^x_{\tau_k}}(d\omega).
			\end{aligned}
		\end{equation}
		Moreover the right-hand side of \eqref{eq:cond-exp} is a Borel function of
		\(\Chi^x_{\tau_k}\), hence it is
		\(\widehat{\cl F}_{\tau_k}^x\)-measurable. Consequently, by the tower property, we get
		\(
		\mathbb E_{\Pi_x}[F\big((\Chi^x_{\tau_k+s})_{s\ge0}\big)
		\,|\,
		\widehat{\cl F}_{\tau_k}^x]
		=
		\int_{D_E}
		F(\omega)
		\,\P_{\Chi^x_{\tau_k}}(d\omega).
		\)		
		Finally we push this identity forward through \(\Chi^x\) to obtain the desired claim.
	\end{proof}

	\begin{lemma} 
		\label{lem:before-first-jump}
		Let \(S\) be a bounded \((\cl F_t^x)\)-stopping time
		and let \(A\in\cl F_S^x\) satisfy
		\(
		A\subset\{S<\tau_1\}.
		\)
		Then, for every bounded Borel functional \(F:D_E\to\mathbb R\),
		\[
		\mathbf 1_A
		\mathbb E_x
		\left[
		F\big((X_{S+r})_{r\ge0}\big)
		\,\middle|\,
		\cl F_S^x
		\right]
		=
		\mathbf 1_A
		\int_{D_E}F(\omega)\,\P_{X_S}(d\omega),\quad\P_x\text{-a.s.}
		\]
	\end{lemma}
	
	\begin{proof}
		Again, we prove the identity first on \((J,\cl J,\Pi_x)\) and then
		push it forward.		
		Define the lifted stopping time and event by
		\(
		\cl S(j):=S(\Chi^x(j))
		\text{ and }
		\cl A:=(\Chi^x)^{-1}(A).
		\)
		Then \(\cl S\) is a bounded stopping time for \(\widehat{\cl F}^x\), \(\cl A\) belongs to $\widehat{\cl F}_{\cl S}^x$, and
		\(
		\cl A\subset\{\cl S<\tau_1\}.
		\)		
		Define the shifted first input
		\(
		\theta_{\cl S}h_0
		:=
		(\eta_0(\cl S+\cdot),\,m_0^{\cl S}),
		\)
		where, for $r\ge0$ and measurable \(B\subset(0,\overline\Lambda]\times(0,1)\), we set
		\(
		m_0^{\cl S}((0,r]\times B)
		:=
		m_0((\cl S,\cl S+r]\times B).
		\)
		
		Now, we work on the event $\cl A$. We have
		\(
		\Chi^x_r=\eta_0(r),
		\text{ for }0\le r\le \cl S,
		\)
		and in particular
		\(
		\Chi^x_{\cl S}=\eta_0(\cl S).
		\) Moreover, by the strong Markov property of the stratumwise diffusions and of Poisson processes (\cite{zuyev2006strong}), the joint
		input \(h_0=(\eta_0,m_0)\) satisfies
		\(
		\cl L_{\Pi_x}
		(\theta_{\cl S}h_0
		\,|\,
		\cl G^x_{\cl S})
		=
		\bb K(\Chi^x_{\cl S};\cdot ).
		\)
		Since \(\cl S<\tau_1\), the first accepted atom
		after time \(\cl S\) for the original input \(h_0\) is exactly the first
		accepted atom of the shifted input \(\theta_{\cl S}h_0\). Hence
		\(
		\Psi(h_0)
		=
		\Psi(\theta_{\cl S}h_0).
		\)
		Therefore, denoting \(\theta_{\cl S} j=(\theta_S h_0,h_1,h_2,\dots)\), by \eqref{eq:pi-x}, we have 		
		\(\cl L_{\Pi_x}(\theta_{\cl S}j
		\,|\,
		\cl G^x_{\cl S})
		=
		\Pi_{\Chi^x_{\cl S}} \text{ on } \cl A.
		\)
		
		Now the argument is concluded with the same strategy as in the proof of Lemma~\ref{lem:restart-at-jump-times}: apply the measurable construction map, take conditional expectations, use the tower property, and push forward through $\Chi^x$ to get the desired identity.
	\end{proof}
	
	We now give the details of the proof of Theorem~\ref{thm:strong-markov} using Lemmas \ref{lem:restart-at-jump-times} and \ref{lem:before-first-jump}.
	
	\begin{proofof}{Theorem~\ref{thm:strong-markov}}
			We first consider a bounded stopping time $T$.
			Since the jump times satisfy \(\tau_k\to\infty\) \(\P_x\)-a.s.,
			the events
			\(
			A_k:=\{\tau_k\le T<\tau_{k+1}\},
			\  k\ge0,
			\)
			form a countable \(\cl F_T^x\)-measurable partition of the
			probability space.
			
			Fix \(k\ge0\). Define the shifted process after the \(k\)-th jump by
			\(
			X^{(k)}_s:=X_{\tau_k+s},
			\  s\ge0,
			\)
			and its shifted filtration by
			\(
			\cl F^{x,(k)}_s:=\cl F^x_{\tau_k+s}.
			\)
			Also set
			\(
			\rho_k:=\tau_{k+1}-\tau_k,
			\)
			which is the first jump time of the shifted process \(X^{(k)}\), and
			\(
			S_k:=(T-\tau_k)^+.
			\)
			Then \(S_k\) is a bounded \((\cl F^{x,(k)}_s)\)-stopping time and, on
			the event \(A_k\), one has
			\(
			S_k=T-\tau_k<\rho_k,
			\) 
			\(
			X^{(k)}_{S_k}=X_T \text{ and }
			(X^{(k)}_{S_k+r})_{r\ge0}
			=
			(X_{T+r})_{r\ge0}.
			\)

			By Lemma~\ref{lem:restart-at-jump-times}, conditionally on \(\cl F^x_{\tau_k}\), the shifted
			process \(X^{(k)}\) has law \(\P_{X_{\tau_k}}\). Therefore, applying
			Lemma~\ref{lem:before-first-jump} to the shifted process \(X^{(k)}\), at the stopping time \(S_k\),
			on the event \(A_k\subset\{S_k<\rho_k\}\), we obtain
			\(
			\mathbf 1_{A_k}
			\mathbb E_x[F\bigl((X_{T+r})_{r\ge0}\bigr)\,|\,\cl F^{x,(k)}_{S_k}]
			=
			\mathbf 1_{A_k}
			\int_{D_E}F(\omega)\,\P_{X_T}(d\omega).
			\)
			
			On \(A_k\), the stopped \(\sigma\)-fields
			\(\cl F_T^x\) and \(\cl F^{x,(k)}_{S_k}\) coincide locally, that is
			\(
			\{A_k\cap B:B\in\cl F_T^x\}
			=
			\{A_k\cap B:B\in\cl F^{x,(k)}_{S_k}\}.
			\)
			Consequently,
			\(
			\mathbf 1_{A_k}
			\mathbb E_x
			[F\bigl((X_{T+r})_{r\ge0}\bigr)\,|\,\cl F_T^x]
			=
			\mathbf 1_{A_k}
			\int_{D_E}F(\omega)\,\P_{X_T}(d\omega).
			\)
			
			Finally, summing over \(k\ge0\), we obtain
			\begin{equation}
				\label{eq:bounded-st-time}
				\mathbb E_x
				\left[
				F\bigl((X_{T+r})_{r\ge0}\bigr)
				\,\middle|\,
				\cl F_T^x
				\right]
				=
				\int_{D_E}F(\omega)\,\P_{X_T}(d\omega).
			\end{equation}
			We now pass to finite stopping times. Let \(T\) be finite and define
			\(
			T_n:=T\wedge n.
			\)
			Each \(T_n\) is bounded, so \eqref{eq:bounded-st-time} applies.
			Fix \(B\in\cl F_T^x\). For every \(n\), the set
			\(
			B_n:=B\cap\{T\le n\}
			\)
			belongs to \(\cl F_{T_n}^x\). Moreover, on \(\{T\le n\}\), one has
			\(T_n=T\), and therefore
			\(
			F((X_{T_n+r})_{r\ge0})=F((X_{T+r})_{r\ge0}) \text{ and }
			\int_{D_E}F(\omega)\,\P_{X_{T_n}}(d\omega)=\int_{D_E}F(\omega)\,P_{X_{T}}(d\omega).
			\)
			Applying \eqref{eq:bounded-st-time} at \(T_n\) and testing
			against \(\mathbf 1_{B_n}\), we get
			\(
			\mathbb E_x[\mathbf 1_{B_n}F((X_{T+r})_{r\ge0})]
			=
			\mathbb E_x[\mathbf 1_{B_n}\int_{D_E}F(\omega)\,P_{X_{T}}(d\omega)].
			\)
			Since \(T<\infty\) \(\bb P_x\)-a.s., one gets
			\(
			B_n\to B , \ \bb P_x\text{-a.s.}
			\)
			and, because \(F\) is bounded, both sides converge by dominated
			convergence.
			Since $B \in \cl F_T^x$ was arbitrary, this proves the claim.
	\end{proofof}

	\section{Uniqueness}
	\label{appsec:uniqueness}
	
	\begin{proofof}{Lemma~\ref{lem:direct-sum-feller}}
		For \(f=(f_\alpha)_{\alpha \in I} \in C_0(\widetilde E)\) we have
		\(
		\|\widetilde P_t f\|_\infty
		=
		\sup_{\alpha\in I}\|P_t^\alpha f_\alpha\|_\infty
		\leq
		\sup_{\alpha\in I}\|f_\alpha\|_\infty
		=
		\|f\|_\infty.
		\)
		Thus \(\widetilde P_t\) is a contraction.		
		Moreover, \(\widetilde P_t f\in C_0(\widetilde E)\). Indeed, for every $\alpha \in I$, $P_t^\alpha f_\alpha \in C_0(S_\alpha)$ and \(
		\|P_t^\alpha f_\alpha\|_\infty
		\leq
		\|f_\alpha\|_\infty,
		\) so, for every
		\(\varepsilon>0\),
		the set of \(\alpha\)'s for which
		\(\|P_t^\alpha f_\alpha\|_\infty\geq\varepsilon\) is contained in the finite set of
		\(\alpha\)'s for which \(\|f_\alpha\|_\infty\geq\varepsilon\).
		
		The semigroup property follows componentwise from the semigroup property of each
		\((P_t^\alpha)_{t\geq0}\).
		
		It remains to prove strong continuity. Fix \(f\in C_0(\widetilde E)\) and
		\(\varepsilon>0\). Choose a finite set \(F\subset I\) such that
		\(
		\sup_{\alpha\notin F}\|f_\alpha\|_\infty<\varepsilon.
		\)
		Then
		\[
		\begin{aligned}
			\|\widetilde P_t f-f\|_\infty
			&=
			\sup_{\alpha\in I}
			\|P_t^\alpha f_\alpha-f_\alpha\|_\infty  
			\leq
			\max_{\alpha\in F}
			\|P_t^\alpha f_\alpha-f_\alpha\|_\infty
			\vee
			\sup_{\alpha\notin F}
			\|P_t^\alpha f_\alpha-f_\alpha\|_\infty .
		\end{aligned}
		\]
		For \(\alpha\notin F\),
		\(
		\|P_t^\alpha f_\alpha-f_\alpha\|_\infty
		<2\varepsilon,
		\)
		while for the finitely many \(\alpha\in F\), strong continuity of \(P_t^\alpha\) gives
		\(
		\max_{\alpha\in F}
		\|P_t^\alpha f_\alpha-f_\alpha\|_\infty
		\to0
		\text{ as }t\to0.
		\)
		Since $\varepsilon$ was arbitrary, we get
		\(
		\|\widetilde P_t f-f\|_\infty\to0
		\text{ as }t\to0.
		\)
		Thus \((\widetilde P_t)_{t\geq0}\) is a strongly continuous contraction semigroup on
		\(C_0(\widetilde E)\).
		
		We now give a description of the generator $(\widetilde{L}, D(\widetilde{L}))$ of $(\widetilde{P}_t)_{t \ge 0}$.		
		If \(f\in D(\widetilde L)\), then for each \(\alpha\),
		\[
		\frac{P_t^\alpha f_\alpha-f_\alpha}{t}
		=
		\left.
		\frac{\widetilde P_t f-f}{t}
		\right|_{S_\alpha}
		\to
		(\widetilde L f)|_{S_\alpha}, \qquad \text{in }C_0(S_\alpha).
		\]
		Hence \(f_\alpha\in D(L_\alpha)\) and
		\(
		(\widetilde Lf)|_{S_\alpha}
		=
		L_\alpha f_\alpha.
		\)
		Since \(\widetilde Lf\in C_0(\widetilde E)\), the family
		\((L_\alpha f_\alpha)_\alpha\) belongs to \(C_0(\widetilde E)\).
		This gives the inclusion
		\(
		D(\widetilde L)
		\subset
		\{
		f\in C_0(\widetilde E):
		f_\alpha\in D( L_\alpha)\ \forall\alpha,
		\
		\bigl( L_\alpha f_\alpha\bigr)_\alpha
		\in C_0(\widetilde E)
		\}.
		\)
		
		Conversely, suppose that
		\(
		f_\alpha\in D( L_\alpha)
		\text{ for all }\alpha \text{ and }
		g:=(g_\alpha)_\alpha=( L_\alpha f_\alpha)_\alpha\in C_0(\widetilde E).
		\)
		For each \(\alpha\),
		\(
		(P_t^\alpha f_\alpha-f_\alpha)/t
		\to
		g_\alpha
		\)
		in \(C_0(S_\alpha)\). To upgrade this to convergence in the $C_0(\widetilde{E})$ norm,
		fix \(\varepsilon>0\). Choose a finite \(F\subset I\) such that
		\(	\sup_{\alpha\notin F}\|f_\alpha\|_\infty<\varepsilon \text{ and }
		\sup_{\alpha\notin F}\|g_\alpha\|_\infty<\varepsilon.
		\)
		Using the identity
		\(
		(P_t^\alpha f_\alpha-f_\alpha)/t
		=
		t^{-1}\int_0^t P_s^\alpha g_\alpha\,ds,
		\)
		valid for \(f_\alpha\in D( L_\alpha)\), we obtain, for \(\alpha\notin F\),
		\[
		\left\|
		\frac{P_t^\alpha f_\alpha-f_\alpha}{t}
		-
		g_\alpha
		\right\|_\infty
		\leq
		\frac1t\int_0^t \|P_s^\alpha g_\alpha\|_\infty\,ds
		+
		\|g_\alpha\|_\infty
		\leq
		2\|g_\alpha\|_\infty
		<2\varepsilon.
		\]
		On the finite set \(F\), convergence is uniform after taking the maximum over
		\(\alpha\in F\). Therefore \(\|(P_t^\alpha f_\alpha-f_\alpha)/t-g\|_\infty \to 0\).
		Thus \(f\in D(\widetilde L)\) and \(\widetilde Lf=g\), giving
		\(
		D(\widetilde L)
		=
		\{
		f\in C_0(\widetilde E):
		f_\alpha\in D( L_\alpha)\ \forall\alpha,
		\
		\bigl( L_\alpha f_\alpha\bigr)_\alpha
		\in C_0(\widetilde E)
		\}.
		\)
		
		Finally, we prove that \(\cl D_0\) is a core for \(\widetilde L\). Let
		\(f\in D(\widetilde L)\).
		Since \(f,\widetilde L f\in C_0(\widetilde E)\), for every \(\varepsilon>0\) there exists a finite
		set \(F\subset I\) such that
		\(
		\sup_{\alpha\notin F}\|f_\alpha\|_\infty<\varepsilon,
		\text{ and }
		\sup_{\alpha\notin F}\|L_\alpha f_\alpha\|_\infty<\varepsilon.
		\)
		For every \(\alpha\in F\), since \(\cl C_\alpha\) is a core for
		\( L_\alpha\), we can choose \(f^{(n)}_\alpha\in \cl C_\alpha\) such that
		\(
		\|f^{(n)}_\alpha-f_\alpha\|_\infty
		+
		\|L_\alpha f^{(n)}_\alpha- L_\alpha f_\alpha\|_\infty
		\to0.
		\)
		Define \(\overline f^{(n)}\in\cl D_0\) by
		\[
		\overline f^{(n)}_\alpha
		=
		\begin{cases}
			f^{(n)}_\alpha, & \alpha\in F,\\
			0, & \alpha\notin F.
		\end{cases}
		\]
		Then, since $\varepsilon$ was arbitrary, we get
		\(
		\|\overline f^{(n)}-f\|_\infty
		+
		\|\widetilde L \overline f^{(n)}-\widetilde L f\|_\infty
		\to0,
		\) proving the claim.
	\end{proofof}

	\section{Covariance application}
	\label{appsec:cov-application}

	\begin{proofof}{Lemma~\ref{lem:phi-diffeomorphism}}
		Let \(\pi_M:M_r \to \Gr(r,n)\) be the canonical projection and define \(\pi_S:S_r\to\Gr(r,n), \ \Sigma \mapsto \Imm(\Sigma)\). Then \(\phi_r\) is a bundle map over \(\Gr(r,n)\), because
		\(
		\Imm(Q\exp(H)Q^\top)=\Span(Q).
		\)
		
		We first work on local trivializations. Let \(U\subset\Gr(r,n)\) be an open set over which the Stiefel bundle admits a smooth local section \(R_U:U\to\St(n,r)\), so that \(\Span(R_U(P))=P\) for every \(P\in U\). If \([Q,H]\in\pi_M^{-1}(U)\) and \(P=\Span(Q)\), then there is a unique \(G\in O(r)\) such that \(Q=R_U(P)G\). In this trivialization,
		\(
		[Q,H]=[R_U(P)G,H]=[R_U(P),GHG^\top].
		\)
		Thus \(\pi_M^{-1}(U)\) is identified with \(U\times\Sym(r)\) by the map
		\(
		[Q,H]\mapsto (P,GHG^\top).
		\)

		Similarly, \(\pi_S^{-1}(U)\) is identified with \(U\times\Sym^+(r)\) by
		\(
		\Sigma\mapsto (P,R_U(P)^\top\Sigma R_U(P)),
		\text{ with } P=\Imm(\Sigma).
		\)
		The inverse of this trivialization is \((P,A)\mapsto R_U(P)AR_U(P)^\top\).
		
		In these two local trivializations, the map \(\phi_r\) takes the simple form
		\(
		(P,K)\mapsto (P,\exp(K)),
		\text{ with } P\in U,\ K\in\Sym(r).
		\)
		Denoting by $\log$ the principal matrix logarithm, the map \((P,K) \mapsto (P,\exp(K))\)  is a diffeomorphism with inverse 
		\begin{equation}
			\label{eq:local-inverse}
			(P,A)\mapsto (P,\log(A)).
		\end{equation}
		Hence \(\phi_r\) is smooth and locally a diffeomorphism. 
				
		We now check that these local inverses agree on overlaps. Let \(U,V\subset\Gr(r,n)\)
		be two trivializing open sets with $U \cap V \ne \emptyset$. The two local sections
		are related by a unique smooth map \(G_{UV}:U\cap V\to O(r)\) such that
		\(R_V=R_UG_{UV}\). 
		For \(\Sigma \in \pi_S^{-1}(U \cap V)\) and \(P \in U \cap V\), we have, by equivariance of the principal logarithm under orthogonal conjugation, 
		\(
		\log(R_V(P)^\top\Sigma R_V(P))=G_{UV}(P)^\top
		\log(R_U(P)^\top\Sigma R_U(P))
		G_{UV}(P).
		\)
		Thus, \[\begin{aligned}
			[R_V(P),\log(R_V(P)^\top\Sigma R_V(P))]= &[R_U(P)G_{UV}(P), G_{UV}(P)^\top\log(R_U(P)^\top\Sigma R_U(P))G_{UV}(P)]\\=&[R_U(P),\log(R_U(P)^\top\Sigma R_U(P))].
		\end{aligned}\]
		Therefore, the local inverses \eqref{eq:local-inverse} agree on overlaps and glue to a unique global
		smooth map \(\psi_r:S_r\to M_r\).
		Equivalently, this global map is given by
		\(
		\psi_r(\Sigma)=[Q,\log(Q^\top\Sigma Q)],
		\)
		where \(Q\in\St(n,r)\) is any orthonormal frame with
		\(\Span(Q)=\Imm(\Sigma)\). The preceding overlap calculation shows
		precisely that this expression is independent of the choice of \(Q\). Moreover, \(\phi_r \circ \psi_r =\id_{M_r}\) and \(\psi_r \circ \phi_r=\id_{S_r}.\) Thus \(\phi_r\) has a smooth global inverse and is, therefore, a diffeomorphism.
	\end{proofof}
	
	\begin{proofof}{Proposition~\ref{prop:cov-state-space}}
		First, $(\Cov(n),d)$ is a locally compact separable metric space.
		Second, each \(S_r\) is a connected smooth finite-dimensional manifold without
		boundary. For \(r=0\), this is trivial. For \(r\ge1\), this follows from the
		diffeomorphism
		\(
		S_r\simeq
		\St(n,r)\times_{O(r)}\Sym(r),
		\)
		because \(\Gr(r,n)\) is connected smooth manifold without boundary and so is the fiber
		\(\Sym(r)\).
		
		Moreover, \(S_r\) is a Borel subset of \(\Cov(n)\). Indeed, for
		\(0\le r\le n\), let
		\(
		\Cov_{\le r}(n)
		:=
		\{\Sigma\in\Cov(n):\rk(\Sigma)\le r\}.
		\)
		This set is closed, since it is described by the vanishing of all
		\((r+1)\times(r+1)\) minors. Therefore
		\(
		S_r
		=
		\Cov_{\le r}(n)
		\setminus
		\Cov_{\le r-1}(n)
		\)
		is Borel.
		
		Finally, since \(\Cov(n)\) carries the topology induced from \(\Sym(n)\), the subspace
		topology on \(S_r\) inherited from \(\Cov(n)\) agrees with the subspace topology
		inherited from \(\Sym(n)\). Therefore the inclusion \(\iota_r:=S_r \hookrightarrow \Cov(n)\) is a homeomorphism onto its image.
	\end{proofof}
	
	\begin{proofof}{Proposition~\ref{prop:cov-intrastratum-diffusions}}	
		For \(r=0\) the stratum is the singleton \(S_0=\{0\}\), and the diffusion is constant. Hence Assumption \assref{ass:diffusions} is trivially satisfied
		
		Fix now \(1\le r\le n\).		
		The \(Q\)-equation \eqref{eq:q-sde} is a Stratonovich equation with smooth vector fields on the Stiefel
		manifold. Since \(\St(n,r)\) is compact,
		there is no explosion in the \(Q\)-component. The \(H\)-component driven by \eqref{eq:h-sde} is the Euclidean
		Ornstein--Uhlenbeck process
		and is therefore defined for all \(t\ge 0\). Thus \((Q_t,H_t)\) is a conservative
		continuous diffusion on \(\St(n,r)\times \Sym(r)\). Therefore $\Sigma_t=Q_t \exp(H_t) Q_t^\top \in S_r$ for all $t \ge 0$ and, since $\phi_r \circ q_r$ is smooth, the paths $t \mapsto \Sigma_t$ are continuous. This proves \ref{it:D2}.
		
		A short calculation gives the It\^o form of \eqref{eq:q-sde} in $\R^{n\times r}$, that is
		\begin{equation}
			\label{eq:q-ito-form}
			dQ_t=(I_n-Q_tQ_t^\top)dB_t -\frac{n-r}{2}Q_t dt.
		\end{equation}
		The solution to the SDE remains on $\St(n,r)$ and the map $\St(n,r) \ni Q \mapsto I_n-Q Q^\top$ is Lipschitz there. Indeed, for $Q ,R \in \St(n,r)$
		\begin{equation*}
			\begin{aligned}
				\|QQ^\top -RR^\top\|_F&=\|(Q-R)Q^\top+R(Q^\top-R^\top)\|_F \le \|Q-R\|_F \|Q^\top\|_{\mathrm{op}}+\|Q^\top-R^\top\|_F \|R\|_{\mathrm{op}}=2\|Q-R\|_F,
			\end{aligned}
		\end{equation*}
		where $\|\cdot\|_{\mathrm{op}}$ denotes the operator norm.
		It follows that standard estimates that use Burkholder--Davis--Gundy (BDG) inequality and Gronwall lemma (see for example the proof of \cite[Prop. 5.2.9]{karatzas2014brownian}) give, for any $T <+\infty$ and $Q^{(n)} \to Q$ in $\St(n,r)$
		\begin{equation}
			\label{eq:estimate-q}
			\E\left[\sup_{0\le t \le T} \|Q_t^{(n)}-Q_t\|_F\right]\le C_T\|Q^{(n)}-Q\|_F,
		\end{equation}
		where $C_T$ is a constant depending on $T$. 
		
		For the \(H\)-component, we write the explicit Ornstein--Uhlenbeck representation
		\begin{equation}
			\label{eq:ou-solution}
			H_t=e^{-\theta_rt}H+Z_t, \quad Z_t:=\sigma_r\int_0^t e^{-\theta_r(t-s)}\,dW_s
		\end{equation}
		Hence, for any $H^{(n)}\to H$, if the solutions are driven by the same Brownian motion
		\(
		H_t^{(n)}-H_t=e^{-\theta_rt}(H^{(n)}-H),
		\)
		and therefore
		\begin{equation}
			\label{eq:estimate-h}
			\sup_{0\le s\le T}\|H_s^{(n)}-H_s\|_F
			\le
			\|H^{(n)}-H\|_F.
		\end{equation}
		It follows that, for every \(T<\infty\),
		the law \(\widetilde{\bb Q}_{(Q,H)}^r\) of \((Q_t,H_t)_{0\le t\le T}\) depends weakly continuously on the initial condition
		\((Q,H)\). Since the quotient map $q_r$
		is smooth and the construction is \(O(r)\)-equivariant, the same weak continuity
		holds for \([Q,H]\mapsto \overline{\mathbb Q}^r_{[Q,H]}:=\cl L_{[Q,H]}([Q_t,H_t])\).
		
		Now define, for \(\Sigma\in S_r\),
		\(
		\bb Q^r_\Sigma
		:=
		\overline{\mathbb Q}^r_{\phi_r^{-1}(\Sigma)}
		\circ
		\phi_r^{-1}.
		\)
		Because \(\phi_r\) is a diffeomorphism the map
		\(
		\Sigma\mapsto \bb Q^r_\Sigma
		\)
		is weakly continuous from \(S_r\) into \(\cl P(C([0,\infty);S_r))\). Since
		\(C([0,\infty);S_r)\) is a Borel subset of \(D([0,\infty);S_r)\), this also gives Borel
		measurability as a map into \(\cl P(D([0,\infty);S_r))\). This proves \ref{it:D1}.

		We now show that \((\widetilde P_t^r)_{t\ge0}\) is a strongly continuous contraction
		semigroup on \(C_0(\St(n,r)\times \Sym(r))\).		
		The contraction property is immediate:
		\[
		|\widetilde P_t^r\tilde f|=|\mathbb E_{(Q,H)}\bigl[\tilde f(Q_t,H_t)\bigr]|
		\le \mathbb E_{(Q,H)}\bigl[|\tilde f(Q_t,H_t)|\bigr]\le
		\|\tilde f\|_\infty.
		\]
		The semigroup property follows from the Markov property of the lifted SDE.
		
		Let \(\tilde f\in C_0(\St(n,r) \times \Sym(r))\). We first prove that
		\(\widetilde P_t^r\tilde f\) is continuous. Let
		\(
		(Q^{(n)},H^{n})\to(Q,H)
		\) in \(\St(n,r)\times \Sym(r)\). Estimates \eqref{eq:estimate-q} and \eqref{eq:estimate-h} hold.
		
		Thus \((Q_t^{(n)},H_t^{(n)})\to(Q_t,H_t)\) in probability for each fixed \(t\). Since
		\(\tilde f\) is bounded and continuous,
		\(
		\tilde f(Q_t^{(n)},H_t^{(n)})
		\to
		\tilde f(Q_t,H_t)
		\)
		in probability and the random variables are uniformly bounded by
		\(\|\tilde f\|_\infty\). Hence
		\(
		\widetilde P_t^r\tilde f(Q^{(n)},H^{(n)})
		\to
		\widetilde P_t^r\tilde f(Q,H),
		\)
		so \(\widetilde P_t^r\tilde f\) is continuous.
		
		We now prove that \(\widetilde P_t^r\tilde f\) vanishes at infinity. Since
		\(\St(n,r)\) is compact, a sequence \((Q^{n},H^{n})\) goes to infinity in
		\(\St(n,r) \times \Sym(r)\) if and only if \(\|H^{(n)}\|_F\to\infty\). Fix \(\varepsilon>0\). Since
		\(\tilde f\in C_0(\St(n,r)\times \Sym(r))\), there exists \(R>0\) such that
		\(
		|\tilde f(Q,H)|\le\varepsilon
		\text{ whenever } \|H\|_F\ge R.
		\)
		By \eqref{eq:ou-solution},
		for \(t>0\),
		\(
		\P_H(\|H_t\|_F\le R)
		=
		\P(\|e^{-\theta_rt}H+Z_t\|_F\le R)
		\to 0
		\)
		as \(\|H\|_F\to\infty\). Consequently
		\begin{equation}
			\label{eq:prob-pt-estimate}
			\begin{aligned}
				|\widetilde P_t^r\tilde f(Q,H)|&\le \E_{(Q,H)}\left[|\tilde f(Q_t,H_t)|\one_{\{\|H_t \|_F >R\}}\right]+\E_{(Q,H)}\left[|\tilde f(Q_t,H_t)|\one_{\{\|H_t \|_F \le R\}}\right]\\
				&\le
				\varepsilon
				+
				\|\tilde f\|_\infty
				\P_H(\|H_t\|_F\le R)\to \varepsilon, \qquad \|H\|_F\to\infty.
			\end{aligned}
		\end{equation}
		Since \(\varepsilon\)
		is arbitrary, \(\widetilde P_t^r\tilde f\in C_0(\St(n,r) \times \Sym(r))\).
		
		It remains to prove strong continuity. Fix \(\tilde f\in C_0(\St(n,r) \times \Sym(r))\) and
		\(\varepsilon>0\). Choose \(R>0\) such that
		\(
		|\tilde f(Q,H)|\le \varepsilon
		\text{ for }\|H\|_F\ge R.
		\)
		On the compact set
		\(
		K:=\St(n,r)\times\{H:\|H\|_F\le R\},
		\)
		the function \(\tilde f\) is uniformly continuous. 
		For the \(Q\)-component,
		since \(Q_t\in \St(n,r)\), the coefficients of the Itô form \eqref{eq:q-ito-form} are uniformly bounded. More precisely,
		\(
		\|(I_n-QQ^\top)\|_{\mathrm{op}}\le 1,
		\text{ and }
		\|Q\|_F^2=r.
		\)
		Hence, by the BDG inequality, denoting $c:=(n-r)/2$,
		\[
		\begin{aligned}
			\E_Q\left[
			\sup_{0\le s\le t}\|Q_s-Q\|_F^2
			\right]
			&\le 2\left[\E_Q\left[\sup_{0\le s\le t}\left\| \int_0^s \Pi_{Q_u}^\perp dB_u\right\|_F^2\right]+c^2\E_Q\left[\sup_{0\le s\le t}\left\|\int_0^sQ_u \,du\right\|_F^2\right]\right]\\&\le C \E_Q \left\langle\int_0^\cdot \Pi_{Q_u}^\perp dB_u\right\rangle_t +rt^2\le C nrt+rt^2 \le C_{n,r}(t+t^2),
		\end{aligned}
		\]
		where $\langle \cdot \rangle_t$ denotes the quadratic variation.
		Therefore, for every \(\delta>0\),
		\begin{equation}
			\label{eq:q-estimate}
		\sup_{Q\in\St(n,r)}
		\P_Q\left(
		\sup_{0\le s\le t}\|Q_s-Q\|_F>\delta
		\right)
		\le
		\frac{C_{n,r}(t+t^2)}{\delta^2}
		\to 0, \qquad t\to0.
		\end{equation}

		For the $H$-component, by \eqref{eq:ou-solution}, we get, uniformly for \(\|H\|_F\le R\),
		\[
		\mathbb E_H\|H_t-H\|_F^2
		\le
		2(1-e^{-\theta_rt})^2R^2
		+
		\frac{\sigma_r^2d_r}{\theta_r}
		(1-e^{-2\theta_rt}),
		\qquad d_r=\dim\Sym(r).
		\]
		Thus, for every \(\delta>0\),
		\begin{equation}
			\label{eq:h-estimate}
		\sup_{\|H\|_F\le R}
		\P_H(\|H_t-H\|_F>\delta)
		\to0.
		\end{equation}
		Combining \eqref{eq:q-estimate} and \eqref{eq:h-estimate}, we get
		\(
		\sup_{(Q,H)\in K}
		\P_{(Q,H)}
		(
		\|(Q_t,H_t)-(Q,H)\|_F>\delta
		)
		\to0
		\text{ as } t\to0,
		\)
		for every compact set
		\(
		K\subset \St(n,r)\times\Sym(r).
		\)		
		Hence
		\(
		\sup_{(Q,H)\in K}
		|\widetilde P_t^r\tilde f(Q,H)-\tilde f(Q,H)|
		\to 0\text{ as }t\to0.
		\)
		On the complement of a sufficiently large compact set, both
		\(\tilde f(Q,H)\) and, by \eqref{eq:prob-pt-estimate}, \(\widetilde P_t^r\tilde f(Q,H)\) are uniformly small for
		all sufficiently small \(t\). 
		Therefore
		\(
		\|\widetilde P_t^r\tilde f-\tilde f\|_\infty
		\to 0
		\text{ as }t\to0.
		\)
		Thus \((\widetilde P_t^r)_{t\ge0}\) is a strongly continuous contraction semigroup on
		\(C_0(\St(n,r)\times \Sym(r))\).
		
		We now descend to the quotient
		\(M_r.\) If \(f\in C_0(M_r)\), then
		\(\tilde f:=f\circ q_r\) is an \(O(r)\)-invariant element of \(C_0(\St(n,r)\times \Sym(r))\).
		Now, \eqref{eq:invariance} gives $O(r)$-equivariance of $\widetilde P_t^r$, i.e., for $g \in B_b(\St(n,r)\times \Sym(r))$,
		\[
		\begin{aligned}
			\widetilde P_t (g \circ R_G)(Q,H)=\E_{Q,H}\left[g(R_G(Q_t,H_t))\right]=\E_{R_G(Q,H)}\left[g(Q_t,H_t)\right]=(\widetilde P_t g) \circ  R_G(Q,H).
		\end{aligned}\]
		Hence \(\widetilde P_t^r\) preserves invariant functions and thus, for $f \in B_b(M_r)$, there exists a unique
		function \(P_t^{M_r}f\) on \(M_r\) such that
		\(
		(P_t^{M_r}f)\circ q_r
		=
		\widetilde P_t^r(f\circ q_r).
		\)
		
		Since $O(r)$ is compact and acts smoothly on \(\St(n,r)\times\Sym(r)\), \(q_r\) is proper.
		Therefore, if \(\widetilde P_t^r(f\circ q_r)\in C_0(\St(n,r)\times \Sym(r))\), then
		\(P_t^{M_r}f\in C_0(M_r)\). Moreover,
		\(
		\|P_t^{M_r}f\|_\infty=\|\widetilde P_t^r(f \circ q_r)\|_\infty \le \|f \circ q_r\|_{\infty}
		=
		\|f\|_\infty,
		\)
		and
		\(\|P_t^{M_r}f-f\|_\infty
			=
			\|\widetilde P_t^r(f\circ q_r)-f\circ q_r\|_\infty
			\to0.
		\)
		Thus \((P_t^{M_r})_{t\ge0}\) is a strongly continuous contraction semigroup on
		\(C_0(M_r)\).
		
		Finally, transport the semigroup to \(S_r\) through the diffeomorphism \(\phi_r:M_r \to S_r\).
		For \(g\in C_0(S_r)\), define
		\(
		P_t^r g(\Sigma)
		:=
		P_t^{M_r}(g\circ\phi_r)(\phi_r^{-1}(\Sigma)).
		\)
		Since \(\phi_r\) is a homeomorphism, \(g\circ\phi_r\in C_0(M_r)\) and
		\(\|g\circ\phi_r\|_\infty=\|g\|_\infty\). Hence
		\(
		P_t^r C_0(S_r)\subset C_0(S_r),
		\ 
		\|P_t^r g\|_\infty\le \|g\|_\infty,
		\)
		and
		\(
		\|P_t^r g-g\|_\infty
		=
		\|P_t^{M_r}(g\circ\phi_r)-g\circ\phi_r\|_\infty
		\to0.
		\)
		Therefore \((P_t^r)_{t\ge0}\) is a strongly continuous contraction semigroup on
		\(C_0(S_r)\). This proves \ref{it:D3}.
		
		It remains to verify \ref{it:D4}. Again, we work first on the product space \(\St(n,r)\times \Sym(r)\). Set $\cl C_r=C_c^\infty(S_r)$. Since $S_r$ is a smooth second-countable finite-dimensional manifold, it is in particular a locally compact separable metric space. Hence $C_c^\infty(S_r)$, which is uniformly dense in $C_0(S_r)$, is measure-determining.

		We first show that $C^\infty(\St(n,r))$ is a core for $L_{Q}$. As observed in \eqref{eq:c2-subset-domain}, $C^\infty(\St(n,r)) \subset D(L_Q)$.
		Furthermore, we can show that \(C^\infty(\St(n,r))\) is invariant under \(P_t^{\St}\). To do that, we work with the It\^o form \eqref{eq:q-ito-form}. Since the vector fields
		\(X_{ij}\) are smooth on the compact manifold \(\St(n,r)\), they can be extended to vector fields $\widetilde X_{ij} \in C_b^{\infty}(\R^{n \times r}; \R^{n \times r})$. Similarly, we extend the drift as a map $\tilde b  \in C_b^{\infty}(\R^{n \times r}; \R^{n \times r})$ such that \(\tilde b(Q)=\frac{n-r}{2} Q\) for \(Q \in \St(n,r)\).
		
		Now in view of \cite[Thm. 4.6.5]{kunita1990stochastic}, the $Q$-SDE generates a smooth
		stochastic flow of $C^k$-diffeomorphisms for every $ k \ge 1$. In particular the map $Q \mapsto Q_t(\omega)$ is smooth almost surely and its derivatives have finite moments on compact time intervals. Thus, for \(f\in C^\infty(\St(n,r))\), $Q \mapsto f(Q_t(\omega))$ is smooth almost surely and differentiation under the expectation is justified. It follows that
		\(
		P_t^{\St}f(Q)
		\)
		is again smooth in \(Q\). Therefore
		\(
		P_t^{\St} C^\infty(\St(n,r))\subset C^\infty(\St(n,r)).
		\)
		
		Finally, \(C^\infty(\St(n,r))\) is dense in \(C(\St(n,r))\), because $\St(n,r)$ is compact. By \cite[Prop. II.1.7]{engel2000one}, any dense subspace of the generator domain which is invariant under the
		semigroup is a core. Hence \(C^\infty(\St(n,r))\) is a core for \(L_Q\).

		Next, since \(\Sym(r)\) is a finite-dimensional Euclidean vector space, the
		Ornstein--Uhlenbeck semigroup \(P_t^{\Sym}\) is the Mehler semigroup which is described in \cite{lunardi2020ornstein}. In particular, we have
		\(
		\cl S(\Sym(r))\subset D(L_H), \ 
		L_H\cl S(\Sym(r))\subset \cl S(\Sym(r))
		\), and \(P_t^{\Sym}\cl S(\Sym(r))\subset \cl S(\Sym(r))\).	Moreover \(\cl S(\Sym(r))\) is dense in \(C_0(\Sym(r))\). Hence, by \cite[Prop. II.1.7]{engel2000one}, \(\cl S(\Sym(r))\) is a core for \(L_H\).
		
		Define the algebraic tensor product
		\(
		\cl D_\otimes
		:=
		C^\infty(\St(n,r))\otimes \cl S (\Sym(r)),
		\)
		consisting of finite sums
		\[
		u(Q,H)=\sum_{j=1}^N f_j(Q) g_j(H),
		\qquad
		f_j\in C^\infty(\St(n,r)),\quad
		g_j\in\cl S (\Sym(r)).
		\]
		The space \(\cl D_\otimes\) is dense in \(C_0(\St(n,r)\times \Sym(r))\), because
		\(C^\infty(\St(n,r))\) is dense in \(C(\St(n,r))\), \(\cl S (\Sym(r))\) is dense in
		\(C_0(\Sym(r))\), and finite sums of separated functions are dense in
		\(C_0(\St(n,r)\times \Sym(r))\).
		
		Furthermore, \(\cl D_\otimes\subset D(\widetilde L_r)\). Indeed, for a simple
		tensor \(u=f\otimes g\),
		\[
		\frac{\widetilde P_t^r u-u}{t}
		=
		\frac{P_t^{\St}f-f}{t}\otimes P_t^{\Sym} g
		+
		f\otimes\frac{P_t^{\Sym} g- g}{t}.
		\]
		As \(t\to0\), the first term converges uniformly to
		\(
		(L_Qf)\otimes g,
		\)
		and the second term converges uniformly to
		\(
		f\otimes(L_H g).
		\)
		Therefore
		\(
		(\widetilde P_t^r u-u)/t \to 
		\widetilde L_r(f\otimes g).
		\)
		By linearity, the same conclusion holds for every element of \(\cl D_\otimes\).
		
		Moreover \(\widetilde P_t^r(f\otimes g)
		=
		(P_t^{\St}f)\otimes(P_t^{\Sym} g)\) gives \(\widetilde P_t^r\cl D_\otimes\subset \cl D_\otimes .\)
		\cite[Prop. II.1.7]{engel2000one} now implies that
		\(
		\cl D_\otimes
		\)
		is a core for the \(C_0(\St(n,r) \times \Sym(r))\)-generator \(\widetilde L_r\).

		Now, choose \(\chi\in C_c^\infty(\Sym(r))\) such that \(0\le\chi\le1\), \(\chi=1\) on
		\(\{\|H\|_F\le1\}\), and \(\chi=0\) on \(\{\|H\|_F\ge2\}\). Set
		\(
		\chi_m(H):=\chi(H/m)
		\)
		and, for $u \in \cl D_\otimes$,
		\(
		u_m(Q,H):=\chi_m(H) u(Q,H) \in \cl C^\infty(\St(n,r))\otimes C_c^\infty(\Sym(r)).
		\)
		Then 
		\(
		\|u_m-u\|_\infty\to0 \text{ as }m \to +\infty.
		\)
		Since $\chi_m$ is a function of $H$ only, we have
		\(
		L_Qu_m=\chi_m L_Q u\).
		For the $H$-part
		a brief computation gives
		\[
		L_H(u_m)
		=
		\chi_mL_Hu
		+
		\frac{\sigma_r^2}{2}
		\bigl(
		2\langle\nabla_H\chi_m,\nabla_H u\rangle_F
		+
		u\,\Delta_H\chi_m
		\bigr)
		-
		\theta_ru\,\langle H,\nabla_H\chi_m\rangle_F .
		\]
		Therefore
		\[
		\widetilde L_r(u_m)-\widetilde L_r(u)
		=
		(\chi_m-1)\widetilde L_r u
		+
		\frac{\sigma_r^2}{2}
		\bigl(
		2\langle\nabla_H\chi_m,\nabla_H u\rangle_F
		+
		u\,\Delta_H\chi_m
		\bigr)
		-
		\theta_ru\,\langle H,\nabla_H\chi_m\rangle_F .
		\]
		The first term tends to zero uniformly because \(\widetilde L_r u\in C_0(\St(n,r)\times \Sym(r))\). The second
		and third terms are supported in the annulus
		\(
		\{m\le \|H\|_F\le 2m\}.
		\)
		On this annulus,
		\(
		\|\nabla_H\chi_m\| \le C_1/m,\ 
		|\Delta_H\chi_m|\le C_2/m^2\text{ and }
		\|H\|_F \le 2 m.
		\)
		These give, as $m \to \infty$
		\begin{gather*}
			|\langle \nabla_H \chi_m, \nabla_H u \rangle_F| \le \frac{C_1}{m}\|\nabla_H u\|_\infty \to 0,\\
			\left| u \Delta_H \chi_m\right|\le \frac{C_2}{m}\|u\|_{\infty} \to 0,\\
			\|u\langle H,\nabla_H \chi_m \rangle_F\|_\infty \le 2C_1 \sup_{Q \in \St(n,r), \ \|H\|_F \ge m} |u(Q,H)| \to 0.
		\end{gather*}
		Hence
		\(
		\|u_m- u\|_\infty
		+
		\|\widetilde L_r(u_m)- \widetilde L_r u\|_\infty
		\to 0.
		\)
		Therefore \(C^\infty(\St(n,r))\otimes C_c^\infty(\Sym(r))\) is a core for \(\widetilde L_r\) and hence \(C_c^\infty(\St(n,r)\times \Sym(r)) \supset C^\infty(\St(n,r))\otimes C_c^\infty(\Sym(r))\) is as well.

		We now pass to the quotient. Let \(C_0(\St(n,r)\times \Sym(r))^{O(r)}\) denote the closed subspace of \(C_0(\St(n,r)\times \Sym(r))\) consisting of \(O(r)\)-invariant
		functions. For $f \in C_0(M_r)$, \(f \circ q_r\) is invariant and, since $q_r$ is proper, vanishes at infinity. On the other hand, since $q_r(K)$ is compact in $M_r$ for every compact $K \subset \St(n,r)\times \Sym(r)$, \(g \in C_0(\St(n,r)\times \Sym(r))^{O(r)}\)  descends to a function in $C_0(M_r)$. This, together with $\|f \circ q_r\|_{\infty}=\|f\|_\infty$, implies that the pullback
		\[
		q_r^\ast:C_0(M_r)\to C_0(\St(n,r)\times \Sym(r))^{O(r)},
		\quad
		f \mapsto f\circ q_r,
		\]
		is an isometric isomorphism.
		
		We have \(q_r^\ast P_t^{M_r} f=\widetilde P_t^r q_r^\ast f,\) \(q_r^\ast L_r^M  f=\widetilde L_r q_r^\ast f,\) and \(D(L_r^M)=\{f \in C_0(M_r)\,|\, q_r^\ast f \in D(\widetilde L_r)\}.\) 
		Hence it suffices to show that
		\(
		C_c^\infty(\St(n,r)\times \Sym(r))^{O(r)}
		\)
		is a core for $\widetilde L_r$ restricted to \(C_0(\St(n,r)\times \Sym(r))^{O(r)}\). 
		
		Let
		\(f\in D(\widetilde L_r)\cap C_0(\St(n,r) \times \Sym(r))^{O(r)}\). Since
		\(C_c^\infty(\St(n,r) \times \Sym(r))\) is a core for \(\widetilde L_r\), there exist
		\(f_n\in C_c^\infty(\St(n,r) \times \Sym(r))\) such that
		\(
		\|f_n -f\|_\infty \to 0\text{ and }
		\|\widetilde L_r f_n - \widetilde L_r f\|_\infty \to 0.
		\)
		Let \(\mu\) be normalized Haar measure on \(O(r)\), and define
		\(
		\overline f_n
		:=
		\int_{O(r)} f_n\circ R_G\,\mu(dG) .
		\)
		Then \(\overline f_n\in C_c^\infty(\St(n,r) \times \Sym(r))^{O(r)}\). Compact support is
		preserved because \(O(r)\) is compact. Moreover,
		\[
		|\overline f_n(Q,H)-f(Q,H)| \le \int_{O(r)} |f_n(R_G(Q,H))-f(R_G(Q,H))|\, \mu(dG) 
		\le
		\|f_n-f\|_\infty.
		\]
		Now, equivariance of the generator gives
		\(
		\widetilde L_r\overline f_n
		=
		\int_{O(r)}(\widetilde L_r f_n)\circ R_G\,dG.
		\)
		It follows that
		\(
		\|\widetilde L_r\overline f_n-\widetilde L_r f\|_\infty
		\le
		\|\widetilde L_r f_n-\widetilde L_r f\|_\infty .
		\)
		Thus \(C_c^\infty(\St(n,r) \times \Sym(r))^{O(r)}\) is a core for the invariant generator.

		Finally, we transport the statement to the rank-\(r\) stratum through the diffeomorphism
		\(
		\phi_r:M_r\to S_r.
		\)
		The pullback
		\(
		\phi^\ast:C_0(S_r) \to C_0(M_r):
		\ 
		f \mapsto f\circ\phi_r,
		\)
		is an isometric isomorphism.
		It identifies the semigroup on \(S_r\) with the semigroup
		on \(M_r\), and therefore identifies their generators by
		\(
		D(L_r)=\{f \in C_0(S_r)\,|\, \phi^\ast f \in D(L_r^M)\},
		\text{ and }
		L_r^M \phi^\ast f=\phi^\ast L_r f .
		\)
		Moreover, since $\phi$ is a diffeomorphism, \(\phi^\ast C_c^\infty(S_r)=C_c^\infty(M_r)\). Therefore
		\(
		C_c^\infty(S_r)
		\)
		is a core for \(L_r\) and the proof is concluded.
		
	\end{proofof}

	\begin{proofof}{Proposition~\ref{prop:cov-lyapunov-drift}}		
		Let \(1\le r\le n\), and let \(\Sigma\in S_r\). We write
		\(
		\Sigma=Q\exp(H)Q^\top,
		\text{ with }
		Q\in \St(n,r)
		\text{ and }
		H\in\Sym(r).
		\)
		Then, denoting \(F(H):=\Tr(\exp(H))\) we write
		\(
		V(\Sigma)=1+F(H).
		\)		
		We have
		\(
		L_rV(\Sigma)=L_HF(H)=(\sigma_r^2/2)\Delta_{H}F(H)
		-
		\theta_r\langle H,\nabla F(H)\rangle_F .
		\)		
		We estimate the two terms. First, since
		\(
		\nabla F(H)=\exp(H),
		\)
		if \(h_1,\ldots,h_r\) are the eigenvalues of \(H\), then
		\(
		\langle H,\nabla F(H)\rangle_F
		=
		\Tr(H\exp(H))
		=
		\sum_{i=1}^r h_i \exp(h_i).
		\)
		
		Now call $(S_{ij})_{1 \le i \le j \le r}$ the standard basis of $\Sym(r)$, that is, if \((E_{ij})_{1\le i,j\le r}\) is the standard basis of $\R^{r \times r}$, set
		\[S_{ij}=\begin{cases}
			E_{ii}, & i=j,\\
			\frac{1}{\sqrt{2}}(E_{ij}+E_{ji}), & i \ne j.
		\end{cases}\]
		Then
		\[
			\Delta_H F(H)= \sum_{1 \le i \le j \le r} D^2 F(H)\left(S_{ij},S_{ij}\right)=\sum_{i=1}^r \exp(h_i)+ \sum_{1 \le i < j \le r} \frac{\exp(h_i)-\exp(h_j)}{h_i-h_j},
		\]
		with the convention $(\exp(h_i)-\exp(h_j))/(h_i-h_j)=\exp(h_i)$ if $h_i=h_j$.
		Now an application of the mean value theorem gives
		\(
			0 \le \Delta_H F(H) \le \sum_{i=1}^r \exp(h_i)+ \sum_{1 \le i < j \le r}\exp(h_i)+\exp(h_j) \le r \sum_{i=1}^r \exp(h_i).
		\)	
		Consequently,
		\[
			L_rV(\Sigma)
			\le
			\sum_{i=1}^r
			\left(
			-\theta_r h_i+\frac{\sigma_r^2}{2}r
			\right)\exp(h_i),
		\]
		and
		\begin{equation}
			\label{eq:L-abs-bound}
			\frac{\sigma_r^2}{2}r\sum_{i=1}^r e^{h_i}
			+
			\theta_r\sum_{i=1}^r |h_i|e^{h_i} \le C_r\sum_{i=1}^r(1+|h_i|)e^{h_i}.
		\end{equation}

		For fixed \(\lambda_r>0\), the scalar function
		\(
		h\mapsto
		(-\theta_r h+(\sigma_r^2/2)r+\lambda_r)e^h
		\)
		is bounded above on \(\mathbb R\). Therefore there is a constant
		\(c_r<\infty\) such that
		\(
		(
		-\theta_rh_i+(\sigma_r^2/2)r
		)\exp(h_i)
		\le
		-\lambda_r \exp(h_i)+c_r.
		\)
		Summing over \(i=1,\ldots,r\), we get
		\begin{equation}
			\label{eq:diff-estimate}
		L_rV(\Sigma)
		\le
		-\lambda_r\Tr(\exp(H))+rc_r=-\lambda_rV(\Sigma)+b_r, \quad b_r:=(\lambda_r+rc_r).
		\end{equation}

		For \(r=0\), the stratum \(S_0\) consists only of the zero matrix. Hence
		\(
		V(0)=1 \text{ and }
		L_0V(0)=0,
		\)
		and we can select any \(\lambda_0>0\) and set \(b_0:=\lambda_0\).
		Since there are only finitely many ranks, define
		\(
		\lambda:=\min_{0\le r\le n}\lambda_r>0
		\text{ and }
		b_D:=\max_{0\le r\le n}b_r<\infty.
		\)
		Then, for $\Sigma \in S_r, \ 0 \le r \le n$, 
		\(
		L_rV(\Sigma)
		\le
		-\lambda V(\Sigma)+b_D.
		\)
		
		We now focus on the jump contribution.	Since \(R_r\) is supported on \(A_r\), every
		\(\Xi\in\supp R_r\) has all positive eigenvalues bounded above by
		\(\overline \ell\). Therefore
		\(
		V(\Xi)=1+\Tr(\Xi)
		\le
		1+r\overline \ell
		\le
		1+n\overline \ell.
		\)
		Hence, for every \(\Sigma\in\Cov(n)\),
		\begin{equation}
			\label{eq:jump-estimate}
		JV(\Sigma)
		\le
		\overline\Lambda(1+n\overline \ell)
		=:b_J.
		\end{equation}
		
		Combining \eqref{eq:diff-estimate} with \eqref{eq:jump-estimate} we get, for $\Sigma \in S_r$,
		\(
		G(\Sigma)
		:=
		L_rV(\Sigma)+JV(\Sigma)\le
		-\lambda V(\Sigma)+b
		\)
		where $b:= b_D+ b_J$.
		
		Next, we observe that, if $\Sigma \in S_r \cap O_m$, then \(\max_{1 \le i \le r} h_i \le \log(m-1)\). It follows, thanks to \eqref{eq:L-abs-bound}, that 
		\begin{equation}
			\label{eq:L-bounded}
			\sup_{x \in O_m \cap S_r}|L_r V(x)| <\infty.
		\end{equation}
		Let
		\(
		\overline J:=\max_{0\le r\le n}\sup_{\Xi\in A_r}V(\Xi)<\infty.
		\)
		Then we have, for $\Sigma \in O_m$
		\begin{equation}
			\label{eq:J-bounded}
		\int_{\Cov(n)} |V(\Xi)-V(\Sigma)|\,\nu(\Sigma,d\Xi) \le \int_{\Cov(n)} \bigl(V(\Xi)+V(\Sigma)\bigr)\,\nu(\Sigma,d\Xi) \le \overline\Lambda(\overline J+m).
		\end{equation}

		On the event \(\{T_m\le t\}\), the exit at \(T_m\) either occurs continuously, in
		which case \(V(\Sigma_{T_m})=m\), or it occurs by a rank-changing jump, in which
		case the post-jump point belongs to one of the compact landing sets \(A_r\), and
		therefore \(V(\Sigma_{T_m})\le \overline J\). Hence, for every fixed \(t<\infty\),
		\begin{equation}
			\label{eq:V-bounded}
		V(\Sigma_{t\wedge T_m})
		\le
		m\vee \overline J
		\end{equation}
		Applying Itô's formula to \(V\) on each inter-jump interval and adding the
		compensated jump contribution gives that
		\[
		V(\Sigma_{t\wedge T_m})
		-
		V(\Sigma_0)
		-
		\int_0^{t\wedge T_m}G(\Sigma_s)\,ds
		\]
		is a local martingale. By \eqref{eq:L-bounded},\eqref{eq:J-bounded} and
		\eqref{eq:V-bounded}, this local martingale is bounded in \(L^1\) on every
		bounded time interval. Therefore it is a true martingale. Equivalently,
		\[
		V(\Sigma_t^{(m)})
		-
		V(\Sigma_0)
		-
		\int_0^t
		\mathbf 1_{O_m}(\Sigma_s^{(m)})\,G(\Sigma_s^{(m)})\,ds
		\]
		is a martingale. This proves \(V\in D(A_m^\ast)\), with
		\(
		A_m^\ast V(\Sigma)
		=
		G(\Sigma)
		\text{ on } O_m.		\)
		
	Now choose \(k<\infty\) such that
	\(
	k\ge 2b/{\lambda}.
	\)
	If \(\Sigma\notin C_k\), then \(V(\Sigma)>k\), and therefore
	\(
	b\le (\lambda V(\Sigma))/2.
	\)
	Hence
	\(
	A^\ast_mV(\Sigma)
	\le
	-(\lambda V(\Sigma))/2.
	\)
	On the other hand, if \(\Sigma\in C_k\), then
	\(
	A^\ast_m V(\Sigma)
	\le
	-\lambda V(\Sigma)+b
	\le
	-(\lambda V(\Sigma))/2+b.
	\)
	Thus, 
	we get the global drift estimate
	\[
	A^\ast_m V(\Sigma)
	\le
	-\frac{\lambda}{2} V(\Sigma)+b\mathbf 1_{C_k}(\Sigma).
	\]

	\end{proofof}

	\section*{Acknowledgements}
	
	Supported by the University of Bologna through the programme
	``Incentives for collaboration with Universities in North America''.

	\printcredits
	
	\bibliographystyle{cas-model2-names}

\begin{thebibliography}{36}
		\expandafter\ifx\csname natexlab\endcsname\relax\def\natexlab#1{#1}\fi
		\providecommand{\url}[1]{\texttt{#1}}
		\providecommand{\href}[2]{#2}
		\providecommand{\path}[1]{#1}
		\providecommand{\DOIprefix}{doi:}
		\providecommand{\ArXivprefix}{arXiv:}
		\providecommand{\URLprefix}{URL: }
		\providecommand{\Pubmedprefix}{pmid:}
		\providecommand{\doi}[1]{\href{http://dx.doi.org/#1}{\path{#1}}}
		\providecommand{\Pubmed}[1]{\href{pmid:#1}{\path{#1}}}
		\providecommand{\bibinfo}[2]{#2}
		\ifx\xfnm\relax \def\xfnm[#1]{\unskip,\space#1}\fi
		\bibitem[{Ahdida and Alfonsi(2013)}]{AhdidaAlfonsi2013}
		\bibinfo{author}{Ahdida, A.}, \bibinfo{author}{Alfonsi, A.},
		\bibinfo{year}{2013}.
		\newblock \bibinfo{title}{A mean-reverting {SDE} on correlation matrices}.
		\newblock \bibinfo{journal}{Stochastic Process. Appl.} \bibinfo{volume}{123},
		\bibinfo{pages}{1472--1520}.
		\newblock \DOIprefix\doi{10.1016/j.spa.2012.12.008}.
		\bibitem[{Barlow et~al.(1989)Barlow, Pitman and Yor}]{barlow2006walsh}
		\bibinfo{author}{Barlow, M.T.}, \bibinfo{author}{Pitman, J.},
		\bibinfo{author}{Yor, M.}, \bibinfo{year}{1989}.
		\newblock \bibinfo{title}{On {Walsh}'s {Brownian} motions}, in:
		\bibinfo{booktitle}{S{\'e}minaire de Probabilit{\'e}s XXIII}.
		\bibinfo{publisher}{Springer}, \bibinfo{address}{Berlin}. volume
		\bibinfo{volume}{1372} of \textit{\bibinfo{series}{Lecture Notes in
				Mathematics}}, pp. \bibinfo{pages}{275--293}.
		\newblock \DOIprefix\doi{10.1007/BFb0083979}.
		\bibitem[{Bendokat et~al.(2024)Bendokat, Zimmermann and
			Absil}]{bendokat2024grassmann}
		\bibinfo{author}{Bendokat, T.}, \bibinfo{author}{Zimmermann, R.},
		\bibinfo{author}{Absil, P.A.}, \bibinfo{year}{2024}.
		\newblock \bibinfo{title}{A {Grassmann} manifold handbook: basic geometry and
			computational aspects}.
		\newblock \bibinfo{journal}{Adv. Comput. Math.} \bibinfo{volume}{50},
		\bibinfo{pages}{6}.
		\newblock \DOIprefix\doi{10.1007/s10444-023-10090-8}.
		\bibitem[{Bonnabel and Sepulchre(2010)}]{bonnabel2010riemannian}
		\bibinfo{author}{Bonnabel, S.}, \bibinfo{author}{Sepulchre, R.},
		\bibinfo{year}{2010}.
		\newblock \bibinfo{title}{Riemannian metric and geometric mean for positive
			semidefinite matrices of fixed rank}.
		\newblock \bibinfo{journal}{SIAM J. Matrix Anal. Appl.} \bibinfo{volume}{31},
		\bibinfo{pages}{1055--1070}.
		\newblock \DOIprefix\doi{10.1137/080731347}.
		\bibitem[{Brze{\'z}niak and Elworthy(2000)}]{brzezniak2000stochastic}
		\bibinfo{author}{Brze{\'z}niak, Z.}, \bibinfo{author}{Elworthy, K.D.},
		\bibinfo{year}{2000}.
		\newblock \bibinfo{title}{Stochastic differential equations on {Banach}
			manifolds}.
		\newblock \bibinfo{journal}{Methods Funct. Anal. Topol.} \bibinfo{volume}{6},
		\bibinfo{pages}{43--84}.
		\bibitem[{Bujorianu and Lygeros(2006)}]{bujorianu2006toward}
		\bibinfo{author}{Bujorianu, M.L.}, \bibinfo{author}{Lygeros, J.},
		\bibinfo{year}{2006}.
		\newblock \bibinfo{title}{Toward a general theory of stochastic hybrid
			systems}, in: \bibinfo{booktitle}{Stochastic Hybrid Systems: Theory and
			Safety Critical Applications}. \bibinfo{publisher}{Springer},
		\bibinfo{address}{Berlin}. volume \bibinfo{volume}{337} of
		\textit{\bibinfo{series}{Lecture Notes in Control and Information Sciences}},
		pp. \bibinfo{pages}{3--30}.
		\newblock \DOIprefix\doi{10.1007/11587392_1}.
		\bibitem[{Chen and Lou(2019)}]{chen2019brownian}
		\bibinfo{author}{Chen, Z.Q.}, \bibinfo{author}{Lou, S.}, \bibinfo{year}{2019}.
		\newblock \bibinfo{title}{{Brownian} motion on some spaces with varying
			dimension}.
		\newblock \bibinfo{journal}{Ann. Probab.} \bibinfo{volume}{47},
		\bibinfo{pages}{213--269}.
		\newblock \DOIprefix\doi{10.1214/18-AOP1260}.
		\bibitem[{{\c{C}}inlar(2011)}]{ccinlar2011probability}
		\bibinfo{author}{{\c{C}}inlar, E.}, \bibinfo{year}{2011}.
		\newblock \bibinfo{title}{Probability and Stochastics}. volume
		\bibinfo{volume}{261} of \textit{\bibinfo{series}{Graduate Texts in
				Mathematics}}.
		\newblock \bibinfo{publisher}{Springer}, \bibinfo{address}{New York}.
		\newblock \DOIprefix\doi{10.1007/978-0-387-87859-1}.
		\bibitem[{Daley and Vere-Jones(2008)}]{daley2008introductionII}
		\bibinfo{author}{Daley, D.J.}, \bibinfo{author}{Vere-Jones, D.},
		\bibinfo{year}{2008}.
		\newblock \bibinfo{title}{An Introduction to the Theory of Point Processes.
			Vol. II: General Theory and Structure}.
		\newblock Probability and Its Applications. \bibinfo{edition}{2} ed.,
		\bibinfo{publisher}{Springer}, \bibinfo{address}{New York}.
		\newblock \DOIprefix\doi{10.1007/978-0-387-49835-5}.
		\bibitem[{Davis(1984)}]{davis1984piecewise}
		\bibinfo{author}{Davis, M.H.A.}, \bibinfo{year}{1984}.
		\newblock \bibinfo{title}{Piecewise-deterministic {Markov} processes: a general
			class of non-diffusion stochastic models}.
		\newblock \bibinfo{journal}{J. Roy. Statist. Soc. Ser. B} \bibinfo{volume}{46},
		\bibinfo{pages}{353--376}.
		\newblock \DOIprefix\doi{10.1111/j.2517-6161.1984.tb01308.x}.
		\bibitem[{Engel and Nagel(2000)}]{engel2000one}
		\bibinfo{author}{Engel, K.J.}, \bibinfo{author}{Nagel, R.},
		\bibinfo{year}{2000}.
		\newblock \bibinfo{title}{One-Parameter Semigroups for Linear Evolution
			Equations}. volume \bibinfo{volume}{194} of \textit{\bibinfo{series}{Graduate
				Texts in Mathematics}}.
		\newblock \bibinfo{publisher}{Springer}, \bibinfo{address}{New York}.
		\newblock \DOIprefix\doi{10.1007/b97696}.
		\bibitem[{Ethier and Kurtz(1986)}]{ethier2009markov}
		\bibinfo{author}{Ethier, S.N.}, \bibinfo{author}{Kurtz, T.G.},
		\bibinfo{year}{1986}.
		\newblock \bibinfo{title}{Markov Processes: Characterization and Convergence}.
		\newblock Wiley Series in Probability and Mathematical Statistics,
		\bibinfo{publisher}{Wiley}, \bibinfo{address}{New York}.
		\newblock \DOIprefix\doi{10.1002/9780470316658}.
		\bibitem[{Freidlin and Sheu(2000)}]{freidlin2000diffusion}
		\bibinfo{author}{Freidlin, M.}, \bibinfo{author}{Sheu, S.J.},
		\bibinfo{year}{2000}.
		\newblock \bibinfo{title}{Diffusion processes on graphs: stochastic
			differential equations, large deviation principle}.
		\newblock \bibinfo{journal}{Probab. Theory Relat. Fields}
		\bibinfo{volume}{116}, \bibinfo{pages}{181--220}.
		\bibitem[{Glynn and Sigman(1992)}]{glynn1992uniform}
		\bibinfo{author}{Glynn, P.W.}, \bibinfo{author}{Sigman, K.},
		\bibinfo{year}{1992}.
		\newblock \bibinfo{title}{Uniform {Cesaro} limit theorems for synchronous
			processes with applications to queues}.
		\newblock \bibinfo{journal}{Stochastic Process. Appl.} \bibinfo{volume}{40},
		\bibinfo{pages}{29--43}.
		\newblock \DOIprefix\doi{10.1016/0304-4149(92)90135-D}.
		\bibitem[{Hajri and Raimond(2016)}]{HajriRaimond2016}
		\bibinfo{author}{Hajri, H.}, \bibinfo{author}{Raimond, O.},
		\bibinfo{year}{2016}.
		\newblock \bibinfo{title}{Stochastic flows and an interface {SDE} on metric
			graphs}.
		\newblock \bibinfo{journal}{Stochastic Process. Appl.} \bibinfo{volume}{126},
		\bibinfo{pages}{33--65}.
		\newblock \DOIprefix\doi{10.1016/j.spa.2015.07.014}.
		\bibitem[{Hsu(2002)}]{hsu2002stochastic}
		\bibinfo{author}{Hsu, E.P.}, \bibinfo{year}{2002}.
		\newblock \bibinfo{title}{Stochastic Analysis on Manifolds}.
		volume~\bibinfo{volume}{38} of \textit{\bibinfo{series}{Graduate Studies in
				Mathematics}}.
		\newblock \bibinfo{publisher}{American Mathematical Society},
		\bibinfo{address}{Providence, RI}.
		\newblock \DOIprefix\doi{10.1090/gsm/038}.
		\bibitem[{Hu et~al.(2000)Hu, Lygeros and Sastry}]{hu2000towards}
		\bibinfo{author}{Hu, J.}, \bibinfo{author}{Lygeros, J.},
		\bibinfo{author}{Sastry, S.}, \bibinfo{year}{2000}.
		\newblock \bibinfo{title}{Towards a theory of stochastic hybrid systems}, in:
		\bibinfo{booktitle}{Hybrid Systems: Computation and Control},
		\bibinfo{publisher}{Springer}, \bibinfo{address}{Berlin}. pp.
		\bibinfo{pages}{160--173}.
		\newblock \DOIprefix\doi{10.1007/3-540-46430-1_16}.
		\bibitem[{Ikeda and Watanabe(1989)}]{ikeda2014stochastic}
		\bibinfo{author}{Ikeda, N.}, \bibinfo{author}{Watanabe, S.},
		\bibinfo{year}{1989}.
		\newblock \bibinfo{title}{Stochastic Differential Equations and Diffusion
			Processes}. volume~\bibinfo{volume}{24} of
		\textit{\bibinfo{series}{North-Holland Mathematical Library}}.
		\newblock \bibinfo{edition}{2} ed., \bibinfo{publisher}{North-Holland},
		\bibinfo{address}{Amsterdam}.
		\bibitem[{Karatzas and Shreve(1991)}]{karatzas2014brownian}
		\bibinfo{author}{Karatzas, I.}, \bibinfo{author}{Shreve, S.E.},
		\bibinfo{year}{1991}.
		\newblock \bibinfo{title}{Brownian Motion and Stochastic Calculus}. volume
		\bibinfo{volume}{113} of \textit{\bibinfo{series}{Graduate Texts in
				Mathematics}}.
		\newblock \bibinfo{edition}{2} ed., \bibinfo{publisher}{Springer},
		\bibinfo{address}{New York}.
		\newblock \DOIprefix\doi{10.1007/978-1-4612-0949-2}.
		\bibitem[{Khasminskii et~al.(2007)Khasminskii, Zhu and
			Yin}]{KhasminskiiZhuYin2007}
		\bibinfo{author}{Khasminskii, R.Z.}, \bibinfo{author}{Zhu, C.},
		\bibinfo{author}{Yin, G.}, \bibinfo{year}{2007}.
		\newblock \bibinfo{title}{Stability of regime-switching diffusions}.
		\newblock \bibinfo{journal}{Stochastic Process. Appl.} \bibinfo{volume}{117},
		\bibinfo{pages}{1037--1051}.
		\newblock \DOIprefix\doi{10.1016/j.spa.2006.12.001}.
		\bibitem[{Kunita(1990)}]{kunita1990stochastic}
		\bibinfo{author}{Kunita, H.}, \bibinfo{year}{1990}.
		\newblock \bibinfo{title}{Stochastic Flows and Stochastic Differential
			Equations}. volume~\bibinfo{volume}{24} of \textit{\bibinfo{series}{Cambridge
				Studies in Advanced Mathematics}}.
		\newblock \bibinfo{publisher}{Cambridge University Press},
		\bibinfo{address}{Cambridge}.
		\bibitem[{Last and Penrose(2017)}]{LastPenrose2017}
		\bibinfo{author}{Last, G.}, \bibinfo{author}{Penrose, M.},
		\bibinfo{year}{2017}.
		\newblock \bibinfo{title}{Lectures on the Poisson Process}.
		volume~\bibinfo{volume}{7} of \textit{\bibinfo{series}{Institute of
				Mathematical Statistics Textbooks}}.
		\newblock \bibinfo{publisher}{Cambridge University Press},
		\bibinfo{address}{Cambridge}.
		\newblock \DOIprefix\doi{10.1017/9781316104477}.
		\bibitem[{Lewis and Shedler(1979)}]{lewis1979simulation}
		\bibinfo{author}{Lewis, P.A.W.}, \bibinfo{author}{Shedler, G.S.},
		\bibinfo{year}{1979}.
		\newblock \bibinfo{title}{Simulation of nonhomogeneous {Poisson} processes by
			thinning}.
		\newblock \bibinfo{journal}{Naval Res. Logist. Quart.} \bibinfo{volume}{26},
		\bibinfo{pages}{403--413}.
		\newblock \DOIprefix\doi{10.1002/nav.3800260304}.
		\bibitem[{Lunardi et~al.(2020)Lunardi, Metafune and
			Pallara}]{lunardi2020ornstein}
		\bibinfo{author}{Lunardi, A.}, \bibinfo{author}{Metafune, G.},
		\bibinfo{author}{Pallara, D.}, \bibinfo{year}{2020}.
		\newblock \bibinfo{title}{The {Ornstein--Uhlenbeck} semigroup in finite
			dimension}.
		\newblock \bibinfo{journal}{Philos. Trans. Roy. Soc. A} \bibinfo{volume}{378},
		\bibinfo{pages}{20200217}.
		\newblock \DOIprefix\doi{10.1098/rsta.2020.0217}.
		\bibitem[{Mao and Yuan(2006)}]{mao2006stochastic}
		\bibinfo{author}{Mao, X.}, \bibinfo{author}{Yuan, C.}, \bibinfo{year}{2006}.
		\newblock \bibinfo{title}{Stochastic Differential Equations with {Markovian}
			Switching}.
		\newblock \bibinfo{publisher}{Imperial College Press},
		\bibinfo{address}{London}.
		\newblock \DOIprefix\doi{10.1142/p473}.
		\bibitem[{Mayerhofer et~al.(2011)Mayerhofer, Pfaffel and
			Stelzer}]{MayerhoferPfaffelStelzer2011}
		\bibinfo{author}{Mayerhofer, E.}, \bibinfo{author}{Pfaffel, O.},
		\bibinfo{author}{Stelzer, R.}, \bibinfo{year}{2011}.
		\newblock \bibinfo{title}{On strong solutions for positive definite jump
			diffusions}.
		\newblock \bibinfo{journal}{Stochastic Process. Appl.} \bibinfo{volume}{121},
		\bibinfo{pages}{2072--2086}.
		\newblock \DOIprefix\doi{10.1016/j.spa.2011.05.006}.
		\bibitem[{Meyer(1975)}]{meyer1975renaissance}
		\bibinfo{author}{Meyer, P.A.}, \bibinfo{year}{1975}.
		\newblock \bibinfo{title}{Renaissance, recollements, m{\'e}langes,
			ralentissement de processus de {Markov}}.
		\newblock \bibinfo{journal}{Ann. Inst. Fourier (Grenoble)}
		\bibinfo{volume}{25}, \bibinfo{pages}{465--497}.
		\newblock \DOIprefix\doi{10.5802/aif.593}.
		\bibitem[{Meyn and Tweedie(1993a)}]{meyn1993stability}
		\bibinfo{author}{Meyn, S.P.}, \bibinfo{author}{Tweedie, R.L.},
		\bibinfo{year}{1993}a.
		\newblock \bibinfo{title}{Stability of {Markovian} processes ii:
			continuous-time processes and sampled chains}.
		\newblock \bibinfo{journal}{Adv. Appl. Probab.} \bibinfo{volume}{25},
		\bibinfo{pages}{487--517}.
		\newblock \DOIprefix\doi{10.2307/1427521}.
		\bibitem[{Meyn and Tweedie(1993b)}]{meyn1993stability3}
		\bibinfo{author}{Meyn, S.P.}, \bibinfo{author}{Tweedie, R.L.},
		\bibinfo{year}{1993}b.
		\newblock \bibinfo{title}{Stability of {Markovian} processes iii:
			{Foster--Lyapunov} criteria for continuous-time processes}.
		\newblock \bibinfo{journal}{Adv. Appl. Probab.} \bibinfo{volume}{25},
		\bibinfo{pages}{518--548}.
		\newblock \DOIprefix\doi{10.2307/1427522}.
		\bibitem[{Meyn and Tweedie(2009)}]{meyn2012markov}
		\bibinfo{author}{Meyn, S.P.}, \bibinfo{author}{Tweedie, R.L.},
		\bibinfo{year}{2009}.
		\newblock \bibinfo{title}{Markov Chains and Stochastic Stability}.
		\newblock Cambridge Mathematical Library. \bibinfo{edition}{2} ed.,
		\bibinfo{publisher}{Cambridge University Press},
		\bibinfo{address}{Cambridge}.
		\newblock \DOIprefix\doi{10.1017/CBO9780511626630}.
		\bibitem[{Sturm(1998)}]{sturm1998diffusion}
		\bibinfo{author}{Sturm, K.T.}, \bibinfo{year}{1998}.
		\newblock \bibinfo{title}{Diffusion processes and heat kernels on metric
			spaces}.
		\newblock \bibinfo{journal}{Ann. Probab.} \bibinfo{volume}{26},
		\bibinfo{pages}{1--55}.
		\newblock \DOIprefix\doi{10.1214/aop/1022855410}.
		\bibitem[{Walsh(1978)}]{walsh1978diffusion}
		\bibinfo{author}{Walsh, J.B.}, \bibinfo{year}{1978}.
		\newblock \bibinfo{title}{A diffusion with a discontinuous local time}.
		\newblock \bibinfo{journal}{Ast{\'e}risque} \bibinfo{volume}{52--53},
		\bibinfo{pages}{37--45}.
		\bibitem[{Xi(2009)}]{Xi2009}
		\bibinfo{author}{Xi, F.}, \bibinfo{year}{2009}.
		\newblock \bibinfo{title}{Asymptotic properties of jump-diffusion processes
			with state-dependent switching}.
		\newblock \bibinfo{journal}{Stochastic Process. Appl.} \bibinfo{volume}{119},
		\bibinfo{pages}{2198--2221}.
		\newblock \DOIprefix\doi{10.1016/j.spa.2008.11.001}.
		\bibitem[{Xi and Zhu(2018)}]{XiZhu2018}
		\bibinfo{author}{Xi, F.}, \bibinfo{author}{Zhu, C.}, \bibinfo{year}{2018}.
		\newblock \bibinfo{title}{On the martingale problem and {Feller} and strong
			{Feller} properties for weakly coupled {L{\'e}vy}-type operators}.
		\newblock \bibinfo{journal}{Stochastic Process. Appl.} \bibinfo{volume}{128},
		\bibinfo{pages}{4277--4308}.
		\newblock \DOIprefix\doi{10.1016/j.spa.2018.02.005}.
		\bibitem[{Yin and Zhu(2010)}]{yin2009hybrid}
		\bibinfo{author}{Yin, G.G.}, \bibinfo{author}{Zhu, C.}, \bibinfo{year}{2010}.
		\newblock \bibinfo{title}{Hybrid Switching Diffusions: Properties and
			Applications}. volume~\bibinfo{volume}{63} of
		\textit{\bibinfo{series}{Stochastic Modelling and Applied Probability}}.
		\newblock \bibinfo{publisher}{Springer}, \bibinfo{address}{New York}.
		\newblock \DOIprefix\doi{10.1007/978-1-4419-1105-6}.
		\bibitem[{Zuyev(2006)}]{zuyev2006strong}
		\bibinfo{author}{Zuyev, S.}, \bibinfo{year}{2006}.
		\newblock \bibinfo{title}{Strong {Markov} property of {Poisson} processes and
			{Slivnyak} formula}, in: \bibinfo{booktitle}{Case Studies in Spatial Point
			Process Modeling}. \bibinfo{publisher}{Springer}, \bibinfo{address}{New
			York}. volume \bibinfo{volume}{185} of \textit{\bibinfo{series}{Lecture Notes
				in Statistics}}, pp. \bibinfo{pages}{77--84}.
		\newblock \DOIprefix\doi{10.1007/0-387-31144-0_3}.
		
	\end{thebibliography}
	

	
	
\end{document}